\documentclass[a4paper,11pt]{amsart}
\usepackage[utf8]{inputenc}

\usepackage{enumitem}

\usepackage{a4wide}

\usepackage{orcidlink}
\usepackage{amsmath,amssymb}
\usepackage{bm}
\usepackage{multirow}
\usepackage{algpseudocode}
\usepackage{algorithm2e}

\usepackage{hyperref}

\usepackage[misc]{ifsym} 

\usepackage{color}
\usepackage{graphicx}
\usepackage{subcaption}

\usepackage{amsfonts}
\usepackage{amsthm}
\usepackage{mathtools}
\mathtoolsset{showonlyrefs=true}

\usepackage[mathscr]{eucal}

\usepackage{pgfplots}
\pgfplotsset{compat=1.18}

\usepackage{soul}

\usepackage{bbm}
\ifpdf
\DeclareGraphicsExtensions{.eps,.pdf,.png,.jpg}
\else
\DeclareGraphicsExtensions{.eps}
\fi

\newtheorem{remark}{Remark}

\allowdisplaybreaks

\usepackage{mathtools}

\providecommand{\E}[1]{{\ensuremath{\mathbb{E}}\mspace{-2mu}\left[#1\right]}}    
\providecommand{\var}[1]{{\ensuremath{\mathrm{Var}}\mspace{-2mu}\left[#1\right]}}

\DeclarePairedDelimiter\abs{\lvert}{\rvert}%
\DeclarePairedDelimiter\norm{\lVert}{\rVert}%
\DeclarePairedDelimiter\ceil{\lceil}{\rceil}%
\DeclarePairedDelimiter\floor{\lfloor}{\rfloor}%

\newtheorem{Definition}{Definition}

\newtheorem{Theorem}{Theorem}

\newtheorem{Lemma}{Lemma}

\newcommand{\RN}[1]{%
	\textup{\uppercase\expandafter{\romannumeral#1}}%
}

\renewcommand{\le}{\leqslant}
\renewcommand{\ge}{\geqslant}
\renewcommand{\leq}{\leqslant}
\renewcommand{\geq}{\geqslant}
\renewcommand{\emptyset}{\varnothing}

\newcommand{\natu}{\mathbb{N}}

\newcommand{\bsk}{\boldsymbol{k}}

\newcommand{\bsx}{\boldsymbol{x}}

\newcommand{\bszero}{\boldsymbol{0}}
\newcommand{\bsone}{\boldsymbol{1}}

\newcommand{\EE}{\mathbb{E}}
\newcommand{\PP}{\mathbb{P}}

\newcommand{\rank}{\mathrm{rank}}

\newcommand{\bwal}[1]{ {}_{b}\mathrm{wal}_{#1} }

\newcommand{\med}{\mathrm{med}}

\newcommand{\supp}{\mathrm{supp}}

\title{Quasi-Monte Carlo with A Hankel Random Digital Net}

\begin{document}

\author[T.~Goda]{Takashi Goda}
\address[T.~Goda]{Graduate School of Engineering, The University of Tokyo, 7-3-1 Hongo, Bunkyo-ku, Tokyo 113-8656,
Japan} \email[]{goda@frcer.t.u-tokyo.ac.jp}
\thanks{The work of T.G.\ was supported by JSPS KAKENHI Grant Numbers JP23K03210 and JP26K00620.}

\author[Y.~Liu]{Yang Liu}
\address[Y.~Liu]{CEMSE, King Abdullah University of Science and Technology, Thuwal, Saudi Arabia} \email[]{yang.liu.3@kaust.edu.sa}

\author[R.~Tempone]{Raúl Tempone}
\address[R.~Tempone]{CEMSE, King Abdullah University of Science and Technology, Thuwal, Saudi Arabia} \email[]{raul.tempone@kaust.edu.sa}
\thanks{The work of R.T.\ was supported by the Alexander von Humboldt Foundation and the King Abdullah
University of Science and Technology (KAUST) office of sponsored research (OSR) under Award
No. OSR-2019-CRG8-4033}

\begin{abstract}
	This paper proposes a new randomized design of digital nets in which the generating matrices are chosen to be random Hankel matrices. Compared with previous randomized design of digital nets, this approach simplifies the construction process and reduces the number of random variables required, while still achieving desirable convergence rates when combined with appropriate estimators. We analyze the properties of the proposed design, derive bounds for Walsh coefficients, and provide error analysis for both the median-of-means estimator and a newly proposed greedy selection estimator, i.e. the selection of the best design from a batch in terms of a worst-case error bound. Numerical experiments validate our theoretical findings and demonstrate the practical performance of the proposed methods.

\end{abstract}



\maketitle

\sloppy

\section{Introduction}

Digital nets are an important class of low-discrepancy sequences that are widely used in quasi-Monte Carlo methods for numerical integration. The concept of digital nets originated from the work of Sobol' \cite{Sob67}, is later generalized by Niederreiter \cite{Nie87,Nie92}, leading to the development of $(t,m,s)$-nets and $(t,s)$-sequences. Optimized construction of the digital nets in base 2 in terms of two-dimensional projections is studied by Joe and Kuo~\cite{JK08}, which is also widely implemented in various software packages up to dimension 21,201. 


Recently, there has been growing interest in randomized designs of digital nets, where the generating matrices are chosen randomly. This approach can simplify the construction process and provide good average-case performance---as well as probabilistic worst-case error guarantees---without the need for intricate deterministic designs. Goda and L'Ecuyer~\cite{GL22} and Goda, Suzuki, and Matsumoto~\cite{GSM24} proposed a randomized design of (higher-order) polynomial lattice rules, a special class of digital nets, by randomly selecting numerator and denominator polynomials from specific admissible sets. Pan and Owen~\cite{PO23,PO24} introduced a \emph{uniform random design} (URD) of base-$2$ generating matrices, where each entry in the matrix is independently chosen from $\{0, 1\}$. Both approaches require no construction parameters and achieve desirable convergence rates when combined with the median-of-means estimator. 

In this work, we propose a randomized design of digital nets in which the entries of generating matrices are chosen randomly to form a Hankel matrix structure. We refer to this construction as the \emph{Hankel random design} (HRD) throughout this paper. This approach serves as a bridge between the randomized polynomial lattice rules \cite{GL22,GSM24} and the URD \cite{PO23,PO24}. While the former's generating matrices are inherently Hankel but constrained by the algebraic properties of the chosen polynomials, and the latter's are entirely unstructured, our design exploits the Hankel structure to maintain a degree of algebraic regularity while allowing for greater flexibility through randomness. Consequently, compared to the previously mentioned random designs, our approach is simpler to implement, involves fewer random variables, and improves certain error bounds.

The rest of the paper is organized as follows. In Section~\ref{sec:back}, we introduce the necessary background and notation for digital nets and their randomizations. In Section~\ref{sec:random-design-digital-net}, we present our proposed HRD, analyze its properties, compare with the uniform random design and revisit the linear matrix scrambling. In Section~\ref{sec:walsh_bounds} we derive the bounds for Walsh coefficients, which are used to the error analysis in Section~\ref{sec:convergence_rates} for both the median-of-means estimator and newly proposed best design from a batch estimator. Finally, we present numerical experiments in Section~\ref{sec:numerical_results} to validate our theoretical findings and demonstrate the practical performance of the proposed methods.

\section{Background and notation}\label{sec:back}

\subsection{Notation}
We introduce some notations that we use throughout the paper. 
\begin{itemize}
    \item Let $\mathbb{N}$ be the set of positive integers, $\mathbb{N}_0 = \mathbb{N} \cup \{0\}$, $\mathbb{N}_0^s = \mathbb{N}_0 \times \cdots \times \mathbb{N}_0$ ($s$ times) and $\mathbb{N}_*^s = \mathbb{N}_0^s \setminus \{\bm{0}\}$. 
    \item $\mathbb{F}_b$ denotes the finite field of order $b$, where $b$ is a fixed prime number. $\mathbb{F}_b^{\times}$ denotes the set of non-zero elements in $\mathbb{F}_b$. We identify $\mathbb{F}_b$ with the set $\{0,1,\ldots,b-1\}$ equipped with addition and multiplication modulo $b$. 
    \item $(t, m, s)$ are the parameters of the digital net. 
    \item $E\in \natu \cup \{\infty\}$ is the precision of the digital net. In the case of base $b = 2$ with standard IEEE 754 double precision, we typically set $E = 53$, consisting of the 52-bit fractional part and the one implicit leading bit.
    \item $C_1, ..., C_s \in \mathbb{F}_b^{E \times m}$ are the generating matrices. We use the notation $C_j(1{:}q_j, :)$ to denote the first $q_j$ rows of the matrix $C_j$, for $j = 1, ..., s$. $C_j(1{:}q_j, :) = \varnothing$, when $q_j = 0$. For $\bm{q} = (q_1, ..., q_s) \in \{0, 1, \dotsc, E\}^s$, denote $C_{\bm{q}}$ to define the vertically stacked matrix:
    \begin{equation}
        C_{\bm{q}} = \begin{bmatrix}
            C_1(1{:}q_1,:)\\
            \vdots\\
            C_s(1{:}q_s,:)
        \end{bmatrix}
        \in \mathbb{F}_b^{\abs{\bm{q}}_1 \times m}.
    \end{equation}
    \item For $k \in \mathbb{N}$ with $k = \sum_{i=1}^{\nu} k_{[i]} b^{i - 1}$ and $k_{[\nu]}\ne 0$, we denote $\vec{k} = (k_{[1]}, k_{[2]}, \dotsc, k_{[\nu]})^{\top} \in \mathbb{F}_b^{\nu}$. For a precision $E\in \natu \cup \{\infty\}$, we also write 
    \[ \vec{k}_E = (k_{[1]}, k_{[2]}, \dotsc, k_{[E]})^{\top} \in \mathbb{F}_b^{E},\]
    where $k_{[i]}=0$ if $i>\nu$.
    Let $\kappa(k) = \{i \in \mathbb{N} : k_{[i]} \neq 0\}$ be the set of positions of non-zero digits in the base-$b$ expansion of $k$. We use $k_{(\alpha)}$ to denote the $\alpha$-th largest element of $\kappa$, setting $k_{(\alpha)} = 0$ if $\alpha > \abs{\kappa}$. 
    \item Niederreiter--Rosenbloom--Tsfasman (NRT) and Dick weights \cite{Nie86,RT97,Dic08}: For $k \in \mathbb{N}$, we define $\mu_1(k) = k_{(1)}$ with $\mu_1(0) = 0$. For $\alpha \in \mathbb{N}$, the higher-order weight is given by $\mu_{\alpha}(k)= \sum_{i=1}^{\alpha} k_{(i)}$. For a vector $\bm{k} \in \mathbb{N}^s_{0}$, we define $\mu_{\alpha}(\bm{k})=\sum_{j=1}^{s}\mu_{\alpha}(k_j)$.
    \item We vectorize the above notation: $\bm{\mu}_{\alpha} (\bm{k}) = (\mu_{\alpha} (k_1), \dotsc,  \mu_{\alpha} (k_s))$. 
\end{itemize}


\subsection{Digital net construction and randomizations}
In this section, we briefly discuss the digital net construction and randomization. For a more comprehensive treatment of digital nets and sequences, we refer the reader to the monographs \cite{DP10,Nie92}.


Let $b$ be a fixed prime number, and let $m, s$ be positive integers. A digital net is a point set of $N=b^m$ points in $[0,1)^s$ constructed using linear algebra over the finite field $\mathbb{F}_b$. The construction is defined by $s$ generating matrices $C_1, \dots, C_s \in \mathbb{F}_b^{E \times m}$, where $E$ is the precision and $m$ determines the quadrature size, typically $E \geq m$. 

Now we describe the procedure to generate the $j$-th coordinate of the $n$-th point $\bm{x}_n$, denoted by $x_{n,j}$, where $j = 1, \dots, s$ and $n = 0, \dots, b^m - 1$. We first represent $n$ in its base-$b$ expansion:
\begin{equation}
  n = n_0 + n_1 b + \cdots + n_{m-1} b^{m-1},
\end{equation}
which corresponds to a vector $\vec{n}_m = (n_0, \dots, n_{m-1})^{\top} \in \mathbb{F}_b^m$. Then, we compute a vector of digits $\vec{y}_{n,j} = (y_{n,j,1}, \dots, y_{n,j,E})^{\top} \in \mathbb{F}_b^E$ via the matrix-vector multiplication over $\mathbb{F}_b$:
\begin{equation}
    \label{eq:digital-net-point}
  \vec{y}_{n,j} = C_j \vec{n}_m.
\end{equation}
Here, if $E=\infty$, we assume that every column of the generating matrices has only finitely many non-zero entries to ensure the well-definedness of this matrix-vector multiplication. Finally, these digits are used to form the coordinate $x_{n,j}$ by mapping them to a real number in $[0,1)$:
\begin{equation}
  x_{n,j} = \sum_{k=1}^{E} y_{n,j,k} b^{-k}.
\end{equation}
Applying the above procedure to all $j = 1, \dotsc, s$ and $n = 0, \dotsc, b^{m} -1$ yields the resulting digital net $\{\bm{x}_0, \dots, \bm{x}_{b^m-1}\}$.

As mentioned in the introduction, recent works explore randomized construction of digital nets through randomly chosen generating matrices $C_1, \dotsc, C_s$.  Goda and L'Ecuyer~\cite{GL22} and Goda, Suzuki, and Matsumoto~\cite{GSM24} proposed random designs of (higher-order) polynomial lattice rules, a special class of digital nets proposed by Niederreiter~\cite{Nie87} and then generalized by Dick and Pillichshammer~\cite{DP07}. In particular, the numerator and denominator polynomials are randomly selected from a specific admissible set. However, the requirement that the denominator polynomial be irreducible makes this approach less straightforward. Pan and Owen~\cite{PO23,PO24} consider a URD of base-$2$ generating matrices: each entry in the matrices is randomly selected in $\{0, 1\}$. This design is more straightforward to implement. Notably, when combined with the median-of-means estimator, these random designs achieve nearly optimal convergence rates for smooth functions. A significant advantage of this approach is its universality with respect to smoothness: it attains these rates without prior knowledge of the function's smoothness, which is typically required in existing constructions of optimal deterministic/randomized digital nets \cite{Dic08,Dic11,DP07,God15,GD15,GSY18}.

The random design we introduce in this work combines the algebraic structure of polynomial lattice rules with the implementation simplicity of uniform random designs. We describe the details of this construction and its theoretical properties in the next section.

\section{Randomized design of digital nets}
\label{sec:random-design-digital-net}

In this section, we first introduce our HRD of digital nets. Then, we analyze its properties and compare it with the URD. Lastly, we revisit the properties of linear matrix scrambling \cite{Mat98} and relate them to Owen-scrambling \cite{Owe95,Owen97}.

\subsection{Hankel random design of digital nets}

We describe the proposed randomized design of digital nets. We present the procedure in one dimension as the construction is identical and independent across dimensions.

We generate the entries of the generating matrix $C \in \mathbb{F}_b^{E \times m}$ using a random vector $\bm{u} = (u_{1}, u_{2}, \dotsc, u_{E + m - 1}) \in \mathbb{F}_b^{E + m - 1}$, where each component $u_i$ is chosen independently and uniformly from $\mathbb{F}_b$. The entries of $C$ are given by
\begin{equation}
  c_{j,r} = u_{r+j-1},
\end{equation}
for $1\le j\le E$ and $1\le r\le m$. Notice that $C$ has a Hankel structure that also arises in polynomial lattice rules. Then, the point construction, as in Equation~\eqref{eq:digital-net-point}, can be written as
\begin{equation}
  \label{eq:hankel_u}
  \begin{bmatrix}
  u_{1} & u_{2} & \cdots & u_{m} \\
  u_{2} & u_{3} & \cdots & u_{m+1} \\
  \vdots  & \vdots  & \ddots & \vdots  \\
  u_{E} & u_{E+1} & \cdots & u_{E+m-1}
  \end{bmatrix}
  \begin{bmatrix}
  n_{0}  \\
  n_{1}  \\
  \vdots \\
  n_{m-1} 
  \end{bmatrix}
  =
  \begin{bmatrix}
  y_{1}  \\
  y_{2}  \\
  \vdots \\
  y_{E} 
  \end{bmatrix}
\end{equation}

Similar to linear matrix scrambling \cite{Mat98}, we further introduce a random digital shift $\vec{d}_j \in \mathbb{F}_b^E$, drawn independently and uniformly from $\mathbb{F}_b^E$ for each $j = 1, \dotsc, s$. This shift is applied to all points:
\begin{equation}
    \label{eq:digital-net-point-with-shift}
    \vec{y}_{n,j} = C_j \vec{n}_m \oplus \vec{d}_j,
\end{equation}
where $\oplus$ denotes addition over $\mathbb{F}_b$. This additional randomization ensures that each individual point $\bm{x}_n$ is uniformly distributed in $[0, 1)^s$ for $E = \infty$ (or $b^{-E}$-equidistributed for finite $E$) marginally, while preserving the structural properties of the digital net collectively. A fast implementation of the construction procedure is presented in Appendix~\ref{sec:fast_implementation_digital_nets}.

With the digital net $\{\bsx_n\}$ generated randomly by~\eqref{eq:digital-net-point-with-shift}, the sample average
\begin{equation}
  Q_{N}(f)
  \;=\;
  \frac{1}{b^m}\sum_{n=0}^{b^m-1} f(\bm{x}_n).
  \label{eq:poly-lat-rule}
\end{equation}
is an unbiased estimator of the integral
\[
I(f) = \int_{[0, 1)^s} f(\bm{x})\,\mathrm{d}\bm{x},
\]
for any $f\in L^1([0,1)^s)$, where we write $N=b^m$.

Before analyzing the integration error, we introduce the Walsh functions in base $b$.
\begin{Definition}
    For an integer base $b\ge 2$, let $\omega_b = e^{2\pi i / b}$ be the primitive $b$-th root of unity. For a non-negative integer $k\in \natu_0$ with $b$-adic expansion $k=\sum_{i=1}^{\infty}k_{[i]}b^{i-1}$ (where $k_{[i]}\in \mathbb{F}_b$ and the sum is finite) and a real number $x\in [0,1)$ with $b$-adic expansion $x=\sum_{i=1}^{\infty}x_{[i]}b^{-i}$ (unique under the condition that infinite trails of $b-1$ are not allowed), the $k$-th Walsh function is defined as
    \[ \bwal{k}(x) \coloneqq \omega_b^{k_{[1]}x_{[1]}+k_{[2]}x_{[2]}+\cdots}.\]
    For dimension $s\ge 1$, $\bm{k}=(k_1,\ldots,k_s)\in \natu_0^s$, and $\bm{x}=(x_1,\ldots,x_s)\in [0,1)^s$, the multidimensional Walsh function is defined as the product
    \[ \bwal{\bm{k}}(\bm{x})\coloneqq \prod_{j=1}^{s}\bwal{k_j}(x_j). \]
\end{Definition}
It is known that the set $\{\bwal{\bm{k}}\mid \bm{k}\in \natu_0^s\}$ forms a complete orthonormal basis of $L_2([0,1)^s)$; see, e.g., \cite[Appendix~A]{DP10}.
Accordingly, for a function $f\in L_2([0,1)^s)$, the $\bm{k}$-th Walsh coefficient is given by
\[ \hat{f}(\bm{k}) = \int_{[0,1)^s}f(\bm{x})\, \overline{\bwal{\bm{k}}(\bm{x})}\,\mathrm{d}\bm{x}, \]
and we have the following equality in the $L_2$ sense:
\[ f(\bm{x}) \sim \sum_{\bm{k}\in \natu_0^s}\hat{f}(\bm{k})\, \bwal{\bm{k}}(\bm{x}).\]

Now, assuming that $f$ can be represented by its pointwise absolutely convergent Walsh series, for any realization of the randomized digital net with a random digital shift, the integration error can be expressed as
\begin{equation}
\label{eq:error-digital-net}
\begin{split}
I(f)-Q_N(f) 
&= I(f) - \frac{1}{b^m}\sum_{n=0}^{b^m-1} f(\bm{x}_n) \\
&= I(f) - \frac{1}{b^m}\sum_{n=0}^{b^m-1} \sum_{\bm{k} \in \mathbb{N}_0^s} \hat{f}(\bm{k}) \omega_b^{\sum_{j=1}^s \vec{k}_j^{\top} \cdot (C_j \vec{n} + \vec{d}_j)} \\
&= -\sum_{\bm{k} \in \mathbb{N}_0^s} \hat{f}(\bm{k}) \ \omega_b^{\sum_{j=1}^s \vec{k}_j^{\top} \vec{d}_j} \ \frac{1}{b^m}\sum_{n=0}^{b^m-1}   \omega_b^{\sum_{j=1}^s \vec{k}_j^{\top} C_j \vec{n}}\\
&= -\sum_{\bm{k} \in \mathcal{D}^\perp \setminus \{\bm{0} \} }\hat{f}
( \bm{k} ) \ \omega_b^{\sum_{j=1}^s \vec{k}_j^{\top} \vec{d}_j},
\end{split}
\end{equation}
where the dual net $\mathcal{D}^\perp$ is defined as
\[
\mathcal{D}^\perp = \mathcal{D}^\perp(C_1,\ldots,C_s) \;=\;
\left\{\bm{k}\in\mathbb{N}_0^s:\,\sum_{j=1}^s {C}_{j}^{\top} \vec{k}_{j,E}=\bszero \in \mathbb{F}_b^m\right\}.
\]
The last equality in \eqref{eq:error-digital-net} follows from the character property of digital nets \cite{DP10}, where the inner sum vanishes unless $\bsk\in \mathcal{D}^{\perp}$. Note that the integration error in \eqref{eq:error-digital-net} vanishes in expectation over the random digital shift $d_j$, consistent with the unbiasedness of $Q_N(f)$. 

The next lemma gives the probability that a fixed frequency $\bm{k} \neq \bszero$ belongs to the dual net $\mathcal{D}^\perp$ with respect to the random choice of the Hankel matrices $C_1,\ldots,C_s$. 

\begin{Lemma}\label{lem:prob_each_k}
Consider the HRD of digital nets with the infinite-precision generating matrices $C_j \in \mathbb{F}_b^{\infty\times m}$ for $j=1,\ldots,s$. For any fixed $\bm{k} \in \mathbb{N}_*^s$, we have
\begin{equation}
  \Pr (\bm{k} \in \mathcal{D}^\perp) = \Pr\nolimits_{C_1,\ldots,C_s} (\bm{k} \in \mathcal{D}^\perp) = b^{-m}.
\end{equation}
\end{Lemma}

\begin{proof}
    We first consider the one-dimensional case. Let $k\in \natu$ and choose $n\in \natu$ such that $k < b^n$. Recall the notation $\vec{k}_{\infty} = (k_{[1]}, k_{[2]}, \dotsc, k_{[n]}, 0,\dotsc,0)^{\top} \in \mathbb{F}_b^{\infty}$. We aim to show that $C^{\top} \vec{k}_{\infty}$ is uniformly distributed in $\mathbb{F}_b^{m}$. Due to the Hankel structure of $C$, the product $C^{\top} \vec{k}_{\infty}$ can be equivalently written as $\bm{M}\boldsymbol{u}$, where $\boldsymbol{u}=(u_1,\ldots,u_{m+n-1})^{\top}$ and $\bm{M}\in \mathbb{F}_b^{m\times (m+n-1)}$ is given by 
  \begin{equation}
  \bm{M} = \begin{bmatrix}
    k_{[1]} & k_{[2]} & k_{[3]} & \cdots & k_{[n]} & 0 & 0 &\cdots & 0 \\
    0 & k_{[1]} & k_{[2]} & \cdots & k_{[n-1]} & k_{[n]} & 0 & \cdots & 0 \\
    \vdots & \vdots & \vdots & \ddots & \vdots & \vdots & \ddots & \cdots & \vdots \\
    0 & 0 & 0 & \cdots & k_{[1]} & k_{[2]} & k_{[3]} & \cdots & k_{[n]} \\
  \end{bmatrix}.
\end{equation}
In the following, we show that $\bm{M}$ has full row rank $m$.

Suppose there exists a vector $\bm{c}=(c_0,\ldots,c_{m-1})^{\top} \in \mathbb{F}_b^m$ such that $\bm{c}^{\top}\bm{M} = \bm{0}$. Let
\begin{equation}
  K(x) = \sum_{i=1}^{n} k_{[i]} x^{i}\quad \text{and}\quad C(x) = \sum_{j=0}^{m-1} c_j x^{j}
\end{equation}
be polynomials in $\mathbb{F}_b[x]$. The condition $\bm{c}^{\top}\bm{M} = \bm{0}$ is equivalent to stating that all coefficients of the product polynomial $P(x)=C(x)K(x)$ are zero. Since $\mathbb{F}_b[x]$ is an integral domain and $K(x)\ne 0$ (as $k\ne 0$), the product $C(x)K(x)=0$ implies $C(x)=0$. Thus, $c_0=c_1=\dots=c_{m-1}=0$, proving that $\bm{M}$ has full row rank $m$.

Since $\bm{M}$ is surjective and $\boldsymbol{u}$ is chosen uniformly at random from $\mathbb{F}_b^{m+n-1}$, the linear transformation $\bm{M}\boldsymbol{u}$ is uniformly distributed in $\mathbb{F}_b^m$.
Thus, for each $k \in \mathbb{N}$, $\Pr(C^{\top} \vec{k}_{\infty} = \bszero) = \Pr(\bm{M}\boldsymbol{u} = \bszero) = b^{-m}$. 

In the multidimensional case, since the generating matrices $C_j$ are chosen independently for $j=1,\ldots,s$, the vectors $C_j^{\top} \vec{{k}}_{j,\infty}$ are independent and each is uniformly distributed in $\mathbb{F}_b^m$. Since the sum of independent uniform random variables over a finite group is also uniform, we conclude that $\Pr(\bm{k} \in \mathcal{D}^{\perp}) = b^{-m}$. This completes the proof.
\end{proof}

Now, we investigate the $L^2$ error (mean square error) of the estimator. Using the error expression in \eqref{eq:error-digital-net} and the independence of $C_j$ and $d_j$, we have: 
\begin{align*}
\mathbb{E}[(I(f)-Q_N(f))^2] &= \mathbb{E} \left[ \sum_{\bm{k}, \bm{\ell} \in \mathbb{N}_*^s } \left( \mathbbm{1}_{\bm{k} \in \mathcal{D}^\perp} \hat{f}
( \bm{k} ) \ \omega_b^{\sum_{j=1}^s \vec{k}_j^{\top} \vec{d}_j} \right) \left({ \mathbbm{1}_{\bm{\ell} \in \mathcal{D}^\perp} \hat{f}
( \bm{\ell} ) \ \omega_b^{\sum_{j=1}^s \vec{\ell}_j^{\top} \vec{d}_j} }\right)^{*} \right]\\
&= \sum_{\bm{k},\bm{\ell} \in \mathbb{N}_*^s } \E{\mathbbm{1}_{\bm{k} \in \mathcal{D}^\perp}\mathbbm{1}_{\bm{\ell} \in \mathcal{D}^\perp}} \hat{f}
( \bm{k} ) \left({\hat{f}
( \bm{\ell} )} \right)^* \, \mathbb{E} \left[ \omega_b^{\sum_{j=1}^s (\vec{k}_j - \vec{\ell}_j)^{\top} \vec{d}_j} \right]\\
& = \sum_{\bm{k} \in \mathbb{N}_*^s } \Pr\nolimits_{C_1,\ldots,C_s} (\bm{k} \in \mathcal{D}^\perp)\,  \abs{\hat{f}
( \bm{k} )}^2\\
&= b^{-m} \sum_{\bm{k} \in \mathbb{N}_*^s } \abs{\hat{f}
( \bm{k} )}^2 = {N}^{-1} \var{f},
\end{align*}
where we used, in the third equality, the fact that the expectation with respect to the digital shift $d_j$ vanishes unless $k_j=\ell_j$ for all $j$, in which case it equals $1$.

This shows that the basic HRD of digital nets only achieves the same variance as the standard Monte Carlo. In the rest of the paper, we introduce additional techniques, such as the median-of-means and greedy optimization, to improve the error rate. Before that, we study probability bounds for a certain event in the next subsection. 

\subsection{The probability of a bad event}
Following Pan and Owen~\cite{PO23,PO24}, certain randomized digital nets exhibit events that dominate the integration error. Particularly, the events correspond to vectors $\bm{k}$ belonging to the dual net whose Walsh coefficients $\abs{\hat{f}(\bm{k})}$ have larger magnitude. Here, we compare the probability of such an event between HRD and URD. 

We define the event 
\[\mathcal{A}_{K} = \bigcup_{\bm{k} \in K} \{ \bm{k} \in \mathcal{D}^\perp \},\] 
where $K \subset \mathbb{N}_*^s$ is a given finite set. A crude upper bound follows from the union bound: $\Pr(\mathcal{A}_{K}) \leq \sum_{\bm{k} \in K} \Pr(\bm{k} \in \mathcal{D}^\perp) = \abs{K} b^{-m}$, where we used Lemma~\ref{lem:prob_each_k} in the last equality. This bound applies to both HRD and URD. To obtain a sharper bound, we study the joint inclusion probabilities of pairs of distinct vectors in the dual net, i.e., $\Pr(\bm{k}_1, \bm{k}_2 \in \mathcal{D}^\perp)$.


First, we consider the event $\{ \bm{k}_1, \bm{k}_2 \in \mathcal{D}^\perp_X \}$ for distinct $\bm{k}_1, \bm{k}_2 \in \mathbb{N}_*^s$, where $X = U$ or $H$ denotes URD or HRD, respectively. For both types of random design, this event is equivalent to 
\begin{equation}
  \sum_{j=1}^s C_j^{\top} K_j = \bm{0}_{m \times 2},
\end{equation}
where $C_j\in \mathbb{F}_b^{\infty\times m}$ is either the URD or HRD generating matrix, $K_j = [\vec{k}_{1, j}, \vec{k}_{2, j}]\in \mathbb{F}_b^{\infty\times 2}$ be the matrix formed by stacking the vectors. 

We let $Y_j = C_j^{\top} K_j$, and define the linear transformation $T_j : \mathbb{F}_b^{\infty \times m} \to \mathbb{F}_b^{m \times 2}$ by $T_j(C_j) = Y_j$. Then each $Y_j$ is uniform on the linear subspace $\mathrm{Im}(T_j)$. Since the matrices $C_j$ are independent, each $Y_j$ is independent and thus
\begin{equation}
  \sum_{j=1}^s C_j^{\top} K_j = \sum_{j=1}^s Y_j
\end{equation}
is uniformly distributed on the linear subspace $V_{\Sigma} := V_1 + V_2 + \cdots + V_s$, where $V_j = \mathrm{Im}(T_j)$. In general, the dimension of $V_{\Sigma}$ satisfies:
\begin{equation}
   \max_j \dim V_j \leq \dim(V_{\Sigma}) \leq \min\left( 2m, \sum_{j=1}^s \dim V_j \right)
\end{equation}

The Chung--Erd\H{o}s bound~\cite{CE52} and Hunter's bound~\cite{Hun76} provide lower and upper bounds for $\Pr(\mathcal{A}_K)$, respectively, using the joint inclusion probabilities: 
\begin{equation}\label{eq:prob-bound_inequality}
\begin{split}
  \frac{\left( \sum_{i \in K} \Pr(\bm{k}_i \in \mathcal{D}^\perp) \right)^2}{\sum_{i \in K} \Pr(\bm{k}_i \in \mathcal{D}^\perp) + \sum_{i_1, i_2 \in K, i_1 \neq i_2} \Pr(\bm{k}_{i_1}, \bm{k}_{i_2} \in \mathcal{D}^\perp)} \leq \Pr(\mathcal{A}_K),\\
    \Pr(\mathcal{A}_K) \leq \sum_{i \in K} \Pr(\bm{k}_i \in \mathcal{D}^\perp) - \max_{T} \sum_{(i_1, i_2) \in T} \Pr(\bm{k}_{i_1}, \bm{k}_{i_2} \in \mathcal{D}^\perp),
\end{split}
\end{equation}
where the maximum is taken over all spanning trees $T$ on the complete graph with the vertex set $K$, and $(i_1, i_2)$ denotes an edge in $T$. 

In the following, we first discuss the one-dimensional case ($s = 1$). For URD, each component in $C_j \in \mathbb{F}_b^{\infty \times m}$ is i.i.d. Thus, 
\begin{equation}
\dim V_j=
\begin{cases}
m, & \text{if }\operatorname{rank}(K_j)=1,\\[4pt]
2m, & \text{if }\operatorname{rank}(K_j)=2,
\end{cases}
\end{equation}
For HRD, the dimension of \(V_j\) depends on the rank of \(K_j\). 
If \(\operatorname{rank}(K_j)=1\), then \(\dim V_j = m\). 
If \(\operatorname{rank}(K_j)=2\), then
\[
m+1 \le \dim V_j \le 2m.
\]
We can show that {$\Pr(\mathcal{A}_K^H) \leq \Pr(\mathcal{A}_K^U)$ as $m \to \infty$.} For simplicity, let us consider $b = 2$. In this case, there is no linearly dependent pair $({k}_{i_1}, {k}_{i_2})$ for distinct $i_1, i_2 \in K$. Then, for URD, $\dim (V_{\Sigma}) = \dim (V_1)= 2m$.
From the inequalities \eqref{eq:prob-bound_inequality}, we have a lower bound
\begin{equation}
  \Pr(\mathcal{A}_K^U) \geq \frac{\abs{K} 2^{-m}}{1 + (\abs{K} - 1)2^{-m-1}},
\end{equation}
and an upper bound
\begin{equation}
  \Pr(\mathcal{A}_K^U) \leq \abs{K} 2^{-m} - (\abs{K} - 1) 2^{-2m}.
\end{equation}
Thus $\Pr(\mathcal{A}_K^U) = {\Theta}\left(\abs{K} 2^{-m}\right)$.

For HRD, suppose that there exist $i_1 = 2^{m_1}, i_2 = 2^{m_2} \in K$, such that $d = \abs*{m_1 - m_2} < m$. Then $\dim (V_{1}) = m+d$ for the pair $(k_{i_1}, k_{i_2})$. Applying the inequalities~\eqref{eq:prob-bound_inequality}, we have an upper bound
\begin{equation}
  \Pr(\mathcal{A}_K^H) \leq \abs{K} 2^{-m} - 2^{-(m+d)} - (\abs{K} - 2) 2^{-2m} \leq \abs{K-2^{-d}} 2^{-m}.
\end{equation}
Then we have the following inequality
\begin{equation}
    \Pr(\mathcal{A}_K^H) \leq \abs{K-2^{-d}} 2^{-m} \leq \abs{K} 2^{-m} \leq \frac{\abs{K} 2^{-m}}{1 + (\abs{K} - 1)2^{-m-1}} \leq \Pr(\mathcal{A}_K^U),
\end{equation}
for $\abs{K} \geq 1$ and $m \geq 0$. This demonstrates the advantage of HRD over URD in one-dimensional case. 

The distinction between HRD and URD is less pronounced in higher dimensions, since in both cases, we have
\begin{equation}
  m \leq \dim (V_{\Sigma}) \leq 2m,
\end{equation}
and the joint inclusion probabilities satisfy
\begin{equation}
  b^{-2m} \leq \Pr(\bm{k}_i, \bm{k}_j \in \mathcal{D}^\perp_X) \leq b^{-m}.
\end{equation}
Nonetheless, our numerical results in Section~\ref{sec:numerical_results} indicate that HRD yields smaller worst-case error bounds than URD in practice.

\subsection{The $t$ parameter}
In this section, we show that the $t$ parameter, a quality measure for the equidistribution properties of digital nets, is small with high probability for the HRD.
Here, the quality parameter $t$ satisfies that $C_{\bm{q}}$ has full row rank whenever $\abs{\bm{q}}_1 \leq m - t$. The definition of the stacked matrix $C_{\bm{q}}$ can be referred in Section~\ref{sec:back}.

In the following, we examine the probability that $C_{\bm{q}}$ has full row rank. First, the union bound leads to
\begin{equation}
  \Pr\left(\rank(C_{\bm{q}}) < \abs{\bm{q}}_1 \right) = \Pr\left( \bigcup_{\bm{y} \in \mathbb{F}_b^{\abs{\bm{q}}_1} \setminus \{\bszero\}  } \left(\bm{y}^{\top} C_{\bm{q}} = \bszero \right)\right) \leq \sum_{\bm{y} \in \mathbb{F}_b^{\abs{\bm{q}}_1} \setminus \{\bszero\}  } \Pr\left( \bm{y}^{\top} C_{\bm{q}} = \bszero \right).
\end{equation}
We now want to show that, for each nonzero column vector $\bm{y}\in \mathbb{F}_b^{\abs{\bm{q}}_1}$,
\[
\Pr\left( \bm{y}^{\top} C_{\bm{q}} = \bszero \right) = b^{-m}.
\] 
Denote $\bm{y}^{\top} = (\bm{y}_1^{\top}, ..., \bm{y}_s^{\top})$, where $\bm{y}_j \in \mathbb{F}_b^{q_j}$. We also denote $\bm{v}_j^{\top} = \bm{y}_j^{\top} C_j(1{:}q_j,:)$ for $j = 1, \dotsc, s$. Then we have 
\begin{equation}
  \bm{y}^{\top} C_{\bm{q}} = \sum_{j=1}^{s} \bm{v}_j^{\top}.
\end{equation}
Notice that if $\bm{y}_j \neq \bszero$, under the randomized design of matrices $C_1, ..., C_s$, each $\bm{v}_j$ is uniformly distributed in $\mathbb{F}_b^{m}$. Thus for $\bm{y}\neq \bszero$, which implies that at least one $\bm{y}_j \neq \bszero$, $j = 1, \dotsc, s$, we have $\sum_{j=1}^{s} \bm{v}_j^{\top}$ is uniformly distributed in $\mathbb{F}_b^m$. Thus,
\begin{equation}
  \Pr\left( \bm{y}^{\top} C_{\bm{q}} = \bszero \right) = b^{-m}.
\end{equation}
Consequently,
\begin{equation}
  \Pr\left(\rank(C_{\bm{q}}) < \abs{\bm{q}}_1 \right) \leq (b^{\abs{\bm{q}}_1} - 1)b^{-m}.
\end{equation}

The probability that $t > t_0$ satisfies
\begin{align*}
    \Pr(t > t_0) &\leq \sum_{\abs{\bm{q}}_1 = m - t_0} \Pr\left(\rank(C_{\bm{q}}) < m - t_0 \right)\\
    &= {{m-t_0+s-1} \choose {s-1}} (b^{\abs{\bm{q}}_1} - 1)b^{-m}\\
    &= {{m-t_0+s-1} \choose {s-1}} \left( b^{-t_0} - b^{-m} \right)
    \leq \frac{m^{s-1}}{(s-1)!} b^{-t_0},
\end{align*}
for $t_0 \geq s-1$. Thus, for instance, if $m\ge b$, we have
\begin{equation}
    \Pr(t > (s-1) \log_b m) \leq \frac{1}{(s-1)!},
\end{equation}
and 
\begin{equation}
    \Pr(t > s \log_b m) \leq \frac{1}{m (s-1)!}.
\end{equation}
These results ensure that the $t$ parameter of the HRD is bounded above by $s\log_bm$ with high probability. This is analogous to the known results for polynomial lattice point sets (see, for instance, \cite[Corollary~10.15]{DP10}). 

\subsection{Linear matrix scrambling revisited}
In this section, we revisit the properties of linear matrix scrambling (LMS)~\cite{Mat98} for digital nets.
With LMS~\cite{Mat98}, the condition $\{\bm{k} \in \mathcal{D}^\perp\}$ is equivalent to
\begin{equation}
  \sum_{j=1}^s \vec{k}_{j,\infty}^{\top} M_j C_j = \bszero,
\end{equation}
where each $M_j \in \mathbb{F}_b^{\infty \times \infty}$ is an independent lower triangular matrix whose diagonal entries are drawn independently and uniformly from $\mathbb{F}_b^{\times}$, and whose entries below the diagonal are independent and uniformly distributed in $\mathbb{F}_b$. Here we assume that all generating matrices $C_1,\ldots,C_s$ are of size $\infty \times m$. The probability of the event $\{\bm{k} \in \mathcal{D}^{\perp}\}$ is given in the following lemma. 
\begin{Lemma}
    \label{lemma:probability_k_in_dual}
    For a non-empty subset $\mathfrak{u}\subseteq \{1,\ldots,s\}$ and a vector $\bm{\ell}_{\mathfrak{u}}\in \natu^{|\mathfrak{u}|}$, define the set 
    \[ A(\mathfrak{u}, \bm{\ell}_{\mathfrak{u}}) = \left\{ \bm{k}\in \natu_0^s \mid \text{$b^{\ell_j - 1} \leq k_j < b^{\ell_j}$ for $j \in \mathfrak{u}$, and $k_j=0$ for $j\notin \mathfrak{u}$} \right\}. \] 
    Let $\mathcal{D}_C^{\perp}$ be the dual net corresponding to fixed (deterministic) generating matrices $C_1, ..., C_s$.
    With LMS, for any $\bsk\in A(\mathfrak{u}, \bm{\ell}_{\mathfrak{u}})$, we have
    \begin{equation}
        \Pr(\bm{k} \in \mathcal{D}^{\perp}) = \frac{\abs{A(\mathfrak{u}, \bm{\ell}_{\mathfrak{u}}) \cap \mathcal{D}_C^{\perp} }}{\abs{A(\mathfrak{u}, \bm{\ell}_{\mathfrak{u}})}},
    \end{equation}
    where $\abs{A(\mathfrak{u}, \bm{\ell}_{\mathfrak{u}})}=b^{|\bm{\ell}_{\mathfrak{u}}|_1-|\mathfrak{u}|}(b-1)^{|\mathfrak{u}|}$.
\end{Lemma}
\begin{proof}
    Denote $\vec{v}_{j,\infty}^{\top} = \vec{k}_{j,\infty}^{\top} M_j$. Since $b$ is a prime number and $M_j$ is a random lower triangular matrix with non-zero diagonal entries, for a fixed $k_j$ with $b^{\ell_j-1}\le k_j<b^{\ell_j}$, its $\ell_j$-th digit (the most significant digit) is multiplied by a uniform element in $\mathbb{F}_b^{\times}$. As a result, the $\ell_j$-th component of $\vec{v}_j$ is uniformly distributed in $\mathbb{F}_b^{\times}$, while all components with indices $i<\ell_j$ are independently and uniformly distributed in $\mathbb{F}_b$ due to the entries of $M_j$ below the diagonal. All components with indices $i>\ell_j$ are zero. Thus the integer $v_j$ is uniformly distributed over the set of integers $[b^{\mu_1(k_j) - 1}, b^{\mu_1(k_j) })$.
    
    Due to the independence between $M_1,\ldots,M_s$, the vector $\bm{v} = (v_1, \dotsc, v_s)$ is uniformly distributed in the set $A(\mathfrak{u}, \bm{\ell}_{\mathfrak{u}})$. The condition $\bm{k} \in \mathcal{D}^\perp$ under LMS is then equivalent to the condition $\bm{v} \in \mathcal{D}_C^\perp$. Thus, the probability is given by the ratio of the number of elements in $\mathcal{A}(\mathfrak{u}, \bm{\ell}_{\mathfrak{u}})$ that satisfy the dual net condition to the total size of the set. The cardinality is $|\mathcal{A}(\mathfrak{u}, \bm{\ell}_{\mathfrak{u}})| = \prod_{j \in \mathfrak{u}} (b^{\ell_j} - b^{\ell_j-1}) = b^{|\bm{\ell}_{\mathfrak{u}}|_1-|\mathfrak{u}|}(b-1)^{|\mathfrak{u}|}$, which completes the proof.
\end{proof}

\begin{remark}[Variance equivalence between Owen-scrambling and LMS]
    The above results can be used to establish a variance equivalence between Owen-scrambling and LMS. Specifically, the probability term $\Pr((\bm{k}_{\mathfrak{u}},\bszero) \in \mathcal{D}^{\perp})$ arising from LMS can be related to the gain coefficient from Owen-scrambling through the identity:
    \begin{equation}
        \label{eq:probability_gain_equivalence}
        \Pr\left(\sum_{j \in \mathfrak{u}} \vec{k}_{j,\infty}^{\top} M_j C_j = \bszero \right) = b^{-m}  \Gamma_{\mathfrak{u}, \bm{\mu}_1 (\bm{k}_{\mathfrak{u}}) - \bm{1}_{\mathfrak{u}} },
    \end{equation}
    where $\Gamma_{\mathfrak{u}, \bm{\mu}_1 (\bm{k}_{\mathfrak{u}}) - \bm{1}_{\mathfrak{u}} }$ is the gain coefficient defined in~\cite{Owen97}. 
    The equivalence~\eqref{eq:probability_gain_equivalence} in base 2 follows by comparing the results in \cite[Section~4]{PO23b} and \cite[Lemma~4]{PO24}. The extension to a general prime base $b$ is obtained by comparing Lemma~\ref{lemma:probability_k_in_dual} with \cite[Lemma~3.3]{GS23}. 
    
    Recall that the Walsh variance components are defined as
    \begin{equation}
        \sigma^2_{\mathfrak{u}, \bm{\ell}} = \sum_{ b^{{\ell}_j} \leq  {k}_{j} < b^{\ell_j + 1}, \ j \in \mathfrak{u} } \abs{\hat{f}(\bm{k}_{\mathfrak{u}}, \bm{0} )}^2.
    \end{equation}
    Using this, we can show the equivalence of the variance under Owen-scrambling and LMS with a digital shift as follows:
    \begin{align*}
        \var{\hat{I}} &= \sum_{\varnothing \neq \mathfrak{u} \subseteq 1{:}s} \sum_{\bm{k}_{\mathfrak{u}} \in \mathbb{N}^{\abs{\mathfrak{u}}}} \Pr \left( (\bm{k}_{\mathfrak{u}}, \bm{0}) \in \mathcal{D}^{\perp} \right) \abs{\hat{f}(\bm{k}_{\mathfrak{u}}, \bm{0})}^2 \\
        &= \sum_{\varnothing \neq \mathfrak{u} \subseteq 1{:}s} \sum_{\bm{\ell} \in \mathbb{N}^{\abs{\mathfrak{u}}}_0 } \sum_{\bm{k} \in A(\mathfrak{u}, \bm{\ell} + \bm{1}_{\mathfrak{u}})} \Pr \left( (\bm{k}_{\mathfrak{u}}, \bm{0}) \in \mathcal{D}^{\perp} \right) \abs{\hat{f}(\bm{k}_{\mathfrak{u}}, \bm{0})}^2\\
        &= \frac{1}{N} \sum_{\varnothing \neq \mathfrak{u} \subseteq 1{:}s} \sum_{\bm{\ell} \in \mathbb{N}^{\abs{\mathfrak{u}}}_0 } \Gamma_{\mathfrak{u}, \bm{\ell}} \sigma^2_{\mathfrak{u}, \bm{\ell}}.
    \end{align*}
\end{remark}

For any non-empty subset $\mathfrak{u}\subseteq \{1,\ldots,s\}$, let $C_{\mathfrak{u}}$ denote the set of generating matrices $\{C_j\}_{j\in \mathfrak{u}}$. We define $t_{\mathfrak{u}}$ as the smallest non-negative integer such that for any vector $\bm{\ell}_{\mathfrak{u}}\in \natu_0^{|\mathfrak{u}|}$ satisfying $|\bm{\ell}_{\mathfrak{u}}|_1=m-t_{\mathfrak{u}}$, the rows of the matrices $\{C_j(1{:}\ell_j,:)\}_{j\in \mathfrak{u}}$ are linearly independent over $\mathbb{F}_b$. We then provide bounds for the probability $\Pr{(\bm{k} \in \mathcal{D}^{\perp})}$ in the following lemma, which serves as the counterpart of Lemma~\ref{lem:prob_each_k} for the proposed HRD. The advantage of the LMS can be found in the regime where $\mu_1(\bm{k}) \leq m - t_{\mathfrak{u}}$, while the HRD provides sharper (or at least equivalent) bounds for the other cases.

\begin{Lemma}\label{lem:prob_each_k_LMS}
    For a non-empty subset $\mathfrak{u}\subseteq \{1,\ldots,s\}$ and a vector $\bm{k}_{\mathfrak{u}}\in \natu^{|\mathfrak{u}|}$, let $\bm{k} = (\bm{k}_{\mathfrak{u}}, \bm{0})$. Under LMS, we have the following bounds:
    \begin{equation}
        \Pr{(\bm{k} \in \mathcal{D}^{\perp})} = \begin{cases}
            0 & \text{if $\mu_1(\bm{k}) \leq m - t_{\mathfrak{u}}$}, \\
            b^{-m} & \text{if $\max_{j\in \mathfrak{u}} \mu_1({k}_j) > m$},
        \end{cases}
    \end{equation}
    and for the intermediate regime where $\max_{j\in \mathfrak{u}} \mu_1({k}_j) \leq m$,
    \begin{equation}
        \label{eq:probability_k_in_dual_bound}
        \Pr{(\bm{k} \in \mathcal{D}^{\perp})} \leq \begin{cases}
       \min\left(\frac{b^{-m + t_{\mathfrak{u}} + \abs{\mathfrak{u}} - 1}}{(b-1)^{\abs{\mathfrak{u}} -1}}, 1\right) & \text{if $m - t_{\frak{u}} < \mu_1(\bm{k}) \leq m - t_{\mathfrak{u}} + \abs{\mathfrak{u}}$}, \\
        b^{-m + t_{\mathfrak{u}}} & \text{if $\mu_1(\bm{k}) > m - t_{\mathfrak{u}} + \abs{\mathfrak{u}}$}.
        \end{cases}
    \end{equation}
\end{Lemma}
\begin{proof}
    The first case $\mu_1(\bm{k}) \leq m - t_{\mathfrak{u}}$ is straightforward from the definition of the $t_{\mathfrak{u}}$ parameter. For instance, the reader may refer to \cite[Lemmas 1 and 5]{Hel02}. The second case $\max_{j\in \mathfrak{u}} \mu_1({k}_j) > m$ is a natural extension of \cite[Lemma~4]{PO24} from base $2$ to a prime base $b$. The remaining cases in~\eqref{eq:probability_k_in_dual_bound} extend \cite[Lemma~4]{PO24}, which we detail as follows. 

    For simplicity, let $\ell_j = \mu_1(k_j)$. We consider the decomposition:
\begin{equation}
    \label{eq:dual_net_decomposition}
    \sum_{j \in \mathfrak{u}} \vec{k}_j^{\top} M_j C_j = \sum_{j \in \mathfrak{u}} k_j(\ell_j) M_j(\ell_j, \ell_j) C_j(\ell_j, :) + \sum_{j \in \mathfrak{u}} \left( \vec{k}_j^{\top} M_j C_j - k_j(\ell_j) M_j(\ell_j, \ell_j) C_j(\ell_j, :) \right).
\end{equation}
For each $j \in \mathfrak{u}$, the summand of the second term on the right-hand side of \eqref{eq:dual_net_decomposition} is uniformly distributed on the span of the rows of $C_j(1{:}\ell_j - 1, :)$. Due to the independence of the random matrices $M_j$ across dimensions, this second sum is uniformly distributed on the linear span of $C_{\bm{\mu}_1(\bm{k}) - \bsone_{\mathfrak{u}}}$. Thus, if $\sum_{j \in \mathfrak{u}} k_j(\ell_j) M_j(\ell_j, \ell_j) C_j(\ell_j, :) \notin \operatorname{span}(C_{\bm{\mu}_1(\bm{k}) - \bsone_{\mathfrak{u}}})$, the conditional probability of the total sum being the zero vector is $0$, i.e.,
\begin{equation}
    \label{eq:cond_prob_dual_net_1}
    \Pr \left( \sum_{j=1}^s \vec{k}_j^{\top} M_j C_j = \bszero \;\middle|\; \sum_{j \in \mathfrak{u}} k_j(\ell_j) M_j(\ell_j, \ell_j) C_j(\ell_j, :) \notin \operatorname{span}(C_{\bm{\mu}_1(\bm{k}) - \bsone_{\mathfrak{u}}}) \right) = 0.
\end{equation}
Otherwise, the conditional probability is given by
\begin{equation}
    \label{eq:cond_prob_dual_net_2}
    \Pr \left( \sum_{j=1}^s \vec{k}_j^{\top} M_j C_j = \bszero \;\middle|\; \sum_{j \in \mathfrak{u}} k_j(\ell_j) M_j(\ell_j, \ell_j) C_j(\ell_j, :) \in \operatorname{span}(C_{\bm{\mu}_1(\bm{k}) - \bsone_{\mathfrak{u}}}) \right) = b^{-\rank(C_{\bm{\mu}_1(\bm{k}) - \bsone_{\mathfrak{u}}})}.
\end{equation}
From the definition of the $t_{\mathfrak{u}}$-parameter, we have the following rank bounds:
    \begin{equation}
    \label{eq:rank_bound_c_mu_minus_1}
    \begin{aligned}
    \rank\!\bigl(C_{\bm{\mu}_1(\bm{k}) - \bm{1}_{\mathfrak{u}}}\bigr)
    &= \mu_1(\bm{k}) - \abs{\mathfrak{u}},
    &&\text{if } \mu_1(\bm{k}) \le m - t_{\mathfrak{u}} + \lvert\mathfrak{u}\rvert, \\[4pt]
    \rank\!\bigl(C_{\bm{\mu}_1(\bm{k}) - \bm{1}_{\mathfrak{u}}}\bigr)
    &\ge m - t_{\mathfrak{u}},
    &&\text{otherwise.}
    \end{aligned}
    \end{equation}

We now study the probability that $\sum_{j \in \mathfrak{u}} k_j(\ell_j) M_j(\ell_j, \ell_j) C_j(\ell_j, :) \in \operatorname{span}(C_{\bm{\mu}_1(\bm{k}) - \bsone_{\mathfrak{u}}})$. Let $\underline{C}_{\bm{\mu}_1(\bm{k})}$ be the vertical concatenation of the rows $C_j(\ell_j, :)$ for $j \in \mathfrak{u}$. Defining $\bm{q} = (k_j(\ell_j) M_j(\ell_j, \ell_j))_{j \in \mathfrak{u}}^{\top} \in (\mathbb{F}_b^{\times})^{|\mathfrak{u}|}$, the condition is equivalent to $\underline{C}_{\bm{\mu}_1(\bm{k})}^{\top} \bm{q} \in \mathrm{im}(C_{\bm{\mu}_1(\bm{k}) - \bsone_{\mathfrak{u}}}^{\top})$. Let $V \subseteq \mathbb{F}_b^{|\mathfrak{u}|}$ be the subspace of vectors $\bm{q}$ satisfying this condition. By the rank-nullity theorem, we have
\begin{equation}
    \dim V = |\mathfrak{u}| + \rank(C_{\bm{\mu}_1(\bm{k}) - \bsone_{\mathfrak{u}}}) - \rank(C_{\bm{\mu}_1(\bm{k})}).
\end{equation}
Since $\rank(C_{\bm{\mu}_1(\bm{k})}) - \rank(C_{\bm{\mu}_1(\bm{k}) - \bsone_{\mathfrak{u}}}) \ge 0$, and specifically, when $m - t_{\mathfrak{u}} < \mu_1(\bm{k}) \le m - t_{\mathfrak{u}} + |\mathfrak{u}|$, we have $\rank(C_{\bm{\mu}_1(\bm{k})}) - \rank(C_{\bm{\mu}_1(\bm{k}) - \bsone_{\mathfrak{u}}}) \ge (m - t_{\mathfrak{u}}) - (\mu_1(\bm{k}) - |\mathfrak{u}|)$, it follows that
\begin{equation}
    \label{eq:bound_dimV}
    \dim V \le 
    \begin{cases}
        -m + t_{\mathfrak{u}} + \mu_1(\bm{k}) & \text{if } m - t_{\mathfrak{u}} < \mu_1(\bm{k}) \le m - t_{\mathfrak{u}} + |\mathfrak{u}|, \\
        |\mathfrak{u}| & \text{if } \mu_1(\bm{k}) > m - t_{\mathfrak{u}} + |\mathfrak{u}|.
    \end{cases}
\end{equation}
Given that $\bm{q}$ is uniformly distributed in $(\mathbb{F}_b^{\times})^{|\mathfrak{u}|}$, the probability is bounded as
\begin{equation}
    \label{eq:bound_dimV_prob}
    \Pr(\underline{C}_{\bm{\mu}_1(\bm{k})}^{\top} \bm{q} \in \mathrm{im}(C_{\bm{\mu}_1(\bm{k}) - \bsone_{\mathfrak{u}}}^{\top})) \le \frac{(b-1)^{\dim V}}{(b-1)^{|\mathfrak{u}|}},
\end{equation}
referring to \cite[Lemma 3.9]{GS23}. Combining \eqref{eq:cond_prob_dual_net_2}, \eqref{eq:rank_bound_c_mu_minus_1}, \eqref{eq:bound_dimV}, and \eqref{eq:bound_dimV_prob} completes the proof.
\end{proof}

Comparing Lemmas~\ref{lem:prob_each_k} and \ref{lem:prob_each_k_LMS} reveals that LMS is advantageous in the regime where $\mu_1(\bm{k}) \leq m - t_{\mathfrak{u}}$ (i.e., for ``small'' $\bm{k}$), whereas HRD (as well as URD) is better in the intermediate regime (i.e., for ``medium'' $\bm{k}$). Importantly, for well-known sequences such as Sobol' sequences \cite{Sob67}, the full-dimensional $t$-value scales as $O(s\log s)$. In such cases, HRD is particularly effective as it avoids the exponential growth of the probability bound with respect to $s$ that would otherwise be incurred by the $t$-value.

\section{Bounds on Walsh Coefficients}
\label{sec:walsh_bounds}

In this section, we provide bounds on the Walsh coefficients to facilitate the derivations of convergence rates in the next section. 

Let $\alpha\in \natu$ denote the smoothness parameter. For $k\in \natu$ with its $b$-adic expansion $k=\sum_{i=1}^{\nu}c_ib^{a_i-1}$ with $c_i\in \mathbb{F}_b^{\times}$ and $a_1>a_2>\cdots>a_{\nu}>0$, we write 
\[ k^{+}_{\alpha} := \sum_{i=1}^{\min(\alpha, \nu)} c_i b^{a_i - 1}\quad \text{and}\quad k^{-}_{\alpha} := \sum_{i=\alpha+1}^{\nu} c_i b^{a_i - 1}.\]
If $k=0$, we set $k^{+}_{\alpha}=k^{-}_{\alpha}=0$. Notice that $k = k^{+}_{\alpha} + k^{-}_{\alpha}$ holds, and $k^{+}_{\alpha} \geq k^{-}_{\alpha}$ for $\alpha \in \mathbb{N}$. As we introduced in Section~\ref{sec:back}, recall that we denote by $\kappa(k) = \{a_1,\ldots,a_{\nu}\}$ the set of positions of non-zero digits in the base-$b$ expansion of $k$, and by $k_{(\alpha)}$ the $\alpha$-th largest element of $\kappa$, i.e., $k_{(\alpha)}=a_{\alpha}$ if $\nu\ge \alpha$ and $k_{(\alpha)}=0$ otherwise. For a vector $\bm{k}\in \natu_0^s$, we write $\bm{k}_{(\alpha)}=(k_{1,(\alpha)},\ldots,k_{s,(\alpha)})$. 

Moreover, let $\bm{n}_{\alpha}(\bm{k})=(n_{\alpha}(k_1),\ldots,n_{\alpha}(k_s))\in \natu_0^s$, where $n_{\alpha}(k_j)=\min(\alpha,|\kappa(k_j)|)$. It is important to note that, for any $k\in \natu_0$ such that $|\kappa(k)|\le \alpha$, we have 
\[ k^{+}_{n_{\alpha}(k)}=k^{+}_{n_{\alpha}(k)+1}=\cdots=k^{+}_{\alpha}\quad \text{and}\quad k^{-}_{n_{\alpha}(k)}=k^{-}_{n_{\alpha}(k)+1}=\cdots=k^{-}_{\alpha}=0. \]
For a function $f: [0, 1)^s\to \mathbb{R}$ that is $\alpha$-times differentiable in each variable, Suzuki and Yoshiki~\cite[Theorem~2.5]{SY16} derived the following formula for the $\bm{k}$-th Walsh coefficient:
\begin{equation}
    \label{eq:walsh_coef}
    \hat{f}(\bm{k}) = (-1)^{|\bm{n}_{\alpha}(\bm{k})|_1} \int_{[0, 1)^s} f^{(\bm{n}_{\alpha}(\bm{k}))} (\bm{x}) \prod_{j=1}^s \overline{\bwal{{k_j}^{-}_{\alpha}}(x_j)} W({k_j}^{+}_{\alpha})(x_j)\, \mathrm{d}\bm{x},
\end{equation}
where $f^{(\bm{n}_{\alpha}(\bm{k}))}$ denotes the partial mixed derivative of $f$ of order $\bm{n}_{\alpha}(\bm{k})$. Furthermore, the auxiliary function $W(k)(\cdot)$ appearing in \eqref{eq:walsh_coef} is periodic with period $b^{-a_{\nu} + 1}$ when $\nu > 0$ and is defined recursively in \cite[Definition 2.1]{SY16}.
\begin{Definition}
    For $k \in \mathbb{N}_0$, $W(k)(\cdot): [0, 1) \to \mathbb{C}$ is defined recursively as:
    \begin{equation}
        \begin{aligned}
            W(0)(x) &= 1,\\
            W(k)(x) &= \int_0^x  \overline{\bwal{{c_\nu} b^{a_{\nu} - 1}}(y)} W(k^{\prime}) (y)\, \mathrm{d}y,
        \end{aligned}
    \end{equation}
    where $k^{\prime} = k - {c_\nu} b^{a_{\nu} - 1}$. 
\end{Definition}

We list a useful property of $W$ in the following lemma. The proof is deferred to Appendix~\ref{sec:proof_W_linfty_bound}. 

\begin{Lemma}[$L^{\infty}$ bound]
    \label{lemma:W_linfty_bound}
    For any $k > 0$, the following bound for the function $W(k_{\alpha}^{+})$ holds:
    \begin{equation}
        \norm{W(k_{\alpha}^{+})}_{L^{\infty}} \leq \left(\frac{1 + \pi}{2} \right) b^{n_{\alpha}(k)-\mu_{\alpha}(k)}
    \end{equation}
    for any fixed base $b \geq 2$. 
\end{Lemma}

\subsection{Bounds for functions with finite generalized Vitali variation}
Here, we focus on a class of smooth functions with finite generalized Vitali variation, which is defined as follows:

\begin{Definition}
    For a function $f: [0, 1)^s\to \mathbb{R}$ and a parameter $p\in (0,1]$, the generalized Vitali variation, denoted by $V_p^{(s)}(f)$, is defined as:
    \[ V_p^{(s)}(f)\coloneqq \sup_{\mathcal{P}} \left( \sum_{J \in \mathcal{P}} \mathrm{Vol}(J) \left| \frac{\Delta (f, J)}{\mathrm{Vol}(J)^p} \right|^2 \right)^{1/2},\]
    where $\mathcal{P}$ is the set of all finite partitions of $[0, 1)^s$ into $s$-dimensional axis-parallel boxes $J=\prod_{j=1}^{s}[a_j,b_j]$. Here $\Delta (f, J)$ denotes the $s$-th order mixed difference of $f$ over the box $J$, i.e.,
    \[ \Delta (f, J) \coloneqq \sum_{\bm{\epsilon} \in \{0, 1\}^s} (-1)^{s - |\bm{\epsilon}|_1} f(a_1 + \epsilon_1(b_1 - a_1), \dots, a_s + \epsilon_s(b_s - a_s)).\]
\end{Definition}

For a vector $\bm{\ell}\in \natu_*^s$, we define the following index sets
    \[ B_{\alpha, \bm{\ell}, s} = \left\{ \bm{k} \in \mathbb{N}_{*}^s \mid  \bm{k}_{(\alpha)} = \bm{\ell} \right\} \quad \text{and}\quad T_{\alpha, \bm{\ell}, s} = \left\{ \bm{k} \in \mathbb{N}_{*}^s \mid  \bm{k}_{(\alpha)}\le \bm{\ell} \right\},  \]
where the inequality $\bm{k}_{(\alpha)}\le \bm{\ell}$ is understood component-wise, i.e., $k_{j,(\alpha)}\le \ell_j$ for all $j\in 1{:}s$. With these definitions of the function's variation and the index sets, we can establish a bound on the sum of squared Walsh coefficients.

\begin{Lemma}\label{lem:bound_on_Walsh_partial_sum}
    Let $\alpha\in \natu$ and $p\in (0,1]$. For a subset $\mathfrak{u} \subseteq 1{:}s$ and a function $f$, define its projection onto the active variables by integrating out the inactive ones:
    $$ f_{\mathfrak{u}}(\bm{x}_{\mathfrak{u}}) \coloneqq \int_{[0, 1)^{s-|\mathfrak{u}|}} f(\bm{x}_{\mathfrak{u}}, \bm{x}_{-\mathfrak{u}}) \, \mathrm{d}\bm{x}_{-\mathfrak{u}}. $$
    
    For fixed vectors $\bm{k}_{\alpha}^{+},\bm{\ell}\in \natu_*^s$ such that $\bsk\in B_{\alpha+1,\bm{\ell},s}$, let $\bm{\tau} = \bm{n}_{\alpha}(\bm{k}_{\alpha}^{+})$ and $\mathfrak{u}\coloneqq \supp(\bm{\tau})=\supp(\bm{k})$. Let $\mathfrak{v} \coloneqq \supp(\bm{\ell})$ and $\mathfrak{w} \coloneqq \mathfrak{u} \setminus \mathfrak{v}$. Then the Walsh coefficients of $f$ satisfy
    \begin{align*} 
    \sum_{\substack{\bsk\in B_{\alpha+1,\bm{\ell},s}\\ \text{with fixed $\bm{k}_{\alpha}^{+}$}}} \abs{\hat{f}(\bm{k})}^2 & \leq  \norm{W((\bm{k}^{+}_{\alpha}))}^2_{L^{\infty}([0, 1)^{s})}(b-1)^{(2p-1)_{+}|\mathfrak{v}|}b^{-2p |\bm{\ell}_{\mathfrak{u}}|_1}\\
    & \quad \times \int_{[0,1)^{|\mathfrak{w}|}} \left(V_p^{(|\mathfrak{v}|)}\left( f_{\mathfrak{u}}^{(\bm{\tau}_{\mathfrak{u}})}(\cdot,\bm{x}_{\mathfrak{w}})\right)\right)^2\, \mathrm{d}\bm{x}_{\mathfrak{w}},
    \end{align*}
    where $W(\bm{k}^{+}_{\alpha})(\bsx)=\prod_{j=1}^{s}W({k_j}^{+}_{\alpha})(x_j)$ and $(x)_{+}=\max(0,x)$.
\end{Lemma}

\begin{proof}
Since $\ell_j > 0$ implies that $k_j$ has at least $\alpha+1$ non-zero digits, we see that $\tau_j=\alpha$ for the corresponding coordinates $j\in \mathfrak{v}$. Thus, we naturally have $\mathfrak{v} \subseteq \mathfrak{u}$. For any $\bm{k} \in B_{\alpha+1,\bm{\ell},s}$ with a fixed $\bm{k}_{\alpha}^{+}$ supported on $\mathfrak{u}$, we have $k_j = 0$ for all $j \notin \mathfrak{u}$. Consequently, the Walsh functions evaluate to $1$ for these inactive coordinates, and we can directly project the function onto $\mathfrak{u}$ by integrating out as:
$$\hat{f}(\bm{k}) = \int_{[0, 1)^{|\mathfrak{u}|}} \left( \int_{[0, 1)^{s-|\mathfrak{u}|}} f(\bm{x}_{\mathfrak{u}}, \bm{x}_{-\mathfrak{u}}) \, \mathrm{d}\bm{x}_{-\mathfrak{u}} \right) \prod_{j\in \mathfrak{u}}\overline{\bwal{k_j}(x_j)} \, \mathrm{d}\bm{x}_{\mathfrak{u}} = \widehat{f_{\mathfrak{u}}}(\bm{k}_{\mathfrak{u}}).$$

For all $j \in \mathfrak{w}$, it holds that ${k_j}_{\alpha}^{-}=0$ since $\ell_j = 0$. Thus, the summation over $\bsk_{\alpha}^{-}$ only varies in the coordinates in $\mathfrak{v}$. Define the following function over $[0, 1)^{|\mathfrak{u}|}$:
\begin{equation}
    \label{eq:beta_function_def}
    \beta_{\bm{\ell}_{\mathfrak{u}}}(\bm{t}_{\mathfrak{u}};(\bm{k}_{\alpha}^{+})_{\mathfrak{u}}) = \sum_{(\bsk_{\alpha}^{-})_{\mathfrak{u}}\in T_{1,\bm{\ell}_{\mathfrak{u}},|\mathfrak{u}|}} \widehat{f_{\mathfrak{u}}}(\bm{k}_{\mathfrak{u}})\, \bwal{(\bsk_{\alpha}^{-})_{\mathfrak{u}}} (\bm{t}_{\mathfrak{u}}).
\end{equation}
By applying the inclusion-exclusion formula over $\mathfrak{v}$, we obtain
\begin{align}
    \sum_{\substack{\bsk\in B_{\alpha+1,\bm{\ell},s}\\ \text{with fixed $\bm{k}_{\alpha}^{+}$}}} \abs{\hat{f}(\bm{k})}^2 = \sum_{(\bsk_{\alpha}^{-})_{\mathfrak{u}}\in B_{1,\bm{\ell}_{\mathfrak{u}},|\mathfrak{u}|}} \abs{\widehat{f_{\mathfrak{u}}}(\bm{k}_{\mathfrak{u}})}^2 = \int_{[0, 1)^{|\mathfrak{u}|}} \left(  \sum_{\mathfrak{z} \subseteq \mathfrak{v}} (-1)^{\abs{\mathfrak{z}}} \beta_{\bm{\ell}_{\mathfrak{u}} - \bm{1}_{\mathfrak{z}} } (\bm{t}_{\mathfrak{u}};(\bm{k}_{\alpha}^{+})_{\mathfrak{u}}) \right)^2\, \mathrm{d} \bm{t}_{\mathfrak{u}}.
\end{align}

We substitute~\eqref{eq:walsh_coef} into~\eqref{eq:beta_function_def} and apply the Walsh--Dirichlet kernel to obtain
\begin{align*}
    & \beta_{\bm{\ell}_{\mathfrak{u}}}(\bm{t}_{\mathfrak{u}};(\bm{k}_{\alpha}^{+})_{\mathfrak{u}}) \\
    &=  \sum_{(\bsk_{\alpha}^{-})_{\mathfrak{u}}\in T_{1,\bm{\ell}_{\mathfrak{u}},|\mathfrak{u}|}} (-1)^{|\bm{\tau}_{\mathfrak{u}}|_1} \int_{[0, 1)^{|\mathfrak{u}|}} f_{\mathfrak{u}}^{(\bm{\tau}_{\mathfrak{u}})} (\bm{x}_{\mathfrak{u}}) \overline{\bwal{(\bsk_{\alpha}^{-})_{\mathfrak{u}}}(\bm{x}_{\mathfrak{u}})} \bwal{(\bsk_{\alpha}^{-})_{\mathfrak{u}}} (\bm{t}_{\mathfrak{u}}) W((\bm{k}^{+}_{\alpha})_{\mathfrak{u}})(\bm{x}_{\mathfrak{u}}) \, \mathrm{d}\bm{x}_{\mathfrak{u}} \\
    &= (-1)^{|\bm{\tau}_{\mathfrak{u}}|_1} \int_{[0, 1)^{|\mathfrak{u}|}} f_{\mathfrak{u}}^{(\bm{\tau}_{\mathfrak{u}})} (\bm{x}_{\mathfrak{u}}) W((\bm{k}^{+}_{\alpha})_{\mathfrak{u}})(\bm{x}_{\mathfrak{u}}) \sum_{(\bsk_{\alpha}^{-})_{\mathfrak{u}}\in T_{1,\bm{\ell}_{\mathfrak{u}},|\mathfrak{u}|}}  \bwal{(\bsk_{\alpha}^{-})_{\mathfrak{u}}} (\bm{t}_{\mathfrak{u}}\ominus \bm{x}_{\mathfrak{u}})  \, \mathrm{d}\bm{x}_{\mathfrak{u}}\\
    &= (-1)^{|\bm{\tau}_{\mathfrak{u}}|_1} b^{\abs{\bm{\ell}_{\mathfrak{u}}}_1} \int_{[0, 1)^{|\mathfrak{u}|}} f_{\mathfrak{u}}^{(\bm{\tau}_{\mathfrak{u}})} (\bm{x}_{\mathfrak{u}}) W((\bm{k}^{+}_{\alpha})_{\mathfrak{u}})(\bm{x}_{\mathfrak{u}}) \mathbbm{1}_{[0, b^{-\bm{\ell}_{\mathfrak{u}}})}(\bm{t}_{\mathfrak{u}}\ominus \bm{x}_{\mathfrak{u}}) \, \mathrm{d}\bm{x}_{\mathfrak{u}}.
\end{align*}

Given $\bm{\ell}_{\mathfrak{u}}$ and $\bm{a}_{\mathfrak{u}}\in \natu_0^{|\mathfrak{u}|}$ with $0\le a_j<b^{\ell_j}$, we write
\begin{equation}
    c_{\bm{\ell}_{\mathfrak{u}}, \bm{a}_{\mathfrak{u}}} = \int_{[\bm{a}_{\mathfrak{u}}b^{-\bm{\ell}_{\mathfrak{u}}}, (\bm{a}_{\mathfrak{u}} + \bm{1}_{\mathfrak{u}}) b^{-\bm{\ell}_{\mathfrak{u}}})} f_{\mathfrak{u}}^{(\bm{\tau}_{\mathfrak{u}})} (\bm{x}_{\mathfrak{u}}) W((\bm{k}^{+}_{\alpha})_{\mathfrak{u}})(\bm{x}_{\mathfrak{u}}) \, \mathrm{d}\bm{x}_{\mathfrak{u}},
\end{equation}
Let $\bm{t}_{\mathfrak{u}} \in [\bm{a}_{\mathfrak{u}}b^{-{\bm{\ell}_{\mathfrak{u}}}}, (\bm{a}_{\mathfrak{u}} + \bm{1}_{\mathfrak{u}}) b^{-\bm{\ell}_{\mathfrak{u}}})$, then we have
\begin{align}
    \label{eq:beta_ell_function}
    \beta_{\bm{\ell}_{\mathfrak{u}}}(\bm{t}_{\mathfrak{u}};(\bm{k}_{\alpha}^{+})_{\mathfrak{u}}) & = (-1)^{|\bm{\tau}_{\mathfrak{u}}|_1} b^{\abs{\bm{\ell}_{\mathfrak{u}}}_1} \int_{[\bm{a}_{\mathfrak{u}}b^{-{\bm{\ell}_{\mathfrak{u}}}}, (\bm{a}_{\mathfrak{u}} + \bm{1}_{\mathfrak{u}}) b^{-\bm{\ell}_{\mathfrak{u}}})} f_{\mathfrak{u}}^{(\bm{\tau}_{\mathfrak{u}})} (\bm{x}_{\mathfrak{u}}) W((\bm{k}^{+}_{\alpha})_{\mathfrak{u}})(\bm{x}_{\mathfrak{u}}) \, \mathrm{d}\bm{x}_{\mathfrak{u}} \\
    & = (-1)^{|\bm{\tau}_{\mathfrak{u}}|_1} b^{\abs{\bm{\ell}_{\mathfrak{u}}}_1} c_{\bm{\ell}_{\mathfrak{u}}, \bm{a}_{\mathfrak{u}}} ,
\end{align}
and, similarly, for any non-empty $\mathfrak{z} \subseteq \mathfrak{v}$,
\begin{align*}
    \beta_{\bm{\ell}_{\mathfrak{u}} - \bm{1}_{\mathfrak{z}} } (\bm{t}_{\mathfrak{u}};(\bm{k}_{\alpha}^{+})_{\mathfrak{u}}) = (-1)^{|\bm{\tau}_{\mathfrak{u}}|_1} b^{\abs{\bm{\ell}_{\mathfrak{u}}}_1-\abs{\mathfrak{z}}}  c_{\bm{\ell}_{\mathfrak{u}} - \bm{1}_{\mathfrak{z}} , (\floor{\bm{a}_{\mathfrak{z}} / b}, \bm{a}_{\mathfrak{u}\setminus \mathfrak{z}})}.
\end{align*}
Let $A_{\bm{\ell}_{\mathfrak{u}}} = \{\bm{a}_{\mathfrak{u}} \in \mathbb{N}_0^{|\mathfrak{u}|} \mid 0 \leq a_j < b^{\ell_j}, j \in \mathfrak{u} \}$. Note that $a_j = 0$ identically for all $j \in \mathfrak{w}$. The above consideration leads to
\begin{equation}
   \sum_{(\bsk_{\alpha}^{-})_{\mathfrak{u}}\in B_{1,\bm{\ell}_{\mathfrak{u}},|\mathfrak{u}|}} \abs{\widehat{f_{\mathfrak{u}}}(\bm{k}_{\mathfrak{u}})}^2 = b^{\abs{\bm{\ell}_{\mathfrak{u}}}_1} \sum_{\bm{a}_{\mathfrak{u}} \in A_{\bm{\ell}_{\mathfrak{u}}}} \abs*{\sum_{\mathfrak{z} \subseteq \mathfrak{v}} (-1)^{\abs{\mathfrak{z}}} b^{-\abs{\mathfrak{z}}} c_{\bm{\ell}_{\mathfrak{u}} - \bm{1}_{\mathfrak{z}}, (\floor{\bm{a}_{\mathfrak{z}}/b}, \bm{a}_{\mathfrak{u} \setminus \mathfrak{z}} ) } }^2.
\end{equation}

Following \cite[Chapter 13, page 429]{DP10}, define $\bm{e}_{\mathfrak{u}} = b \floor{\bm{a}_{\mathfrak{u}} / b}$ and $\bm{d}_{\mathfrak{u}} = \bm{a}_{\mathfrak{u}} - \bm{e}_{\mathfrak{u}}$. Then we have
\begin{align*}
    & \sum_{\mathfrak{z} \subseteq \mathfrak{v}} (-1)^{\abs{\mathfrak{z}}} b^{-\abs{\mathfrak{z}}} c_{\bm{\ell}_{\mathfrak{u}} - \bm{1}_{\mathfrak{z}}, (\floor{\bm{a}_{\mathfrak{z}}/b}, \bm{a}_{\mathfrak{u} \setminus \mathfrak{z}} ) } \\
    & = \sum_{\mathfrak{z} \subseteq \mathfrak{v}} (-1)^{\abs{\mathfrak{z}}} b^{-\abs{\mathfrak{z}}} \sum_{\bm{h}_{\mathfrak{z}} \in A_{\bm{1}_{\mathfrak{z}}}} c_{\bm{\ell}_{\mathfrak{u}}, \bm{e}_{\mathfrak{u}}+(\bm{h}_{\mathfrak{z}}, \bm{d}_{\mathfrak{u}\setminus \mathfrak{z}})} \\
    & = \sum_{\mathfrak{z} \subseteq \mathfrak{v}} (-1)^{\abs{\mathfrak{z}}} b^{-\abs{\mathfrak{z}}}b^{-\abs{\mathfrak{v}}+\abs{\mathfrak{z}}} \sum_{\bm{h}_{\mathfrak{v}} \in A_{\bm{1}_{\mathfrak{v}}}} c_{\bm{\ell}_{\mathfrak{u}}, \bm{e}_{\mathfrak{u}}+(\bm{h}_{\mathfrak{z}}, \bm{d}_{\mathfrak{u}\setminus \mathfrak{z}})} \\
    & = b^{-\abs{\mathfrak{v}}}\sum_{\bm{h}_{\mathfrak{v}} \in A_{\bm{1}_{\mathfrak{v}}}}\sum_{\mathfrak{z} \subseteq \mathfrak{v}} (-1)^{\abs{\mathfrak{z}}}   c_{\bm{\ell}_{\mathfrak{u}}, \bm{e}_{\mathfrak{u}}+(\bm{h}_{\mathfrak{z}}, \bm{d}_{\mathfrak{u}\setminus \mathfrak{z}})} \\
    & = b^{-\abs{\mathfrak{v}}}\sum_{\bm{h}_{\mathfrak{v}} \in A_{\bm{1}_{\mathfrak{v}}}} \int_{[\bm{a}_{\mathfrak{u}}b^{-\bm{\ell}_{\mathfrak{u}}}, (\bm{a}_{\mathfrak{u}} + \bm{1}_{\mathfrak{u}}) b^{-\bm{\ell}_{\mathfrak{u}}})} \\
    & \qquad \qquad \qquad \sum_{\mathfrak{z} \subseteq \mathfrak{v}} (-1)^{\abs{\mathfrak{z}}} f_{\mathfrak{u}}^{(\bm{\tau}_{\mathfrak{u}})} (\bm{x}_{\mathfrak{u}}+b^{-\bm{\ell}_{\mathfrak{u}}}(\bm{h}_{\mathfrak{z}}, \bm{d}_{\mathfrak{u}\setminus \mathfrak{z}})) W((\bm{k}^{+}_{\alpha})_{\mathfrak{u}})(\bm{x}_{\mathfrak{u}}+b^{-\bm{\ell}_{\mathfrak{u}}}(\bm{h}_{\mathfrak{z}}, \bm{d}_{\mathfrak{u}\setminus \mathfrak{z}})) \, \mathrm{d}\bm{x}_{\mathfrak{u}}\\
    & = b^{-\abs{\mathfrak{v}}}\sum_{\bm{h}_{\mathfrak{v}} \in A_{\bm{1}_{\mathfrak{v}}}} \int_{[\bm{a}_{\mathfrak{u}}b^{-\bm{\ell}_{\mathfrak{u}}}, (\bm{a}_{\mathfrak{u}} + \bm{1}_{\mathfrak{u}}) b^{-\bm{\ell}_{\mathfrak{u}}})}\pm \Delta(f_{\mathfrak{u}}^{(\bm{\tau}_{\mathfrak{u}})}(\cdot)W((\bm{k}^{+}_{\alpha})_{\mathfrak{u}})(\cdot), J_{\bm{d}_{\mathfrak{v}},\bm{h}_{\mathfrak{v}},\bm{x}_{\mathfrak{v}}})\, \mathrm{d}\bm{x}_{\mathfrak{u}},
\end{align*}
where we have defined
\[ J_{\bm{d}_{\mathfrak{v}},\bm{h}_{\mathfrak{v}},\bm{x}_{\mathfrak{v}}} \coloneqq \prod_{j\in \mathfrak{v}}\left[\frac{x_j-\min(h_j-d_j,0)}{b^{\ell_j}}, \frac{x_j+\max(h_j-d_j,0)}{b^{\ell_j}} \right),\]
and the sign in the integral depends on $J_{\bm{d}_{\mathfrak{v}},\bm{h}_{\mathfrak{v}},\bm{x}_{\mathfrak{v}}}$. 

For fixed $\bm{x}_{\mathfrak{v}}$, let us write
\begin{align*}
    g(\bm{x}_{\mathfrak{v}}) & \coloneqq \int_{[\bm{a}_{\mathfrak{w}}b^{-\bm{\ell}_{\mathfrak{w}}}, (\bm{a}_{\mathfrak{w}} + \bm{1}_{\mathfrak{w}}) b^{-\bm{\ell}_{\mathfrak{w}}})}f_{\mathfrak{u}}^{(\bm{\tau}_{\mathfrak{u}})}(\bm{x}_{\mathfrak{v}},\bm{x}_{\mathfrak{w}})W((\bm{k}^{+}_{\alpha})_{\mathfrak{w}})(\bm{x}_{\mathfrak{w}}) \, \mathrm{d}\bm{x}_{\mathfrak{w}}\\
    & \: = \int_{[0,1)^{|\mathfrak{w}|}}f_{\mathfrak{u}}^{(\bm{\tau}_{\mathfrak{u}})}(\bm{x}_{\mathfrak{v}},\bm{x}_{\mathfrak{w}})W((\bm{k}^{+}_{\alpha})_{\mathfrak{w}})(\bm{x}_{\mathfrak{w}}) \, \mathrm{d}\bm{x}_{\mathfrak{w}},
\end{align*}
where the second equality comes from the fact that $\ell_j=0$ for all $j\in \mathfrak{w}$.
Since the difference operator acts only on the variables $\bm{x}_{\mathfrak{v}}$ independently of the remaining ones $\bm{x}_{\mathfrak{w}}$, the integration and the difference operator can be reordered, giving 
\begin{align*}
    & \sum_{\mathfrak{z} \subseteq \mathfrak{v}} (-1)^{\abs{\mathfrak{z}}} b^{-\abs{\mathfrak{z}}} c_{\bm{\ell}_{\mathfrak{u}} - \bm{1}_{\mathfrak{z}}, (\floor{\bm{a}_{\mathfrak{z}}/b}, \bm{a}_{\mathfrak{u} \setminus \mathfrak{z}} ) } \\
    & = b^{-\abs{\mathfrak{v}}}\sum_{\bm{h}_{\mathfrak{v}} \in A_{\bm{1}_{\mathfrak{v}}}} \int_{[\bm{a}_{\mathfrak{v}}b^{-\bm{\ell}_{\mathfrak{v}}}, (\bm{a}_{\mathfrak{v}} + \bm{1}_{\mathfrak{v}}) b^{-\bm{\ell}_{\mathfrak{v}}})}\pm \Delta\left( g(\cdot) W((\bm{k}^{+}_{\alpha})_{\mathfrak{v}})(\cdot), J_{\bm{d}_{\mathfrak{v}},\bm{h}_{\mathfrak{v}},\bm{x}_{\mathfrak{v}}}\right)\, \mathrm{d}\bm{x}_{\mathfrak{v}} \\
    & = b^{-\abs{\mathfrak{v}}}\sum_{\bm{h}_{\mathfrak{v}} \in A_{\bm{1}_{\mathfrak{v}}}} \int_{[\bm{a}_{\mathfrak{v}}b^{-\bm{\ell}_{\mathfrak{v}}}, (\bm{a}_{\mathfrak{v}} + \bm{1}_{\mathfrak{v}}) b^{-\bm{\ell}_{\mathfrak{v}}})}\pm \Delta\left(g, J_{\bm{d}_{\mathfrak{v}},\bm{h}_{\mathfrak{v}},\bm{x}_{\mathfrak{v}}}\right)W((\bm{k}^{+}_{\alpha})_{\mathfrak{v}})(\bm{x}_{\mathfrak{v}})\, \mathrm{d}\bm{x}_{\mathfrak{v}}\\
    & = b^{-\abs{\mathfrak{v}}}\sum_{\bm{h}_{\mathfrak{v}} \in A_{\bm{1}_{\mathfrak{v}}}} \int_{[\bm{a}_{\mathfrak{u}}b^{-\bm{\ell}_{\mathfrak{u}}}, (\bm{a}_{\mathfrak{u}} + \bm{1}_{\mathfrak{u}}) b^{-\bm{\ell}_{\mathfrak{u}}})} \pm \Delta\left(f_{\mathfrak{u}}^{(\bm{\tau}_{\mathfrak{u}})}(\cdot,\bm{x}_{\mathfrak{w}}), J_{\bm{d}_{\mathfrak{v}},\bm{h}_{\mathfrak{v}},\bm{x}_{\mathfrak{v}}}\right)W((\bm{k}^{+}_{\alpha})_{\mathfrak{u}})(\bm{x}_{\mathfrak{u}})\, \mathrm{d}\bm{x}_{\mathfrak{u}},
\end{align*}
where, in the second equality, we used the fact that the function $W({k_j}^{+}_{\alpha})$ has a period of $b^{-\nu_j+1}$, where $\nu_j\ge \ell_j+1$ is the position of the least significant non-zero digit in the $b$-adic expansion of ${k_j}^{+}_{\alpha}$. Since $h_j$ and $d_j$ are integers, the shift amount $b^{-\ell_j}(h_j - d_j)$ incurred by the difference operator $\Delta$ over the box $J_{\bm{d}_{\mathfrak{v}},\bm{h}_{\mathfrak{v}},\bm{x}_{\mathfrak{v}}}$ is exactly an integer multiple of the period $b^{-\nu_j+1}$. Therefore, $W((\bm{k}^{+}_{\alpha})_{\mathfrak{w}})$ takes the same value at all vertices of the region defining the difference operator and can be factored out of $\Delta$. Therefore, we further have
\begin{align*}
    & \sum_{(\bsk_{\alpha}^{-})_{\mathfrak{u}}\in B_{1,\bm{\ell}_{\mathfrak{u}},|\mathfrak{u}|}} \abs{\widehat{f_{\mathfrak{u}}}(\bm{k}_{\mathfrak{u}})}^2 \\
    & = b^{\abs{\bm{\ell}_{\mathfrak{u}}}_1-2\abs{\mathfrak{v}}} \sum_{\bm{a}_{\mathfrak{v}} \in A_{\bm{\ell}_{\mathfrak{v}}}} \abs*{\sum_{\bm{h}_{\mathfrak{v}} \in A_{\bm{1}_{\mathfrak{v}}}} \int_{[\bm{a}_{\mathfrak{u}}b^{-\bm{\ell}_{\mathfrak{u}}}, (\bm{a}_{\mathfrak{u}} + \bm{1}_{\mathfrak{u}}) b^{-\bm{\ell}_{\mathfrak{u}}})} \pm \Delta\left(f_{\mathfrak{u}}^{(\bm{\tau}_{\mathfrak{u}})}(\cdot,\bm{x}_{\mathfrak{w}}), J_{\bm{d}_{\mathfrak{v}},\bm{h}_{\mathfrak{v}},\bm{x}_{\mathfrak{v}}}\right)W((\bm{k}^{+}_{\alpha})_{\mathfrak{u}})(\bm{x}_{\mathfrak{u}})\, \mathrm{d}\bm{x}_{\mathfrak{u}}}^2\\
    & \le b^{\abs{\bm{\ell}_{\mathfrak{u}}}_1-2\abs{\mathfrak{v}}}  \sum_{\bm{h}_{\mathfrak{v}},\bm{h}'_{\mathfrak{v}} \in A_{\bm{1}_{\mathfrak{v}}}} \sum_{\bm{a}_{\mathfrak{v}} \in A_{\bm{\ell}_{\mathfrak{v}}}} \int_{[\bm{a}_{\mathfrak{u}}b^{-\bm{\ell}_{\mathfrak{u}}}, (\bm{a}_{\mathfrak{u}} + \bm{1}_{\mathfrak{u}}) b^{-\bm{\ell}_{\mathfrak{u}}})} \pm \Delta\left(f_{\mathfrak{u}}^{(\bm{\tau}_{\mathfrak{u}})}(\cdot,\bm{x}_{\mathfrak{w}}), J_{\bm{d}_{\mathfrak{v}},\bm{h}_{\mathfrak{v}},\bm{x}_{\mathfrak{v}}}\right)W((\bm{k}^{+}_{\alpha})_{\mathfrak{u}})(\bm{x}_{\mathfrak{u}})\, \mathrm{d}\bm{x}_{\mathfrak{u}} \\
    & \qquad \qquad \qquad \qquad\times \int_{[\bm{a}_{\mathfrak{u}}b^{-\bm{\ell}_{\mathfrak{u}}}, (\bm{a}_{\mathfrak{u}} + \bm{1}_{\mathfrak{u}}) b^{-\bm{\ell}_{\mathfrak{u}}})} \pm \Delta\left(f_{\mathfrak{u}}^{(\bm{\tau}_{\mathfrak{u}})}(\cdot,\bm{x}_{\mathfrak{w}}), J_{\bm{d}_{\mathfrak{v}},\bm{h}'_{\mathfrak{v}},\bm{x}_{\mathfrak{v}}}\right)W((\bm{k}^{+}_{\alpha})_{\mathfrak{u}})(\bm{x}_{\mathfrak{u}})\, \mathrm{d}\bm{x}_{\mathfrak{u}}.
\end{align*}

Each integral is bounded above by
\begin{align*}
    & \int_{[\bm{a}_{\mathfrak{u}}b^{-\bm{\ell}_{\mathfrak{u}}}, (\bm{a}_{\mathfrak{u}} + \bm{1}_{\mathfrak{u}}) b^{-\bm{\ell}_{\mathfrak{u}}})} \pm \Delta\left(f_{\mathfrak{u}}^{(\bm{\tau}_{\mathfrak{u}})}(\cdot,\bm{x}_{\mathfrak{w}}), J_{\bm{d}_{\mathfrak{v}},\bm{h}_{\mathfrak{v}},\bm{x}_{\mathfrak{v}}}\right)W((\bm{k}^{+}_{\alpha})_{\mathfrak{u}})(\bm{x}_{\mathfrak{u}})\, \mathrm{d}\bm{x}_{\mathfrak{u}} \\
    & \le \norm{W((\bm{k}^{+}_{\alpha})_{\mathfrak{u}})}_{L^{\infty}([0, 1)^{|\mathfrak{u}|})}\int_{[\bm{a}_{\mathfrak{u}}b^{-\bm{\ell}_{\mathfrak{u}}}, (\bm{a}_{\mathfrak{u}} + \bm{1}_{\mathfrak{u}}) b^{-\bm{\ell}_{\mathfrak{u}}})} \left| \Delta\left(f_{\mathfrak{u}}^{(\bm{\tau}_{\mathfrak{u}})}(\cdot,\bm{x}_{\mathfrak{w}}), J_{\bm{d}_{\mathfrak{v}},\bm{h}_{\mathfrak{v}},\bm{x}_{\mathfrak{v}}}\right)\right| \, \mathrm{d}\bm{x}_{\mathfrak{u}}\\
    & \le \norm{W((\bm{k}^{+}_{\alpha})_{\mathfrak{u}})}_{L^{\infty}([0, 1)^{|\mathfrak{u}|})}b^{-\abs{\bm{\ell}_{\mathfrak{u}}}_1/2}\left(\int_{[\bm{a}_{\mathfrak{u}}b^{-\bm{\ell}_{\mathfrak{u}}}, (\bm{a}_{\mathfrak{u}} + \bm{1}_{\mathfrak{u}}) b^{-\bm{\ell}_{\mathfrak{u}}})} \left| \Delta\left(f_{\mathfrak{u}}^{(\bm{\tau}_{\mathfrak{u}})}(\cdot,\bm{x}_{\mathfrak{w}}), J_{\bm{d}_{\mathfrak{v}},\bm{h}_{\mathfrak{v}},\bm{x}_{\mathfrak{v}}}\right)\right|^2 \, \mathrm{d}\bm{x}_{\mathfrak{u}}\right)^{1/2},
\end{align*}
where we used the Cauchy--Schwarz inequality for the last bound.  Let us denote the square root of the last integral by $D_{\bm{a}_{\mathfrak{v}},\bm{h}_{\mathfrak{v}}}$. Then, summing over $\bm{a}_{\mathfrak{v}}$ and choosing $\bm{h}_{\mathfrak{v}}=\bm{h}^{*}_{\mathfrak{v}}\in A_{\bm{1}_{\mathfrak{v}}}$ that maximizes the sum $\sum_{\bm{a}_{\mathfrak{v}} \in A_{\bm{\ell}_{\mathfrak{v}}}} D^2_{\bm{a}_{\mathfrak{v}},\bm{h}_{\mathfrak{v}}}$ gives
\begin{align*}
    \sum_{(\bsk_{\alpha}^{-})_{\mathfrak{u}}\in B_{1,\bm{\ell}_{\mathfrak{u}},|\mathfrak{u}|}} \abs{\widehat{f_{\mathfrak{u}}}(\bm{k}_{\mathfrak{u}})}^2 
    & \le \norm{W((\bm{k}^{+}_{\alpha})_{\mathfrak{u}})}^2_{L^{\infty}([0, 1)^{|\mathfrak{u}|})}b^{-2\abs{\mathfrak{v}}} \sum_{\bm{h}_{\mathfrak{v}},\bm{h}'_{\mathfrak{v}} \in A_{\bm{1}_{\mathfrak{v}}}} \sum_{\bm{a}_{\mathfrak{v}} \in A_{\bm{\ell}_{\mathfrak{v}}}} D_{\bm{a}_{\mathfrak{v}},\bm{h}_{\mathfrak{v}}}D_{\bm{a}_{\mathfrak{v}},\bm{h}'_{\mathfrak{v}}}\\
    & \le \norm{W((\bm{k}^{+}_{\alpha})_{\mathfrak{u}})}^2_{L^{\infty}([0, 1)^{|\mathfrak{u}|})}\max_{\bm{h}_{\mathfrak{v}},\bm{h}'_{\mathfrak{v}} \in A_{\bm{1}_{\mathfrak{v}}}} \sum_{\bm{a}_{\mathfrak{v}} \in A_{\bm{\ell}_{\mathfrak{v}}}} D_{\bm{a}_{\mathfrak{v}},\bm{h}_{\mathfrak{v}}}D_{\bm{a}_{\mathfrak{v}},\bm{h}'_{\mathfrak{v}}}\\
    & = \norm{W((\bm{k}^{+}_{\alpha})_{\mathfrak{u}})}^2_{L^{\infty}([0, 1)^{|\mathfrak{u}|})} \max_{\bm{h}_{\mathfrak{v}} \in A_{\bm{1}_{\mathfrak{v}}}} \sum_{\bm{a}_{\mathfrak{v}} \in A_{\bm{\ell}_{\mathfrak{v}}}} D^2_{\bm{a}_{\mathfrak{v}},\bm{h}_{\mathfrak{v}}}\\
    & = \norm{W((\bm{k}^{+}_{\alpha})_{\mathfrak{u}})}^2_{L^{\infty}([0, 1)^{|\mathfrak{u}|})} \sum_{\bm{a}_{\mathfrak{v}} \in A_{\bm{\ell}_{\mathfrak{v}}}} D^2_{\bm{a}_{\mathfrak{v}},\bm{h}^*_{\mathfrak{v}}}.
\end{align*}

Again, since the difference operator acts only on the variables $\bm{x}_{\mathfrak{v}}$, we obtain
\begin{align*}
    & \sum_{\bm{a}_{\mathfrak{v}} \in A_{\bm{\ell}_{\mathfrak{v}}}} D^2_{\bm{a}_{\mathfrak{v}},\bm{h}^*_{\mathfrak{v}}} \\
    & = \sum_{\bm{a}_{\mathfrak{v}} \in A_{\bm{\ell}_{\mathfrak{v}}}}\int_{[\bm{a}_{\mathfrak{u}}b^{-\bm{\ell}_{\mathfrak{u}}}, (\bm{a}_{\mathfrak{u}} + \bm{1}_{\mathfrak{u}}) b^{-\bm{\ell}_{\mathfrak{u}}})} \left| \Delta\left(f_{\mathfrak{u}}^{(\bm{\tau}_{\mathfrak{u}})}(\cdot,\bm{x}_{\mathfrak{w}}), J_{\bm{d}_{\mathfrak{v}},\bm{h}^*_{\mathfrak{v}},\bm{x}_{\mathfrak{v}}}\right)\right|^2 \, \mathrm{d}\bm{x}_{\mathfrak{u}}\\
    & \le b^{-|\bm{\ell}_{\mathfrak{u}}|_1}\int_{[0,1)^{|\mathfrak{w}|}}\sum_{\bm{a}_{\mathfrak{v}} \in A_{\bm{\ell}_{\mathfrak{v}}}}\sup_{\bm{x}_{\mathfrak{v}}\in [\bm{a}_{\mathfrak{v}}b^{-\bm{\ell}_{\mathfrak{v}}}, (\bm{a}_{\mathfrak{v}} + \bm{1}_{\mathfrak{v}}) b^{-\bm{\ell}_{\mathfrak{v}}})}\left| \Delta\left(f_{\mathfrak{u}}^{(\bm{\tau}_{\mathfrak{u}})}(\cdot,\bm{x}_{\mathfrak{w}}), J_{\bm{d}_{\mathfrak{v}},\bm{h}^*_{\mathfrak{v}},\bm{x}_{\mathfrak{v}}}\right)\right|^2\, \mathrm{d}\bm{x}_{\mathfrak{w}}\\
    & \le (b-1)^{(2p-1)_{+}|\mathfrak{v}|}b^{-2p |\bm{\ell}_{\mathfrak{u}}|_1}\\
    & \quad \times \int_{[0,1)^{|\mathfrak{w}|}} \sum_{\bm{a}_{\mathfrak{v}} \in A_{\bm{\ell}_{\mathfrak{v}}}}\sup_{\bm{x}_{\mathfrak{u}}\in [\bm{a}_{\mathfrak{u}}b^{-\bm{\ell}_{\mathfrak{u}}}, (\bm{a}_{\mathfrak{u}} + \bm{1}_{\mathfrak{u}}) b^{-\bm{\ell}_{\mathfrak{u}}})}\mathrm{Vol}(J_{\bm{d}_{\mathfrak{v}},\bm{h}^*_{\mathfrak{v}} ,\bm{x}_{\mathfrak{v}}})\left| \frac{\Delta\left(f_{\mathfrak{u}}^{(\bm{\tau}_{\mathfrak{u}})}(\cdot,\bm{x}_{\mathfrak{w}}), J_{\bm{d}_{\mathfrak{v}},\bm{h}^*_{\mathfrak{v}},\bm{x}_{\mathfrak{v}}}\right)}{\mathrm{Vol}(J_{\bm{d}_{\mathfrak{v}},\bm{h}^*_{\mathfrak{v}} ,\bm{x}_{\mathfrak{v}}})^p}\right|^2\, \mathrm{d}\bm{x}_{\mathfrak{w}}\\
    & \le (b-1)^{(2p-1)_{+}|\mathfrak{v}|}b^{-2p |\bm{\ell}_{\mathfrak{u}}|_1} \int_{[0,1)^{|\mathfrak{w}|}} \left(V_p^{(|\mathfrak{v}|)}\left( f_{\mathfrak{u}}^{(\bm{\tau}_{\mathfrak{u}})}(\cdot,\bm{x}_{\mathfrak{w}})\right)\right)^2\, \mathrm{d}\bm{x}_{\mathfrak{w}}.
\end{align*}
This completes the proof.
\end{proof}

\subsection{Bounds for functions in weighted Sobolev spaces}

Motivated by the results for the bounds on the Walsh coefficients shown in the previous subsection, let us consider the following weighted Sobolev space.

\begin{Definition}[Weighted Sobolev--variation space]\label{def:weighted_sobolev_variation}
    Let $\alpha\in \natu$, $p\in (0,1]$, and $\bm{\gamma}=(\gamma_{\mathfrak{u}})_{\mathfrak{u}\subseteq 1{:}s}$ be a set of positive weights. We define the weighted Sobolev--variation space $\mathcal{W}_{s, \alpha, p, \bm{\gamma}}$ equipped with norm
    \begin{equation}
        \norm{f}_{{s, \alpha, p, \bm{\gamma}}} = \sum_{\frak{u} \subseteq 1{:}s} \gamma_{\frak{u}}^{-1} \left( \sum_{\frak{v} \subseteq \frak{u}} \sum_{\bm{\tau}_{\frak{u} \setminus \frak{v}} \in \{1, \dotsc, \alpha-1 \}^{\abs{\frak{u} \setminus \frak{v}}}} \int_{[0,1)^{|\frak{u} \setminus \frak{v}|}} \left(V_p^{(|\mathfrak{v}|)}\left( f_{\mathfrak{u}}^{(\bm{\tau}_{\mathfrak{u}\setminus \mathfrak{v}},\bm{\alpha}_{\mathfrak{v}})}(\cdot,\bm{x}_{\frak{u} \setminus \frak{v}})\right)\right)^2\, \mathrm{d}\bm{x}_{\frak{u} \setminus \frak{v}} \right)^{1/2},
    \end{equation}
    where $f_{\frak{u}}$ is the projection of $f$ onto the variables $\bm{x}_{\mathfrak{u}}$ by integrating out $\bm{x}_{-\mathfrak{u}}$:
    $$ f_{\mathfrak{u}}(\bm{x}_{\mathfrak{u}}) \coloneqq \int_{[0, 1)^{s-|\mathfrak{u}|}} f(\bm{x}_{\mathfrak{u}}, \bm{x}_{-\mathfrak{u}}) \, \mathrm{d}\bm{x}_{-\mathfrak{u}}. $$
\end{Definition}

It follows from the definition that 
\[ \left(\int_{[0,1)^{|\frak{u} \setminus \frak{v}|}} \left(V_p^{(|\mathfrak{v}|)}\left( f_{\mathfrak{u}}^{(\bm{\tau}_{\mathfrak{u}\setminus \mathfrak{v}},\bm{\alpha}_{\mathfrak{v}})}(\cdot,\bm{x}_{\frak{u} \setminus \frak{v}})\right)\right)^2\, \mathrm{d}\bm{x}_{\frak{u} \setminus \frak{v}}\right)^{1/2}\le \gamma_{\mathfrak{u}}\norm{f}_{{s, \alpha, p, \bm{\gamma}}}, \]
for any subsets $\mathfrak{v}\subseteq \mathfrak{u}\subseteq 1{:}s$, and Lemma~\ref{lem:bound_on_Walsh_partial_sum} implies
\[ \sum_{\substack{\bsk\in B_{\alpha+1,\bm{\ell},s}\\ \text{with fixed $\bm{k}_{\alpha}^{+}$}}} \abs{\hat{f}(\bm{k})}^2 \leq  \gamma_{\mathfrak{u}}^2 \norm{f}^2_{{s, \alpha, p, \bm{\gamma}}} \norm{W((\bm{k}^{+}_{\alpha}))}^2_{L^{\infty}([0, 1)^{s})}(b-1)^{(2p-1)_{+}|\mathfrak{v}|}b^{-2p |\bm{\ell}_{\mathfrak{u}}|_1}. \]
Moreover, using Lemma~\ref{lemma:W_linfty_bound}, we obtain
\begin{equation}
\label{eq:intermediate_result}
\begin{split}
    \sum_{\substack{\bsk\in B_{\alpha+1,\bm{\ell},s}\\ \text{with fixed $\bm{k}_{\alpha}^{+}$}}} \abs{\hat{f}(\bm{k})}^2 & \le \gamma_{\mathfrak{u}}^2 \norm{f}^2_{{s, \alpha, p, \bm{\gamma}}} (b-1)^{(2p-1)_{+}|\mathfrak{v}|}b^{-2p |\bm{\ell}_{\mathfrak{u}}|_1}\prod_{j\in \mathfrak{u}}\left(\frac{1 + \pi}{2} \right)^2 b^{2\min(\alpha, \tau_j)-2\mu_{\alpha}((k_j)_{\alpha}^{+})}\\
    & \le \gamma_{\mathfrak{u}}^2 \norm{f}^2_{{s, \alpha, p, \bm{\gamma}}}C_{\alpha,p}^{2|\mathfrak{u}|} \, b^{-2\mu_{\alpha}(\bm{k}_{\alpha}^{+})-2p |\bm{\ell}_{\mathfrak{u}}|_1} ,
\end{split}
\end{equation} 
where 
\begin{align}\label{eq:constant_C_alpha_p}
    C_{\alpha,p}=(b-1)^{(2p-1)_{+}/2}\left(\frac{1 + \pi}{2} \right) b^{\alpha}.
\end{align}
This bound trivially applies to the individual terms, which gives
\begin{align}
    \abs{\hat{f}(\bm{k})} & \le \gamma_{\mathfrak{u}} \norm{f}_{{s, \alpha, p, \bm{\gamma}}}C_{\alpha,p}^{|\mathfrak{u}|} \, b^{-\mu_{\alpha}(\bm{k})-p (\mu_{\alpha+1}(\bm{k})-\mu_{\alpha}(\bm{k}))} \notag \\
    & = \gamma_{\mathfrak{u}} \norm{f}_{{s, \alpha, p, \bm{\gamma}}}C_{\alpha,p}^{|\mathfrak{u}|} \, b^{-((1-p)\mu_{\alpha}(\bm{k})+p \mu_{\alpha+1}(\bm{k}))} \notag \\
    & \le \gamma_{\mathfrak{u}} \norm{f}_{{s, \alpha, p, \bm{\gamma}}}C_{\alpha,p}^{|\mathfrak{u}|} \, b^{-(\alpha+p) \mu_{\alpha+1}(\bm{k})/(\alpha+1)},\label{eq:walsh_bound_each_function}
\end{align}
where, for the last bound, we used the inequality $\mu_{\alpha}(k)/\alpha\ge \mu_{\alpha+1}(k)/(\alpha+1)$ that holds for any $\alpha\in \natu$ and $k\in \natu_0$.

{
\begin{remark}
    We can also connect the weighted Sobolev--variation norm to a weighted Sobolev norm. Let $\alpha\in \natu$, $1\le q< \infty$, $1 \le r < \infty$ and $\bm{\gamma}=(\gamma_{\mathfrak{u}})_{\mathfrak{u}\subseteq 1{:}s}$ be a set of positive weights. Define the weighted Sobolev space $\mathcal{W}_{s, \alpha, q, r, \bm{\gamma}}$ equipped with norm
    \begin{align*}
        & \norm{f}_{{s, \alpha, q, r, \bm{\gamma}}} = \sum_{\frak{u} \subseteq 1{:}s} \gamma_{\frak{u}}^{-1}\\
        & \times \left( \sum_{\frak{v} \subseteq \frak{u}} \sum_{\bm{\tau}_{\frak{u} \setminus \frak{v}} \in \{1, \dotsc, \alpha\}^{\abs{\frak{u} \setminus \frak{v}}}} \int_{[0,1)^{|\frak{u} \setminus \frak{v}|}} \left( \int_{[0, 1)^{|\mathfrak{v}|}} \abs*{ \int_{[0, 1)^{s - \abs{\frak{u}}}} f^{(\bm{\tau}_{\mathfrak{u}\setminus \mathfrak{v}},\bm{\alpha}_{\mathfrak{v}}, \bm{0} )} (\bm{x}) \, \mathrm{d}\bm{x}_{-\frak{u}} }^q \, \mathrm{d}\bm{x}_{\frak{v}} \right)^{r/q} \, \mathrm{d}\bm{x}_{\frak{u} \setminus \frak{v}} \right)^{1/r}. 
    \end{align*}
    Notice that this norm is an upper bound of the norm in \cite[Definition~3.3]{DKLNS14}. 
    
    From~\cite{Liu25}, we bound the Vitali variation by the Sobolev norm. For $p \in [1/2, 1]$, let $q = 2/(3 - 2p)$ and we have:
    \begin{align*}
        & \int_{[0,1)^{|\frak{u} \setminus \frak{v}|}} \left(V_p^{(|\mathfrak{v}|)}\left( f_{\mathfrak{u}}^{(\bm{\tau}_{\mathfrak{u}\setminus \mathfrak{v}},\bm{\alpha}_{\mathfrak{v}})}(\cdot,\bm{x}_{\frak{u} \setminus \frak{v}})\right)\right)^2\, \mathrm{d}\bm{x}_{\frak{u} \setminus \frak{v}}\\
        &\leq \int_{[0,1)^{|\frak{u} \setminus \frak{v}|}} \left( \int_{[0, 1)^{|\mathfrak{v}|}} \abs*{ f_{\mathfrak{u}}^{(\bm{\tau}_{\mathfrak{u}\setminus \mathfrak{v}},\bm{\alpha}_{\mathfrak{v}} + \bm{1}_{\frak{v}} )}(\bm{x}_{\frak{v}},\bm{x}_{\frak{u} \setminus \frak{v}})}^q \, \mathrm{d}\bm{x}_{\frak{v}} \right)^{2/q} \, \mathrm{d}\bm{x}_{\frak{u} \setminus \frak{v}} \\
        &= \int_{[0,1)^{|\frak{u} \setminus \frak{v}|}} \left( \int_{[0, 1)^{|\mathfrak{v}|}} \abs*{ \int_{[0, 1)^{s - \abs{\frak{u}}}} f^{(\bm{\tau}_{\mathfrak{u}\setminus \mathfrak{v}},\bm{\alpha}_{\mathfrak{v}} + \bm{1}_{\frak{v}} , \bm{0})} (\bm{x}) \, \mathrm{d}\bm{x}_{-\frak{u}} }^q \, \mathrm{d}\bm{x}_{\frak{v}} \right)^{2/q} \, \mathrm{d}\bm{x}_{\frak{u} \setminus \frak{v}}.
    \end{align*}
    Thus, 
    \begin{align*}
        \norm{f}_{{s, \alpha, p, \bm{\gamma}}} &= \sum_{\frak{u} \subseteq 1{:}s} \gamma_{\frak{u}}^{-1} \left( \sum_{\frak{v} \subseteq \frak{u}} \sum_{\bm{\tau}_{\frak{u} \setminus \frak{v}} \in \{1, \dotsc, \alpha-1 \}^{\abs{\frak{u} \setminus \frak{v}}}} \int_{[0,1)^{|\frak{u} \setminus \frak{v}|}} \left(V_p^{(|\mathfrak{v}|)}\left( f_{\mathfrak{u}}^{(\bm{\tau}_{\mathfrak{u}\setminus \mathfrak{v}},\bm{\alpha}_{\mathfrak{v}})}(\cdot,\bm{x}_{\frak{u} \setminus \frak{v}})\right)\right)^2\, \mathrm{d}\bm{x}_{\frak{u} \setminus \frak{v}} \right)^{1/2}\\
        &\leq \sum_{\frak{u} \subseteq 1{:}s} \gamma_{\frak{u}}^{-1} \left( \sum_{\frak{v} \subseteq \frak{u}} \sum_{\bm{\tau}_{\frak{u} \setminus \frak{v}} \in \{1, \dotsc, \alpha-1 \}^{\abs{\frak{u} \setminus \frak{v}}}} \right. \\
        & \left. \qquad \qquad \int_{[0,1)^{|\frak{u} \setminus \frak{v}|}} \left( \int_{[0, 1)^{|\mathfrak{v}|}} \abs*{ \int_{[0, 1)^{s - \abs{\frak{u}}}} f^{(\bm{\tau}_{\mathfrak{u}\setminus \mathfrak{v}},\bm{\alpha}_{\mathfrak{v}} + \bm{1}_{\frak{v}}, \bm{0} )} (\bm{x}) \, \mathrm{d}\bm{x}_{-\frak{u}} }^q \, \mathrm{d}\bm{x}_{\frak{v}} \right)^{2/q} \, \mathrm{d}\bm{x}_{\frak{u} \setminus \frak{v}}\right)^{1/2}\\
        & \le \norm{f}_{{s, \alpha+1, q, 2, \bm{\gamma}}}.
    \end{align*}
\end{remark}
}
\section{Error bounds: high order convergence and dimension independence}\label{sec:convergence_rates}
In this section, we discuss the convergence rates from two approaches: the median-of-means estimator and the greedy optimization of Hankel digital nets. 

\subsection{Median-of-means Hankel designs}
\label{sec:median-of-means-estimator}
In this subsection, we discuss the median-of-means estimator, in line with recent works \cite{GK26,GL22,GSM24,Pan25,Pan26,PO23,PO24}. Let $r$ be a positive integer, and let $Q_{N,1},\ldots,Q_{N,2r-1}$ be independent and identically distributed (i.i.d.) copies of the QMC rule based on a Hankel random design with a random digital shift in infinite precision. The median-of-means estimator is then defined as 
\[ Q_N^{(r)}(f)=\med\left( Q_{N,1}(f),\ldots,Q_{N,2r-1}(f)\right),\]
where the median is uniquely defined since $2r-1$ is odd. Our main contribution here is to analyze the root-mean-square error of this estimator in the weighted Sobolev--variation space $\mathcal{W}_{s, \alpha, p, \bm{\gamma}}$. A key highlight is that, without any prior knowledge of the parameters $\alpha, p$, and $\bm{\gamma}$, the estimator $Q_N^{(r)}(f)$ achieves a convergence rate of order $N^{-(\alpha+p+1/2)+\varepsilon}$ for arbitrarily small $\varepsilon>0$. Furthermore, the implied constant can be independent of the dimension $s$ under a suitable summability condition on the weights $\bm{\gamma}$.

The following argument builds upon the approach of Pan \cite{Pan26}. However, we extend the known results to a general prime base $b \ge 2$ and employ Hankel designs (rather than uniform designs) for the generating matrices, while conducting our analysis within the weighted Sobolev--variation space $\mathcal{W}_{s, \alpha, p, \bm{\gamma}}$.
Motivated by the structure of the bound on the sum of Walsh coefficients in \eqref{eq:intermediate_result}, and following~\cite{Pan26}, we define the set $K_{\frak{u}}(T)$ as
\begin{equation}
    K_{\frak{u}}(T) = \left\{ \bm{k}_{\frak{u}} \in \mathbb{N}^{\abs{\frak{u}}} \mid p \sum_{j \in \frak{u}} {k_j}_{(\alpha + 1)} + \sum_{j \in \frak{u}} \mu_{\alpha} (k_j) \leq T \right\},
\end{equation}
for a non-empty subset $\frak{u}\subseteq 1{:}s$ and a given parameter $T > 0$. In what follows, we write
    \[ B_{\alpha, \bm{\ell}_{\frak{u}}, \frak{u}} = \left\{ \bm{k}_{\frak{u}} \in \mathbb{N}^{|\frak{u}|} \mid  {\bm{k}_{\frak{u}}}_{(\alpha)} = \bm{\ell}_{\frak{u}} \right\} \quad \text{and}\quad T_{\alpha, \bm{\ell}_{\frak{u}}, \frak{u}} = \left\{ \bm{k}_{\frak{u}} \in \mathbb{N}^{|\frak{u}|} \mid  {\bm{k}_{\frak{u}}}_{(\alpha)} \le \bm{\ell}_{\frak{u}} \right\},  \]
where the inequality ${\bm{k}_{\frak{u}}}_{(\alpha)} \le \bm{\ell}_{\frak{u}}$ is understood component-wise. 

We compute the sum of squared Walsh coefficients $\abs{\hat{f} (\bm{k}_{\frak{u}}, \bm{0})}^2$ on the complement of the set $K_{\frak{u}}(T)$. For any $\theta\in (0,1)$, it holds that
\begin{align}
    &  \sum_{\bm{k}_{\frak{u}} \in \mathbb{N}^{\abs{\frak{u}}}} \abs{\hat{f} (\bm{k}_{\frak{u}}, \bm{0})}^2 \chi \left( \bm{k}_{\frak{u}} \notin  K_{\frak{u}}(T) \right) \notag \\
    &= \sum_{\bm{\ell}_{\frak{u}} \in \mathbb{N}_0^{\abs{\frak{u}}}} \sum_{\bm{k}_{\frak{u}} \in B_{\alpha+1, \bm{\ell}_{\frak{u}}, \frak{u}}} \abs{\hat{f} (\bm{k}_{\frak{u}}, \bm{0})}^2 \chi \left( \bm{k}_{\frak{u}} \notin K_{\frak{u}}(T) \right) \notag \\
    &= \sum_{\bm{\ell}_{\frak{u}} \in \mathbb{N}_0^{\abs{\frak{u}}}} \sum_{{\bm{k}_{\frak{u}}}_{\alpha}^{+} \in \prod_{j\in \frak{u}} (\mathbb{N} \setminus T_{\alpha, \ell_j, 1}) } \sum_{{\bm{k}_{\frak{u}}}_{\alpha}^{-} \in B_{1, \bm{\ell}_{\frak{u}}, \frak{u}}} \abs{\hat{f}({\bm{k}_{\frak{u}}}, \bm{0})}^2 \chi \left( \bm{k}_{\frak{u}} \notin K_{\frak{u}}(T) \right) \notag \\
    & \leq \sum_{\bm{\ell}_{\frak{u}} \in \mathbb{N}_0^{\abs{\frak{u}}}} \sum_{{\bm{k}_{\frak{u}}}_{\alpha}^{+} \in \prod_{j\in \frak{u}} (\mathbb{N} \setminus T_{\alpha, \ell_j, 1}) } \sum_{{\bm{k}_{\frak{u}}}_{\alpha}^{-} \in B_{1, \bm{\ell}_{\frak{u}}, \frak{u}}} b^{2\theta(p \sum_{j \in \frak{v}} {k_j}_{(\alpha + 1)} + \sum_{j \in \frak{u}} \mu_{\alpha} ({k_j}_{\alpha}^{+}) - T)} \abs{\hat{f}({\bm{k}_{\frak{u}}}, \bm{0})}^2 \notag \\
    & \leq \sum_{\bm{\ell}_{\frak{u}} \in \mathbb{N}_0^{\abs{\frak{u}}}} \sum_{{\bm{k}_{\frak{u}}}_{\alpha}^{+} \in \prod_{j\in \frak{u}} (\mathbb{N} \setminus T_{\alpha, \ell_j, 1}) } b^{2\theta(p \abs{\bm{\ell}_{\frak{u}} }_1 + \sum_{j \in \frak{u}} \mu_{\alpha} ({k_j}_{\alpha}^{+}) - T)}  \sum_{{\bm{k}_{\frak{u}}}_{\alpha}^{-} \in B_{1, \bm{\ell}_{\frak{u}}, \frak{u}}} \abs{\hat{f}({\bm{k}_{\frak{u}}}, \bm{0})}^2 \notag \\
    & \leq \gamma_{\frak{u}}^2 \norm{f}_{{s, \alpha, p, \bm{\gamma}}}^2   b^{-2\theta T} C_{\alpha, p}^{2\abs{\frak{u}}} \sum_{\bm{\ell}_{\frak{u}} \in \mathbb{N}_0^{\abs{\frak{u}}}} b^{2(\theta - 1) p \abs{\bm{\ell}_{\frak{u}} }_1} \sum_{{\bm{k}_{\frak{u}}}_{\alpha}^{+} \in \prod_{j\in \frak{u}} (\mathbb{N}_0 \setminus T_{\alpha, \ell_j, 1}) }  \prod_{j \in \frak{u}}  b^{2(\theta - 1) \mu_{\alpha}({k_j}_{\alpha}^{+})} \notag \\
    & = \gamma_{\frak{u}}^2 \norm{f}_{{s, \alpha, p, \bm{\gamma}}}^2   b^{-2\theta T} C_{\alpha, p}^{2\abs{\frak{u}}} \sum_{\bm{\ell}_{\frak{u}} \in \mathbb{N}_0^{\abs{\frak{u}}}} b^{2(\theta - 1) p \abs{\bm{\ell}_{\frak{u}} }_1} \prod_{j \in \frak{u}} \sum_{{{k}_{j}}_{\alpha}^{+} \in \mathbb{N} \setminus T_{\alpha, \ell_j, 1} } b^{2(\theta - 1) \mu_{\alpha}({k_j}_{\alpha}^{+})}  \notag \\
    & \leq \gamma_{\frak{u}}^2 \norm{f}_{{s, \alpha, p, \bm{\gamma}}}^2   b^{-2\theta T} C_{\alpha, p}^{2\abs{\frak{u}}} \sum_{\bm{\ell}_{\frak{u}} \in \mathbb{N}_0^{\abs{\frak{u}}}} b^{2(\theta - 1) p \abs{\bm{\ell}_{\frak{u}} }_1} \prod_{j \in \frak{u}} \frac{1}{\alpha !} \left( {(b-1)}\sum_{i = \ell_j + 1}^{\infty} b^{2(\theta - 1) i} \right)^{\alpha} \notag \\
    &\leq \gamma_{\frak{u}}^2 \norm{f}_{{s, \alpha, p, \bm{\gamma}}}^2   b^{-2\theta T} C_{\alpha, p, \theta}^{\abs{\frak{u}}}, \label{eq:sum_walsh_notinK}
\end{align}
with
\[ C_{\alpha, p, \theta} = C_{\alpha, p}^2 \frac{{(b-1)^{\alpha}}}{\alpha! (b^{2(1-\theta)} - 1)^{\alpha} (1 - b^{2(\theta - 1)(\alpha + p)})}.\] 

We next derive an upper bound for the size $\abs{K_{\frak{u}}(T)}$. 
\begin{Lemma}
For any $T > 0$, we have
\begin{equation}
    \abs*{K_{\frak{u}} (T)} \leq b^{\abs{\frak{u}}} b^{T/(\alpha + p)} \frac{(T+\abs{\frak{u}}\alpha -1)^{\abs{\frak{u}}\alpha -1}}{(\abs{\frak{u}}\alpha - 1)!}.
\end{equation}
\end{Lemma}

\begin{proof}[Proof]
Recall that we have
\[
\abs{K_{\frak{u}}(T)}:=\#\Bigl\{\bm{k}_{\frak{u}} \in\mathbb N^{\abs{\frak{u}}}:\ 
\sum_{j \in \frak{u}} \left( \sum_{r=1}^\alpha k_{j, (r)} + p k_{j, (\alpha + 1)} \right) \le T\Bigr\}.
\]
In each coordinate $j\in \frak{u}$, we denote by $\nu_j$ the number of non-zero digits of {$k_j$}. Note that we write $k_{j, (\nu_j+1)}=k_{j, (\nu_j+2)}=\cdots=0$. When top-$(\alpha{+1})$ positions in coordinate $j$ are fixed, all lower bits below $k_{j, (\alpha+1)}$ are free to choose. Thus we have  
\begin{equation}
    \abs{K_{\frak{u}}(T)} = \sum_{\substack{k_{j, (1)} > \cdots > k_{j, (\nu_j)} > 0,\, j \in \frak{u} \\ \sum_{j \in \frak{u}} (\sum_{r=1}^\alpha k_{j, (r)} + p k_{j, (\alpha + 1) } )\le T } } {(b-1)^{\sum_{j\in \frak{u}}\min(\nu_j,\alpha+1)}}b^{\sum_{j \in \frak{u}} k_{j, (\alpha+1)}}.
\end{equation}
Notice that the above also holds when $k_{j, (\alpha+1)} = 0$. Now we relax the strict decreasing sequence $k_{j, (1)} > \cdots > k_{j, (\nu_j)} > 0,\, j \in \frak{u}$ to obtain
\begin{equation}
\abs{K_{\frak{u}}(T)} \leq 
\sum_{\substack{k_{j, (1)} \ge\cdots\ge k_{j, (\alpha+1)} \ge 0,\, j \in \frak{u} \\ \sum_{j \in \frak{u}} (\sum_{r=1}^\alpha k_{j, (r)} + p k_{j, (\alpha + 1) } )\le T } }
{(b-1)^{|\frak{u}|(\alpha+1)}}b^{\sum_{j \in \frak{u}} k_{j, (\alpha+1)}}.
\end{equation}

Let $M_j:=\sum_{r=1}^{\alpha} k_{j, (r)} + p k_{j, (\alpha+1)}$. Since $\sum_{r=1}^\alpha k_{j, (r)} \ge \alpha k_{j, (\alpha + 1)} $, we have $0\leq k_{j, (\alpha+1)} \leq \lfloor M_j / (\alpha + p) \rfloor$, leading to
{
\begin{align*}
    \abs{K_{\frak{u}}(T)} & \le (b-1)^{|\frak{u}|(\alpha+1)}\sum_{\substack{k_{j, (1)} \ge\cdots\ge k_{j, (\alpha+1)} \ge 0,\, j \in \frak{u} \\ \sum_{j \in \frak{u}} M_j\le T } }b^{\sum_{j \in \frak{u}} \lfloor M_j / (\alpha + p) \rfloor}  \\
    & \le (b-1)^{|\frak{u}|(\alpha+1)}b^{T / (\alpha + p) }\sum_{\substack{k_{j, (1)} \ge\cdots\ge k_{j, (\alpha)} \ge 0,\, j \in \frak{u} \\ \sum_{j \in \frak{u}} \sum_{r=1}^{\alpha} k_{j, (r)}\le T } }1\\
    & \le (b-1)^{|\frak{u}|(\alpha+1)}b^{T / (\alpha + p) }\sum_{\substack{a_{j,r}\ge 0,\, j \in \frak{u}, r=1,\ldots,\alpha \\ \sum_{j \in \frak{u}} \sum_{r=1}^{\alpha} a_{j,r}\le \lfloor T\rfloor} }1\\
    & = (b-1)^{|\frak{u}|(\alpha+1)}b^{T / (\alpha + p) }\binom{\lfloor T\rfloor+|\frak{u}|\alpha}{|\frak{u}|\alpha} \le (b-1)^{|\frak{u}|(\alpha+1)} b^{T / (\alpha + p) }\frac{(T+|\frak{u}|\alpha)^{|\frak{u}|\alpha}}{(|\frak{u}|\alpha)!},
\end{align*}
which proves the lemma.
}
\end{proof}
Following~\cite{Pan26}, for $0 < \delta < 1$, $\xi > 0$, we let
\begin{equation}
    \label{eq:def_t_u_m}
    T_{\frak{u}, m, \delta, \xi} = (\alpha + p) m - (\alpha + p) \left(\log_b \Gamma_{m, \xi} - \log_b (\delta) - \frac{\log_b \gamma_{\frak{u}}}{\alpha + p + \xi}\right),
\end{equation}
where
\[
    \Gamma_{m, \xi} = 1 + \sum_{\substack{\frak{u} \subseteq 1{:}s \\ \frak{u} \neq \varnothing}} \gamma_{\frak{u}}^{\frac{1}{\alpha + p + \xi}} \frac{(b-1)^{\abs{\frak{u}} (\alpha + 1)} ((\alpha + p)m + \abs{\frak{u}} \alpha)^{\abs{\frak{u}} \alpha} }{(\abs{\frak{u}} \alpha )!}.
\]
In the following, we bound the probability of a ``bad event'', i.e., when there exists a $\bm{k} \in \bigcup_{\substack{\frak{u} \subseteq 1{:}s \\ \frak{u} \neq \varnothing}} K_{\frak{u}} (T_{\frak{u}, m, \delta, \xi})$ that is in the dual net $\mathcal{D}^{\perp}$. 
\begin{Lemma}\label{lem:union_bound}
    Given $0 < \delta < 1$, we have
    \[
        \Pr\left( \bigcup_{\bm{k} \in \bigcup_{\substack{\frak{u} \subseteq 1{:}s \\ \frak{u} \neq \varnothing}} K_{\frak{u}} (T_{\frak{u}, m, \delta, \xi}) }  \left\{ \bm{k} \in \mathcal{D}^{\perp} \right\} \right) \leq \delta. 
    \]
\end{Lemma}
\begin{proof}
    Notice that we have $T_{\frak{u}, m, \delta, \xi} \leq (\alpha + p) m$ for any $\frak{u} \subseteq 1{:}s$ and $m > 0$. Thus, we have
\begin{align*}
    \sum_{\substack{\frak{u} \subseteq 1{:}s \\ \frak{u} \neq \varnothing} }\abs*{K_{\frak{u}} (T_{\frak{u}, m, \delta, \xi})} &\leq \delta b^m \sum_{\substack{\frak{u} \subseteq 1{:}s \\ \frak{u} \neq \varnothing} } \frac{\gamma_{\frak{u}}^{\frac{1}{\alpha + p + \xi}} }{\Gamma_{m, \xi}} (b-1)^{\abs{\frak{u}} (\alpha + 1)} \frac{(T_{\frak{u}, m, \delta, \xi} + \abs{\frak{u}}\alpha )^{\abs{\frak{u}}\alpha }}{(\abs{\frak{u}}\alpha )!} \\
    &\leq \delta b^m \frac{1}{{\Gamma_{m, \xi}}} \sum_{\substack{\frak{u} \subseteq 1{:}s \\ \frak{u} \neq \varnothing} } {\gamma_{\frak{u}}^{\frac{1}{\alpha + p + \xi}} } (b-1)^{\abs{\frak{u}} (\alpha + 1)} \frac{((\alpha + p)m +\abs{\frak{u}}\alpha )^{\abs{\frak{u}}\alpha }}{(\abs{\frak{u}}\alpha )!} \\
    &\leq \delta b^m. \qedhere
\end{align*}
\end{proof}

Now we present the mean squared error of the median-of-means estimator. 
\begin{Theorem}\label{thm:median-of-means}
    Given $0 < \delta < 1/8$, when $f \in \mathcal{W}_{s, \alpha, p, \bm{\gamma}}$ with $\alpha \in \mathbb{N}$, $p\in (0,1]$, the weighted Sobolev--variation space in Definition~\ref{def:weighted_sobolev_variation}, the median-of-means estimator $Q_N^{(r)}(f)$ satisfies
\begin{equation}
    \label{eq:mse_median_estimator}
    \EE\left[\left(Q_N^{(r)}(f) - I\right)^2\right] \leq C_{\epsilon, \theta, s} b^{(\epsilon - 2\theta (\alpha + p) - 1 ) m}  \norm{f}_{{s, \alpha, p, \bm{\gamma}}}^2 + 2(8\delta)^r \norm{f}^2_{L^{\infty}([0, 1)^s)},
\end{equation}
where $\epsilon > 0$, $0 < \theta < 1$, and the constant $C_{\epsilon, \theta, s} \to \infty$ when $\epsilon \to 0$ or $\theta \to 1$. 
\end{Theorem}
\begin{proof}
Let $Q_N(f)$ be the QMC rule based on a single HRD with a random digital shift. Let us define the set of vectors $\mathcal{K}$ as
\[ \mathcal{K}=\bigcup_{\emptyset \ne \frak{u}\subset 1{:}s}\bigcup_{\bsk_{\frak{u}}\in K_{\frak{u}}(T_{\frak{u},m,\delta,\xi})}\left\{ (\bsk_{\frak{u}},\bszero)\right\} \]
and consider the event $\mathcal{A}$ that the dual net contains at least one vector from $\mathcal{K}$, which is given by $\mathcal{A}=\left\{ \mathcal{K}\cap \mathcal{D}^{\perp}\ne \emptyset\right\}$. 
Noting that the event $\mathcal{A}$ happens independently of a random digital shift, we have $\E{Q_N(f)\mid \mathcal{A}^c}=\mu$. Applying Chebyshev's inequality, we obtain
\begin{align*}
    & \Pr\left( \abs{Q_N(f) - I}^2 \ge \delta^{-1}\Pr(\mathcal{A}^c) \var{Q_N(f) \mid \mathcal{A}^c } \right) \\
    & = \Pr\left( \mathcal{A} \right)\Pr\left( \abs{Q_N(f) - I}^2 \ge \delta^{-1}\Pr(\mathcal{A}^c) \var{Q_N(f) \mid \mathcal{A}^c } \mid \mathcal{A} \right) \\
    & \quad + \Pr\left( \mathcal{A}^c \right)\Pr\left( \abs{Q_N(f) - I}^2 \ge \delta^{-1}\Pr(\mathcal{A}^c) \var{Q_N(f) \mid \mathcal{A}^c } \mid \mathcal{A}^c \right) \\
    & \le \Pr\left( \mathcal{A} \right)+\Pr\left( \mathcal{A}^c \right)\cdot \frac{1}{\delta^{-1}\Pr(\mathcal{A}^c)}\le 2\delta,
\end{align*}
where the last bound follows from Lemma~\ref{lem:union_bound}.
Notice that
\begin{align*}
    \delta^{-1}\Pr(\mathcal{A}^c) \var{Q_N(f) \mid \mathcal{A}^c } & = \delta^{-1}\Pr(\mathcal{A}^c) \sum_{\bsk\in \natu_*^s}\abs{\hat{f}(\bsk)}^2\Pr\left( \bsk\in \mathcal{D}^{\perp} \mid \mathcal{A}^c\right) \\
    & \le \delta^{-1}\sum_{\bsk\in \natu_*^s\setminus \mathcal{K}}\abs{\hat{f}(\bsk)}^2\Pr\left( \bsk\in \mathcal{D}^{\perp} \right) \\
    & = \delta^{-1}b^{-m}\sum_{\bsk\in \natu_*^s\setminus \mathcal{K}}\abs{\hat{f}(\bsk)}^2 \\
    & \leq \norm{f}_{{s, \alpha, p, \bm{\gamma}}}^2 \frac{\Gamma_{m, \xi}^{2\theta (\alpha + p)} }{b^{ (1+2\theta (\alpha + p)) m} } \delta^{-2\theta (\alpha + p)} \sum_{\substack{\frak{u} \subseteq 1{:}s \\ \frak{u} \neq \varnothing} }  \gamma_{\frak{u}}^{\frac{2\xi + (2-2\theta)(\alpha + p)}{\alpha + p 
     + \xi}} C_{\alpha, p, \theta, \frak{u}} ,
\end{align*}
where the second equality follows from Lemma~\ref{lem:prob_each_k}, and the last bound follows from \eqref{eq:sum_walsh_notinK} and the definition of $T_{\frak{u}, m, \delta, \xi}$.

For $Q_N^{(r)}(f)$ being the median of $(2r - 1)$ i.i.d.\ estimators, we have
\begin{equation}
    \Pr\left( \abs{Q_N^{(r)}(f) - I}^2 \ge \delta^{-1}\Pr(\mathcal{A}^c) \var{Q_N(f) \mid \mathcal{A}^c } \right) \le {2r -1 \choose r} (2\delta)^r \le (8\delta)^r.
\end{equation}
Notice that when $f \in \mathcal{W}_{s, \alpha, p, \bm{\gamma}}$ for $\alpha \in \mathbb{N}$, $p\in (0,1]$, we see that $\norm{f}_{L^{\infty}([0, 1)^s)}$, and so 
\begin{align*}
    \abs*{Q_N(f) - I} &= \abs*{\frac1N \sum_{i = 0}^{N - 1} \left(\int_{[0, 1)^s} \left(f(\bm{x}_i) - f(\bm{x})\right)\, \mathrm{d}\bm{x} \right)}\\
    &\leq \frac1N \sum_{i = 0}^{N - 1} \abs*{ \int_{[0, 1)^s} \left(f(\bm{x}_i) - f(\bm{x})\right)\, \mathrm{d}\bm{x}} \leq 2 \norm{f}_{L^{\infty}([0, 1)^s)},
\end{align*}
with probability $1$. Thus, the mean squared error is bounded as
\begin{align*}
    \EE[({Q}_N^{(r)}(f) - I)^2]&= \EE[({Q}_N^{(r)}(f) - I)^2 \chi \left( \abs{Q_N(f) - I}^2 < \delta^{-1}\Pr(\mathcal{A}^c) \var{Q_N(f) \mid \mathcal{A}^c } \right) ]\\
    & \quad + \EE[({Q}_N^{(r)}(f) - I)^2 \chi \left( \abs{Q_N(f) - I}^2 \ge \delta^{-1}\Pr(\mathcal{A}^c) \var{Q_N(f) \mid \mathcal{A}^c } \right) ] \\
    &\leq \norm{f}_{{s, \alpha, p, \bm{\gamma}}}^2 \frac{\Gamma_{m, \xi}^{2\theta (\alpha + p)} }{b^{ (1+2\theta (\alpha + p)) m} } \delta^{-2\theta (\alpha + p)} \sum_{\substack{\frak{u} \subseteq 1{:}s \\ \frak{u} \neq \varnothing} }  \gamma_{\frak{u}}^{\frac{2\xi + (2-2\theta)(\alpha + p)}{\alpha + p 
     + \xi}} C_{\alpha, p, \theta, \frak{u}}\\
     & \quad + 2 (8\delta)^r \norm{f}^2_{L^{\infty}([0, 1)^s)}. 
\end{align*}
Notice that $\Gamma_{m, \xi}^{2\theta (\alpha + p)} = o(b^{\epsilon m})$, for any $\epsilon > 0$, enabling us to find a constant $\tilde{C}_{\epsilon, s} > 0$, independent of $m$, such that $\Gamma_{m, \xi}^{2\theta (\alpha + p)} \leq \tilde{C}_{\epsilon, s} b^{\epsilon m}$ for any $m > 0$. This leads to
\begin{align*}
    \EE[({Q}_N^{(r)} (f) - I)^2] & \leq \norm{f}_{{s, \alpha, p, \bm{\gamma}}}^2\frac{\tilde{C}_{\epsilon, s}}{b^{(2\theta (\alpha + p) + 1 - \epsilon ) m}} \delta^{-2\theta (\alpha + p)} \sum_{\substack{\frak{u} \subseteq 1{:}s \\ \frak{u} \neq \varnothing} }  \gamma_{\frak{u}}^{\frac{2\xi + (2- 2\theta)(\alpha + p)}{\alpha + p + \xi}} C_{\alpha, p, \theta, \frak{u}} \\
    & \quad + 2(8\delta)^r \norm{f}^2_{L^{\infty}([0, 1)^s)}.
\end{align*}
By substituting 
\[C_{\epsilon, \theta, s} \coloneqq \tilde{C}_{\epsilon, s} \delta^{-2\theta (\alpha + p)} \sum_{\substack{\frak{u} \subseteq 1{:}s \\ \frak{u} \neq \varnothing} }  \gamma_{\frak{u}}^{\frac{2\xi + (2- 2\theta)(\alpha + p)}{\alpha + p + \xi}} C_{\alpha, p, \theta, \frak{u}}\]
into the above equation, we reach the conclusion.
\end{proof}

{
\begin{remark}[Choice of $r$]\label{rem:choice_r}
The upper bound on the mean squared error given in \eqref{eq:mse_median_estimator} consists of two terms. To ensure that the second term does not dominate, it suffices to set
\[ (8\delta)^r = b^{-(2\theta (\alpha + p) + 1 ) m}=N^{-(2\theta (\alpha + p) + 1 )}. \]
Thus, $r$ can be chosen independently of the dimension $s$ as
\[ r = \left\lceil \left(2\theta (\alpha + p) + 1\right)\frac{\log N}{\log (8\delta)^{-1}}\right\rceil . \]
Since the median-of-means estimator requires $rN$ function evaluations, this choice of $r$ affects the convergence rate only by a dimension-independent logarithmic factor. Furthermore, in terms of the total computational cost, it preserves the dimension independence of the implied constant (as discussed in the next remark). If the smoothness parameters $\alpha$ and $p$ are unknown in advance, one can instead choose $r = \lceil g(N)\log N\rceil$ using a slowly increasing function $g$, such as $g(N)=\max(1,\log\log N)$ or $g(N)=\log N$, as suggested in \cite{GK26}.
\end{remark}

Before moving on, we remark that the factor $\norm{f}_{L^{\infty}([0, 1)^s)}$ in the second term of \eqref{eq:mse_median_estimator} can be bounded independently of the dimension $s$ under a mild condition on the weights $\bm{\gamma}$. For any $f \in \mathcal{W}_{s, \alpha, p, \bm{\gamma}}$, it follows from the fundamental theorem of calculus that the function can be decomposed into its projections:
\[  f(\bm{x}) = \sum_{\mathfrak{u} \subseteq \{1:s\}} \int_{[0, 1)^{|\frak{u}|}} \Delta_{\mathfrak{u}}(f_u; \bm{y}_{\mathfrak{u}}, \bm{x}_{\mathfrak{u}}) \, \mathrm{d}\bm{y}_{\mathfrak{u}}. \]
Here, $\Delta_{\mathfrak{u}}(f_{\mathfrak{u}}; \bm{y}_{\mathfrak{u}}, \bm{x}_{\mathfrak{u}})$ denotes the mixed difference of the projection $f_{\mathfrak{u}}$ with respect to the variables in $\mathfrak{u}$, defined by
\begin{align*}
    \Delta_{\mathfrak{u}}(f_{\mathfrak{u}}; \bm{y}_{\mathfrak{u}}, \bm{x}_{\mathfrak{u}})  = \sum_{\frak{v}\subseteq \frak{u}}(-1)^{\abs{\frak{u}\setminus \frak{v}}}f_u(\bm{x}_{\frak{v}},\bm{y}_{\frak{u}\setminus \frak{v}}) = \int_{y_{j_1}}^{x_{j_1}}\cdots \int_{y_{j_k}}^{x_{j_k}}f_{\frak{u}}^{(\bm{1}_{\frak{u}})}(\bm{t}_{\frak{u}})\, \mathrm{d}\bm{t}_{\frak{u}},
\end{align*} 
where $\frak{u}=\{j_1,\ldots,j_k\}$. In fact, the case $s=1$ is easily confirmed as
\[ f(x) = \int_0^1f(y)\, \mathrm{d}y+\int_0^1(f(x)-f(y))\, \mathrm{d}y = f_{\emptyset}+\int_0^1\Delta_{\{1\}}(f;y,x)\, \mathrm{d}y. \] 
Assuming that the decomposition holds for $s=k$, for the case $s=k+1$, we apply the fundamental theorem of calculus to the variable $x_{k+1}$, and then use the induction hypothesis to obtain
\begin{align*}
    f(\bm{x}_{\{1:k+1\}}) & = \int_0^1 f(\bm{x}_{\{1:k\}}, y_{k+1}) \, \mathrm{d}y_{k+1} + \int_0^1 \Delta_{\{k+1\}}(f; y_{k+1}, x_{k+1}) \, \mathrm{d}y_{k+1} \\
    & = f_{\{1:k\}}(\bm{x}_{\{1:k\}}) + \int_0^1 \sum_{\mathfrak{u} \subseteq \{1:k\}} \int_{[0, 1)^{|\frak{u}|}} \Delta_{\mathfrak{u}}\left(\Delta_{\{k+1\}}(f; y_{k+1}, x_{k+1}); \bm{y}_{\mathfrak{u}}, \bm{x}_{\mathfrak{u}}\right) \, \mathrm{d}\bm{y}_{\mathfrak{u}} \, \mathrm{d}y_{k+1}\\
    & = \sum_{\mathfrak{u} \subseteq \{1:k\}} \int_{[0, 1)^{|\frak{u}|}} \Delta_{\mathfrak{u}}(f_u; \bm{y}_{\mathfrak{u}}, \bm{x}_{\mathfrak{u}}) \, \mathrm{d}\bm{y}_{\mathfrak{u}} + \sum_{k+1\in \mathfrak{u} \subseteq \{1:k+1\}} \int_{[0, 1)^{|\frak{u}|}} \Delta_{\mathfrak{u}}(f_u; \bm{y}_{\mathfrak{u}}, \bm{x}_{\mathfrak{u}}) \, \mathrm{d}\bm{y}_{\mathfrak{u}}\\
    & = \sum_{\mathfrak{u} \subseteq \{1:k+1\}} \int_{[0, 1)^{|\frak{u}|}} \Delta_{\mathfrak{u}}(f_u; \bm{y}_{\mathfrak{u}}, \bm{x}_{\mathfrak{u}}) \, \mathrm{d}\bm{y}_{\mathfrak{u}} .
\end{align*}

By the triangle inequality and the integral representation of the mixed difference, we have
\begin{align*}
|f(\bm{x})| & \le \sum_{\mathfrak{u} \subseteq {1:s}} \int_{[0, 1)^{|\mathfrak{u}|}} \left| \Delta_{\mathfrak{u}}(f_{\mathfrak{u}}; \bm{y}_{\mathfrak{u}}, \bm{x}_{\mathfrak{u}}) \right| \, \mathrm{d}\bm{y}_{\mathfrak{u}} \\
& \le \sum_{\mathfrak{u} \subseteq {1:s}} \int_{[0, 1)^{|\mathfrak{u}|}}|f_{\frak{u}}^{(\bm{1}_{\frak{u}})}(\bm{t}_{\frak{u}})|\, \mathrm{d}\bm{t}_{\mathfrak{u}} \le \sum_{\mathfrak{u} \subseteq {1:s}} \gamma_{\frak{u}}\gamma^{-1}_{\frak{u}}\norm{f^{(\bm{1}_{\frak{u}})}_{\frak{u}}}_{L_2} \\
& \le \left( \max_{\mathfrak{u} \subseteq {1:s}}\gamma_{\frak{u}}\right) \sum_{\mathfrak{u} \subseteq {1:s}} \gamma^{-1}_{\frak{u}}\norm{f^{(\bm{1}_{\frak{u}})}_{\frak{u}}}_{L_2} \le \left( \max_{\mathfrak{u} \subseteq {1:s}}\gamma_{\frak{u}}\right)\norm{f}_{{s, \alpha, p, \bm{\gamma}}}.
\end{align*}
Thus, $\norm{f}_{L^{\infty}([0, 1)^s)}$ is bounded independently of the dimension $s$ if the weights $\gamma_{\frak{u}}$ are uniformly bounded.
}

\begin{remark}[Dimension-independent convergence rate]
Now we consider the case when $s \to \infty$ and investigate whether the constant $C_{\epsilon, \theta, s}$ in the error bound~\eqref{eq:mse_median_estimator} is still finite. To this end, we can consider the following conditions, as $s \to \infty$,
\begin{align}
    \tilde{C}_{\epsilon, s} < \infty \quad \text{and}\quad    \sum_{\substack{\frak{u} \subseteq 1{:}s \\ \frak{u} \neq \varnothing} }  \gamma_{\frak{u}}^{\frac{2\xi + (2- 2\theta)(\alpha + p)}{\alpha + p + \xi}} C_{\alpha, p, \theta, \frak{u}}  < \infty. \label{eq:dimension_bound}
\end{align}
Suppose we have the product weight $\{ \gamma_j \leq 1 , j \in \mathbb{N} \}$ such that
\[
    \gamma_{\frak{u}} \leq \prod_{j \in \frak{u}} \gamma_j, \quad \text{and} \quad \sum_{j=1}^{\infty} \gamma_j^{\min(\frac{1}{\alpha + p + \xi},  \frac{2\xi}{\alpha + p + \xi})} < \infty,
\]
then we can show that the conditions are both satisfied. Specifically, we can find the bound $\tilde{C}_{\epsilon, s}$ in the following derivation:
\begin{align*}
    \Gamma_{m, \xi}^{2\theta(\alpha + p)} &= \left( 1 + \sum_{\substack{\frak{u} \subseteq 1{:}s \\ \frak{u} \neq \varnothing}} \gamma_{\frak{u}}^{\frac{1}{\alpha + p + \xi}} \frac{(b-1)^{\abs{\frak{u}}(\alpha + 1)} ((\alpha + p)m + \abs{\frak{u}} \alpha  )^{\abs{\frak{u}} \alpha } }{(\abs{\frak{u}} \alpha  )!} \right)^{2\theta(\alpha + p)} \\
    &\leq \left( 1 + \sum_{\substack{\frak{u} \subseteq 1{:}s \\ \frak{u} \neq \varnothing}} \gamma_{\frak{u}}^{\frac{1}{\alpha + p + \xi}} \left( \frac{(b - 1)^{\alpha + 1} ((\alpha + p)m + \abs{\frak{u}} \alpha )^{ \alpha} }{(\abs{\frak{u}} \alpha / e )^{\alpha}} \right)^{\abs{\frak{u}}} \right)^{2\theta(\alpha + p)}\\
    &\leq \left( 1 + \sum_{\substack{\frak{u} \subseteq 1{:}s \\ \frak{u} \neq \varnothing}} \gamma_{\frak{u}}^{\frac{1}{\alpha + p + \xi}} \left( {(b - 1)^{\alpha + 1} ( e(\alpha + p)m + e )^{ \alpha} } \right)^{\abs{\frak{u}}} \right)^{2\theta(\alpha + p)}\\
    &\leq \prod_{j=1}^s \left(  1 + \gamma_{j}^{\frac{1}{\alpha + p + \xi}} \left( {(b - 1)^{\alpha + 1} e^{\alpha} ( (\alpha + p)m + 1 )^{ \alpha} } \right) \right)^{2 \theta (\alpha + p)} \leq \tilde{C}_{\epsilon, s} b^{\epsilon m},
\end{align*}
which holds for any $\epsilon > 0$. Notice that we use the inequality $n! \geq (n / e)^n$ for $n \in \mathbb{N}$ in the second line and $\tilde{C}_{\epsilon, s} < \infty$ as $s \to \infty$ whenever $\sum_{j=1}^{\infty} \gamma_j^{\frac{1}{\alpha + p + \xi}} < \infty$; see Lemma 3 of~\cite{HN03} for instance. Similarly, we can show that 
\begin{align*}
    \sum_{\substack{\frak{u} \subseteq 1{:}s \\ \frak{u} \neq \varnothing} }  \gamma_{\frak{u}}^{\frac{2\xi + (2- 2\theta)(\alpha + p)}{\alpha + p + \xi}} C_{\alpha, p, \theta, \frak{u}} &\leq \sum_{\substack{\frak{u} \subseteq 1{:}s \\ \frak{u} \neq \varnothing} }  \gamma_{\frak{u}}^{\frac{2\xi}{\alpha + p + \xi}} C_{\alpha, p, \theta, \frak{u}}\\
    &\leq \prod_{j=1}^s \left( 1 + \gamma_j^{\frac{2\xi}{\alpha + p + \xi}} C_{\alpha, p, \theta} \right)\\
    &\leq \exp \left(\sum_{j=1}^s \gamma_j^{\frac{2\xi}{\alpha + p + \xi}} C_{\alpha, p, \theta} \right) < \infty,
\end{align*}
 whenever $\sum_{j=1}^{\infty} \gamma_j^{\frac{2\xi}{\alpha + p + \xi}} < \infty$. 

To conclude, notice that the sum $\sum_{j=1}^{\infty} \gamma_j^{x}$ decreases as $x > 0$ increases. We seek $\xi^{*} = \arg \max \left\{ \min (\frac{1}{\alpha + p + \xi},  \frac{2\xi}{\alpha + p + \xi}) \right\}$ to admit the largest class of weights $\{\gamma_j \leq 1, j \in \mathbb{N}\}$, which yields $\xi^{*} = 1/2$. Thus, if $\sum_{j=1}^{\infty} \gamma_j^{\frac{1}{\alpha + p + 1/2}} < \infty$, we achieve the dimension-independent convergence rate. This result is consistent with the setting in~\cite{Pan26}. 
\end{remark}

\subsection{Greedy optimization of Hankel designs}
\label{sec:optimization_hankel}

The major advantage of the median-of-means estimator lies in its universality, i.e., although we do not need to specify the smoothness ($\alpha$ and $p$ in our function space setting) and the weights $\bm{\gamma}$ in advance, the estimator can adaptively attain a high-order convergence with dimension-independent implied constants, provided that $r$ is suitably chosen (as discussed in Remark~\ref{rem:choice_r}). In some applications, such as partial differential equations with random coefficients \cite{DKLNS14}, however, the smoothness and weights may be known a priori based on the information of the underlying random fields. In such cases, optimization of QMC point sets is desirable and often performs well in practice. 

In this subsection, we consider a simple greedy optimization of Hankel designs.
Let $\alpha\in \natu$, $p\in (0,1]$, and $\bm{\gamma}$ be given.  For a digital net $P\subset [0,1)^s$ in base $b$ with $N=b^m$ points (without any digital shift applied), the worst-case error in $\mathcal{W}_{s, \alpha, p, \bm{\gamma}}$ is defined as
\[ e(\mathcal{W}_{s, \alpha, p, \bm{\gamma}},P) \coloneqq \sup_{\substack{f\in \mathcal{W}_{s, \alpha, p, \bm{\gamma}}\\ \norm{f}_{{s, \alpha, p, \bm{\gamma}}}\le 1} }|I(f)-Q_N(f)|. \]
Here, we assume that either the worst-case error or its upper bound has a computable formula. Our greedy optimization proceeds by first drawing Hankel designs independently and randomly several times, denoted by $P_1,\ldots,P_r$, and then choosing the one that minimizes the worst-case error:
\[ P^* := \arg\min_{1\le j\le r}e(\mathcal{W}_{s, \alpha, p, \bm{\gamma}},P_j). \]
One can use $P^*$ as a deterministic QMC rule, or apply a random digital shift to construct an unbiased estimator. Unbiasedness allows for straightforward error estimation through the estimator's variance, which is a property that the median-of-means estimator lacks.

Regarding this greedy optimization of Hankel designs, the worst-case error is bounded probabilistically as follows:

\begin{Theorem}
    For $\alpha\in \natu$, $p\in (0,1]$, and $\bm{\gamma}=(\gamma_{\mathfrak{u}})_{\mathfrak{u}\subseteq 1{:}s}$ be a set of positive weights, let $\mathcal{W}_{s, \alpha, p, \bm{\gamma}}$ be the weighted Sobolev--variation space as defined in Definition~\ref{def:weighted_sobolev_variation}. For $r\in \natu$, let $P_1,\ldots,P_r$ be i.i.d.\ copies of HRD with infinite precision $E=\infty$ and let $P^*$ be the one which minimizes the worst-case error in $\mathcal{W}_{s, \alpha, p, \bm{\gamma}}$. Then, for any $\delta \in (0,1)$, we have
    \[ \PP\left[ e(\mathcal{W}_{s, \alpha, p, \bm{\gamma}},P^*)\le \inf_{1/(\alpha+p)<\lambda\le 1}\left(\frac{1}{(1-\delta)b^m}\sum_{\emptyset \neq \mathfrak{u}\subseteq 1{:}s}\gamma_{\mathfrak{u}}^{\lambda} C_{\alpha,p}^{\lambda |\mathfrak{u}|}D_{\alpha,p,\lambda}^{|\mathfrak{u}|}\right)^{1/\lambda}\right]\ge 1-(1-\delta)^r, \]
    where $C_{\alpha,p}$ is as given in \eqref{eq:constant_C_alpha_p}, and $D_{\alpha,p,\lambda}>0$ is a constant depending only on $\alpha,p,\lambda$.
\end{Theorem}

\begin{proof}
For any function $f\in \mathcal{W}_{s, \alpha, p, \bm{\gamma}}$ and digital net $P\subset [0,1)^s$ in base $b$ with $N=b^m$ points (without any digital shift applied), the integration error is bounded as
\begin{align*}
    |I(f)-Q_N(f)| & = \left|\sum_{\bm{k} \in \mathcal{D}^\perp \setminus \{\bm{0} \} }\hat{f}
( \bm{k} ) \right| \le \sum_{\bm{k} \in \mathcal{D}^\perp \setminus \{\bm{0} \} }|\hat{f}
( \bm{k} ) | \\
    & \le \norm{f}_{{s, \alpha, p, \bm{\gamma}}} \sum_{\bm{k} \in \mathcal{D}^\perp \setminus \{\bm{0} \} }\gamma_{\supp(\bm{k})} C_{\alpha,p}^{|\supp(\bm{k})|} \, b^{-(\alpha+p) \mu_{\alpha+1}(\bm{k})/(\alpha+1)},
\end{align*}
where we used \eqref{eq:walsh_bound_each_function}. Therefore, the worst-case error in $\mathcal{W}_{s, \alpha, p, \bm{\gamma}}$ is bounded as
\begin{align}
    e(\mathcal{W}_{s, \alpha, p, \bm{\gamma}},P) \le \sum_{\bm{k} \in \mathcal{D}^\perp \setminus \{\bm{0} \} }\gamma_{\supp(\bm{k})} C_{\alpha,p}^{|\supp(\bm{k})|} \, b^{-(\alpha+p) \mu_{\alpha+1}(\bm{k})/(\alpha+1)}.\label{eq:worst-case_error_bound}
\end{align} 

Now, let $P$ be a randomly chosen Hankel design. By using the inequality $(\sum_i a_i)^{\lambda}\le \sum_i a_i^{\lambda}$ that holds for $(a_i)_i$ with $a_i\ge 0$ and $\lambda\in (0,1]$, and by applying Lemma~\ref{lem:prob_each_k}, we have
\begin{align*}
    \EE[e^{\lambda}(\mathcal{W}_{s, \alpha, p, \bm{\gamma}},P)] & \le \EE\left[ \sum_{\bm{k} \in \mathcal{D}^\perp \setminus \{\bm{0} \} }\gamma_{\supp(\bm{k})}^{\lambda} C_{\alpha,p}^{\lambda |\supp(\bm{k})|} \, b^{-\lambda (\alpha+p) \mu_{\alpha+1}(\bm{k})/(\alpha+1)}\right] \\
    & = \sum_{\bsk \in \natu^{s}_*}\gamma_{\supp(\bm{k})}^{\lambda} C_{\alpha,p}^{\lambda |\supp(\bm{k})|} \, b^{-\lambda (\alpha+p) \mu_{\alpha+1}(\bm{k})/(\alpha+1)}\PP\left[ \bsk\in \mathcal{D}^{\perp}\right]\\
    & = \frac{1}{b^m}\sum_{\bsk \in \natu^{s}_*}\gamma_{\supp(\bm{k})}^{\lambda} C_{\alpha,p}^{\lambda |\supp(\bm{k})|} \, b^{-\lambda (\alpha+p) \mu_{\alpha+1}(\bm{k})/(\alpha+1)}\\
    & = \frac{1}{b^m}\sum_{\emptyset \neq \mathfrak{u}\subseteq 1{:}s}\gamma_{\mathfrak{u}}^{\lambda} C_{\alpha,p}^{\lambda |\mathfrak{u}|}\sum_{\bsk_{\mathfrak{u}} \in \natu^{|\mathfrak{u}|}}  b^{-\lambda (\alpha+p) \mu_{\alpha+1}(\bsk_{\mathfrak{u}})/(\alpha+1)}\\
    & = \frac{1}{b^m}\sum_{\emptyset \neq \mathfrak{u}\subseteq 1{:}s}\gamma_{\mathfrak{u}}^{\lambda} C_{\alpha,p}^{\lambda |\mathfrak{u}|}D_{\alpha,p,\lambda}^{|\mathfrak{u}|},
\end{align*}
for $\lambda\in (1/(\alpha+p),1]$, where we write 
\begin{align*}
    D_{\alpha,p,\lambda} & = \sum_{k=1}^{\infty}b^{-\lambda (\alpha+p) \mu_{\alpha+1}(k)/(\alpha+1)}\\
    & = \sum_{\nu=1}^{\alpha}\prod_{i=1}^{\nu}\frac{b-1}{b^{\lambda i(\alpha+p)/(\alpha+1)}-1}+\frac{b^{\lambda (\alpha+p)}-1}{b^{\lambda (\alpha+p)}-b}\prod_{i=1}^{\alpha+1}\frac{b-1}{b^{\lambda i(\alpha+p)/(\alpha+1)}-1} <\infty.
\end{align*} 
We refer to \cite[Lemma~7]{God16} for the second equality.
Since the above bound on the expectation holds for any $\lambda\in (1/(\alpha+p),1]$, the Markov inequality states that
\[ \PP\left[ e(\mathcal{W}_{s, \alpha, p, \bm{\gamma}},P)\le \inf_{1/(\alpha+p)<\lambda\le 1}\left(\frac{1}{(1-\delta)b^m}\sum_{\emptyset \neq \mathfrak{u}\subseteq 1{:}s}\gamma_{\mathfrak{u}}^{\lambda} C_{\alpha,p}^{\lambda |\mathfrak{u}|}D_{\alpha,p,\lambda}^{|\mathfrak{u}|}\right)^{1/\lambda}\right] \ge \delta, \]
for any $0<\delta<1$. That is, when drawing a Hankel design randomly, the worst-case error is bounded by a quantity arbitrarily close to $\mathcal{O}(N^{-(\alpha+p)})$ with probability at least $\delta$.

Then we have
\begin{align*}
    & \PP\left[ e(\mathcal{W}_{s, \alpha, p, \bm{\gamma}},P^*)\le \inf_{1/(\alpha+p)<\lambda\le 1}\left(\frac{1}{(1-\delta)b^m}\sum_{\emptyset \neq \mathfrak{u}\subseteq 1{:}s}\gamma_{\mathfrak{u}}^{\lambda} C_{\alpha,p}^{\lambda |\mathfrak{u}|}D_{\alpha,p,\lambda}^{|\mathfrak{u}|}\right)^{1/\lambda}\right] \\
    & = 1-\PP\left[ e(\mathcal{W}_{s, \alpha, p, \bm{\gamma}},P^*)> \inf_{1/(\alpha+p)<\lambda\le 1}\left(\frac{1}{(1-\delta)b^m}\sum_{\emptyset \neq \mathfrak{u}\subseteq 1{:}s}\gamma_{\mathfrak{u}}^{\lambda} C_{\alpha,p}^{\lambda |\mathfrak{u}|}D_{\alpha,p,\lambda}^{|\mathfrak{u}|}\right)^{1/\lambda}\right] \\
    & = 1-\prod_{j=1}^{r}\PP\left[ e(\mathcal{W}_{s, \alpha, p, \bm{\gamma}},P_r)> \inf_{1/(\alpha+p)<\lambda\le 1}\left(\frac{1}{(1-\delta)b^m}\sum_{\emptyset \neq \mathfrak{u}\subseteq 1{:}s}\gamma_{\mathfrak{u}}^{\lambda} C_{\alpha,p}^{\lambda |\mathfrak{u}|}D_{\alpha,p,\lambda}^{|\mathfrak{u}|}\right)^{1/\lambda}\right]\\
    & \ge 1-(1-\delta)^r,
\end{align*} 
for any $0<\delta<1$. This completes the proof.
\end{proof}

This result implies that a larger $r$ naturally increases the probability of success. However, a very small $r$ is often sufficient in practice. Choosing $r = \lceil \log(1/\epsilon) / \log(1/(1-\delta)) \rceil$ ensures a success probability of at least $1-\epsilon$. For instance, with $\delta=1/2$, $r=15$ is enough to guarantee a success probability exceeding $0.9999$ (specifically, $1-2^{-15} \approx 0.99997$).

\begin{remark}
    Although the worst-case error is only bounded probabilistically, the bound can be independent of the dimension $s$, if there exists a constant $\lambda\in (1/(\alpha+p),1]$ such that
    \[ \sum_{\substack{\emptyset \neq \mathfrak{u}\subseteq \natu\\ |\mathfrak{u}|<\infty}}\gamma_{\mathfrak{u}}^{\lambda} C_{\alpha,p}^{\lambda |\mathfrak{u}|}D_{\alpha,p,\lambda}^{|\mathfrak{u}|}<\infty. \]
    In the case of product weights, this condition can be simplified into
    \[ \sum_{j=1}^{\infty}\gamma_j^{\lambda}<\infty,\]
    which is always satisfied if the following holds:
    \[ \sum_{j=1}^{\infty}\gamma_j^{1/(\alpha+p)}<\infty.\]
\end{remark}

\begin{remark}
To make our greedy optimization feasible, we need a computable formula for either the worst-case error or its upper bound, as shown in \eqref{eq:worst-case_error_bound}. Let $P=\{\bsx_0,\ldots,\bsx_{b^m-1}\}$ be a fixed digital net in base $b$. In the case of product weights and $p=1$, the upper bound equals
\begin{align*}
    \sum_{\bm{k} \in \mathcal{D}^\perp \setminus \{\bm{0} \} }\gamma_{\supp(\bm{k})} C_{\alpha,1}^{|\supp(\bm{k})|} \, b^{-\mu_{\alpha+1}(\bm{k})} & = \sum_{\bm{k} \in \natu_*^s}\gamma_{\supp(\bm{k})} C_{\alpha,p}^{|\supp(\bm{k})|} \, b^{- \mu_{\alpha+1}(\bm{k})}\, \frac{1}{b^m}\sum_{i=0}^{b^m-1}\bwal{\bm{k}}(\bm{x}_i)\\
    & = -1+\frac{1}{b^m}\sum_{i=0}^{b^m-1}\prod_{j=1}^{s}\left[ 1+\gamma_j C_{\alpha,1}\omega_{\alpha+1}(x_{i,j}) \right],
\end{align*}
where
\[ \omega_{\alpha+1}(x) \coloneqq \sum_{k=1}^{\infty} b^{- \mu_{\alpha+1}(k)} \bwal{k_j}(x). \]
Baldeaux et al.~\cite{BDLNP12} studied how to compute $\omega_{\alpha+1}$ efficiently and also derived an explicit form in special cases. For instance, in the case of $b=2$, \cite[Corollary~1]{BDLNP12} shows that
\begin{align*}
    \omega_2(x) =s_1(x)+\tilde{s}_2(x),\quad \omega_3(x) =s_1(x)+s_2(x)+\tilde{s}_3(x),
\end{align*} 
where
\begin{alignat*}{2} 
    & s_1(x)=1-2x,& \quad & s_2(x)=1/3-2(1-x)x, \\
    & \tilde{s}_2(x)=(1-5t_1)/2-(a_1-2)x, & \quad & \tilde{s}_3(x)=(1-43t_1^2)/18+(5t_1-1)x+(a_1-2)x^2,
\end{alignat*}
with $a_1=t_1=0$ when $x=0$, and $a_1\coloneqq -\lfloor \log_2(x)\rfloor, t_1\coloneqq 2^{-a_1}$ for $0<x<1$. Thus, if $\alpha+1\in \{2,3\}$, the worst-case error bound can be computed with $\mathcal{O}(sN)$ costs.
\end{remark}

\section{Numerical Results}
\label{sec:numerical_results}

This section presents numerical results. We first report results for the median-of-means estimator based on HRD, followed by those obtained using the greedy optimization of HRD.

\subsection{Median-of-means}
Here we present numerical results from the median-of-means estimator studied in Subsection~\ref{sec:median-of-means-estimator}. 
In this subsection and the next, we consider the integrand introduced in~\cite{GSM24}:
\begin{equation}
    \label{eq:numerical_example}
    f(\bm{t}) = \prod_{j=1}^s \left[1 + \gamma_j \left( t_j^c - \frac{1}{1+c} \right)\right],
\end{equation}
and consider the exponential decay weights $\gamma_j = \exp(-\lceil c \rceil j)$. In the following, we fix $s = 50$ and present the numerical results for $c = 1.5$ and $c = 2.5$. It is straightforward to check that in both cases, $f \in \mathcal{W}_{s, \ceil{c}, \bm{\gamma}, 1/2}$.

Figure~\ref{fig:m24_s50_c15_c25} shows the squared error of the median-of-means estimator for the integrand~\eqref{eq:numerical_example}, where the median is taken over 15 independent replicates. 
Since the effective dimension is low due to the rapid decay of the weights, the linearly scrambled Sobol' sequence performs best among the three methods. Moreover, HRD achieves substantially smaller errors than URD, which is consistent with its smaller probability bound for certain events that dominate the integration error. All three methods exhibit the theoretically predicted high-probability convergence rates. 

\begin{figure}[htbp]
    \centering
    \includegraphics[width=0.485\linewidth]{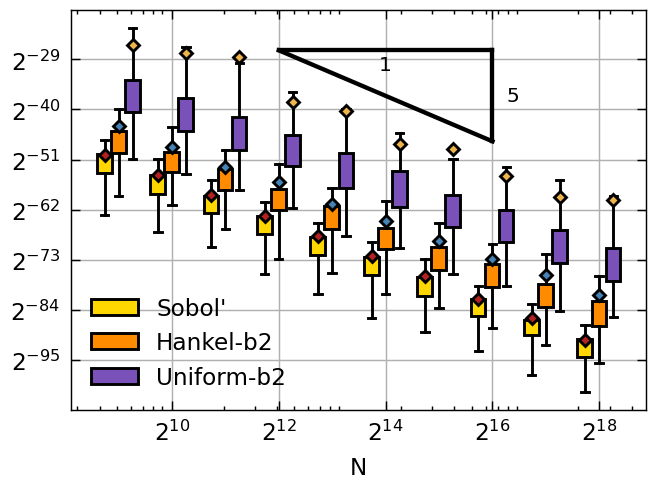}
    \includegraphics[width=0.485\linewidth]{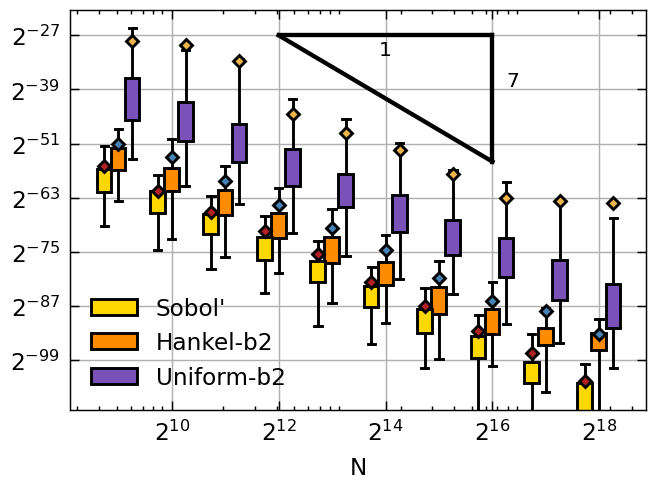}
    \caption{Squared error of the median estimator based on 15 randomized quasi-Monte Carlo (RQMC) samples per replicate. Each boxplot is constructed from 448 samples. Left: $c = 1.5$; right: $c = 2.5$. }
    \label{fig:m24_s50_c15_c25}
\end{figure}

Next, we provide the numerical results for different values of the base $b$ in the HRD in Figure~\ref{fig:m24_s50_c15_c25_b2357}, where the median estimator is taken over 15 independent replicates.  
\begin{figure}[htbp]
    \centering
    \includegraphics[width=0.485\linewidth]{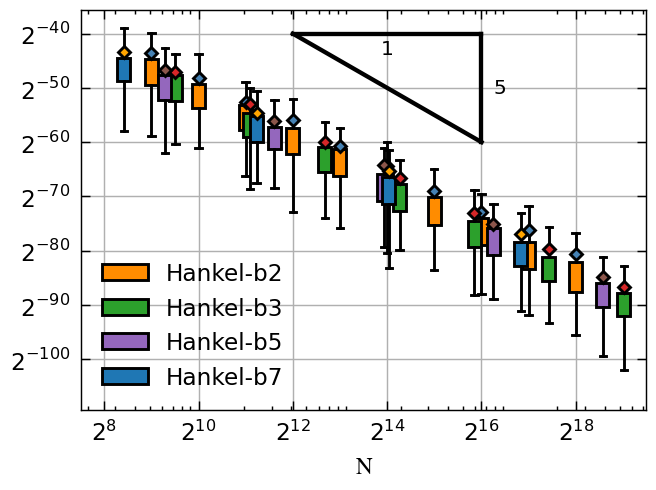}
    \includegraphics[width=0.485\linewidth]{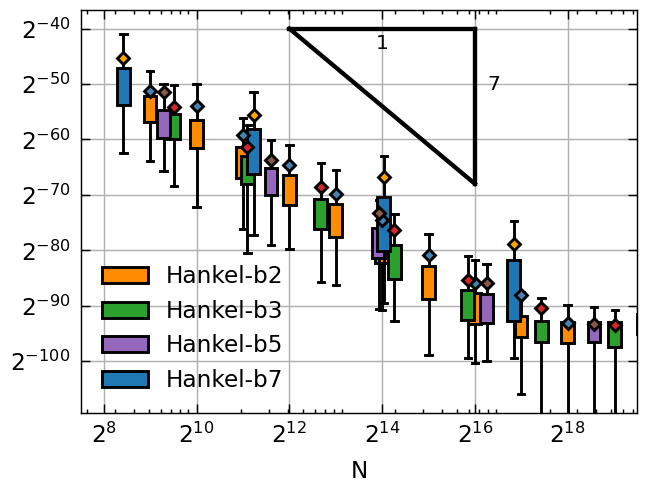}
    \caption{Squared error of the median estimator based on 15 randomized quasi-Monte Carlo (RQMC) samples per replicate. Each boxplot is constructed from 448 samples. Left: $c = 1.5$; right: $c = 2.5$.  }
    \label{fig:m24_s50_c15_c25_b2357}
\end{figure}
The errors of the median-of-means estimators corresponding to all tested values of $b$ exhibit the same convergence rate, and their decay curves largely overlap. This may suggest that the choice of $b$ has only a minor effect on the performance of HRD in this example. 

To further assess the robustness of HRD, we present additional numerical results for isotropic integrands in Appendix~\ref{sec:mom_numex}.

\subsection{Greedy optimization of Hankel designs}

In this subsection, we present the numerical results of the greedy optimization approach proposed in Subsection~\ref{sec:optimization_hankel}. The integrand is the same as that considered in the previous section, namely~\eqref{eq:numerical_example} with $s = 50$. 

First, we compare the computable worst-case error bound between HRD and URD. 
\begin{figure}[htbp] 
    \centering
    \begin{subfigure}[b]{0.48\textwidth} 
        \includegraphics[width=\textwidth]{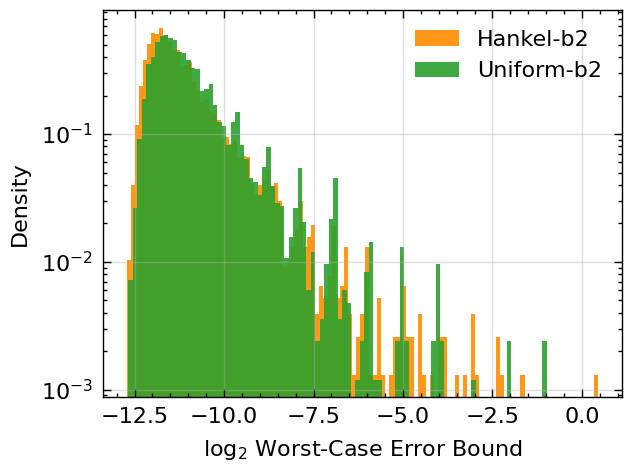}
        \caption{$r = 1, \alpha = 2$}
        \label{fig:wce_r1_a2}
    \end{subfigure}
    \hfill 
    \begin{subfigure}[b]{0.48\textwidth}
        \includegraphics[width=\textwidth]{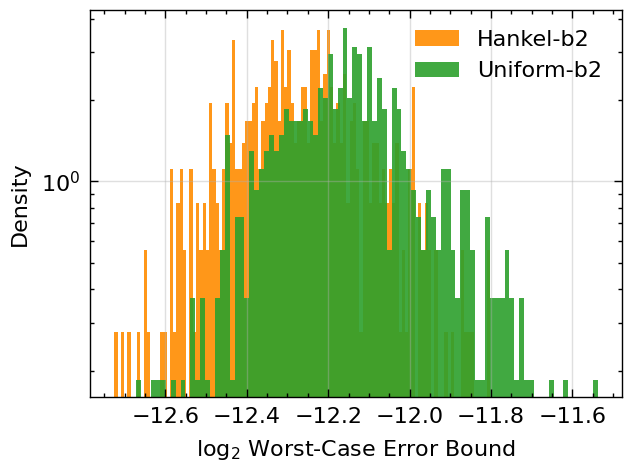}
        \caption{$r = 15, \alpha = 2$}
        \label{fig:wce_r15_a2}
    \end{subfigure}
    
    \vspace{0.5em} 
    
    \begin{subfigure}[b]{0.48\textwidth}
        \includegraphics[width=\textwidth]{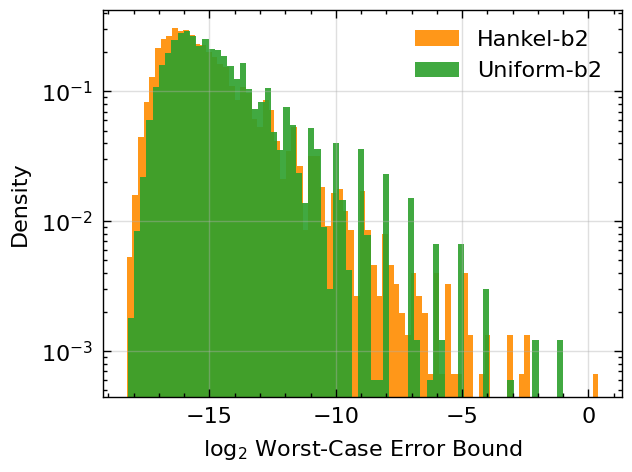}
        \caption{$r = 1, \alpha = 3$}
        \label{fig:wce_r1_a3}
    \end{subfigure}
    \hfill
    \begin{subfigure}[b]{0.48\textwidth}
        \includegraphics[width=\textwidth]{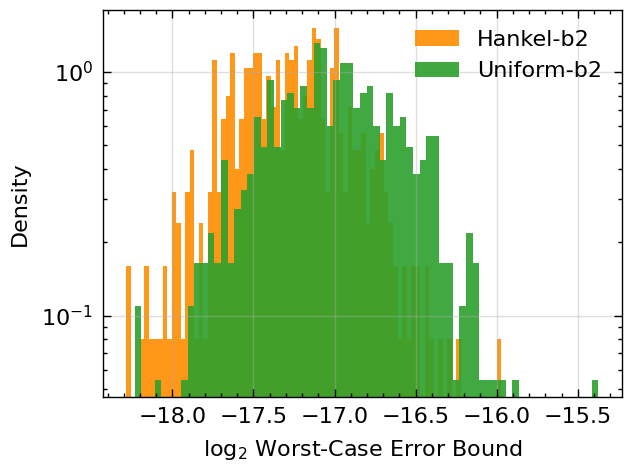}
        \caption{$r = 15, \alpha = 3$}
        \label{fig:wce_r15_a3}
    \end{subfigure}
    
    \caption{Histogram of the computable worst-case error bound for $s = 50$. The upper row corresponds to $\alpha = 2$, and the lower row corresponds to $\alpha  = 3$. The left column ($r = 1$) plots the bounds for 6720 individual samples, and the right column ($r = 15$) plots the smallest bound in each batch of 15 samples, based on 448 batches. }
    \label{fig:wce_s50}
\end{figure}
Figure~\ref{fig:wce_s50} compares the computable worst-case error bounds between HRD and URD in the samplewise case ($r = 1$) and when selecting the best design from a batch of size $15$ ($r = 15)$, with $s = 50$. We use the computable  upper bound from~\cite{BDLNP12}. In the samplewise comparison, HRD has marginally smaller error bounds than URD. The advantage becomes more pronounced when selecting the best design from a batch, as observed in the right column.

In the following, we present the convergence results for the greedy optimization of HRD and compare them with those obtained by LMS applied to the Sobol' sequence. We present numerical results for $c = 1.5$ and $c = 2.5$, and for two types of weights $\bm{\gamma}$: the exponential decay weight $\gamma_j = \frac{1}{\exp(\ceil{c} j)}$ and the equal weight $\gamma_j = 1.0$.

\begin{figure}[htbp] 
    \centering
    \begin{subfigure}[b]{0.48\textwidth} 
        \includegraphics[width=\textwidth]{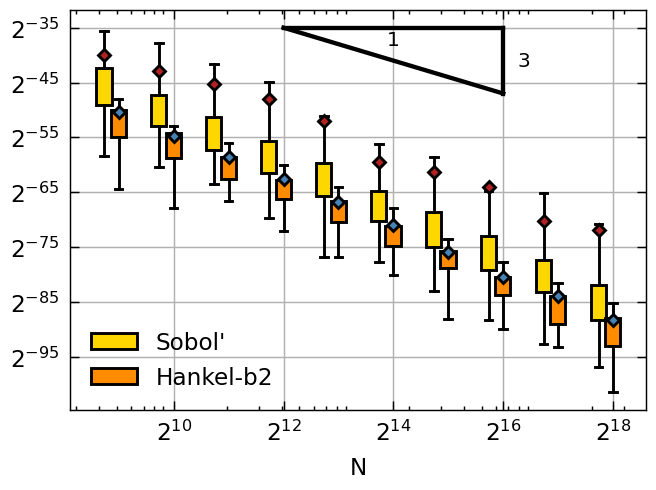}
        \caption{Mean, exponential weight}
        \label{fig:m18_s50_wce_mean_Hankel}
    \end{subfigure}
    \hfill 
    \begin{subfigure}[b]{0.48\textwidth}
        \includegraphics[width=\textwidth]{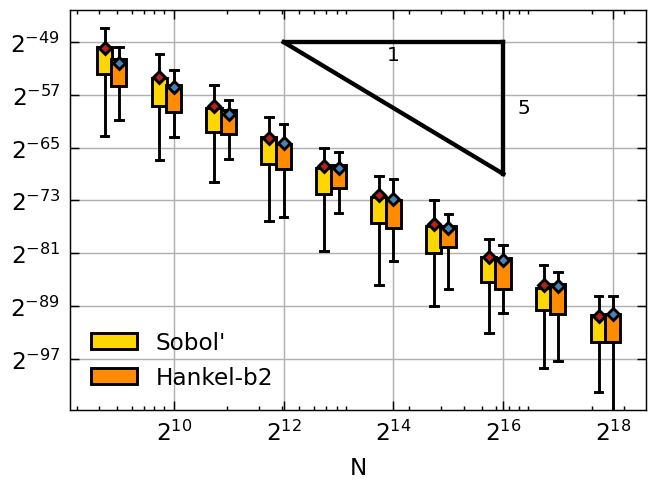}
        \caption{Median, exponential weight}
        \label{fig:m18_s50_wce_median_Hankel}
    \end{subfigure}
    \begin{subfigure}[b]{0.48\textwidth} 
        \includegraphics[width=\textwidth]{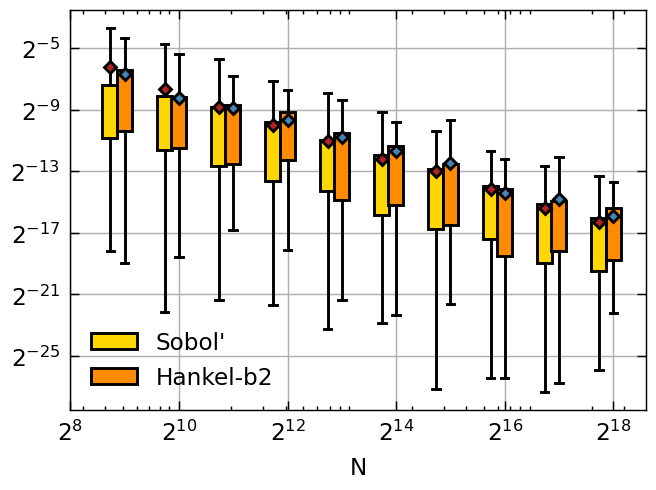}
        \caption{Mean, equal weight}
        \label{fig:m18_s50_wce_mean_Hankel}
    \end{subfigure}
    \hfill 
    \begin{subfigure}[b]{0.48\textwidth}
        \includegraphics[width=\textwidth]{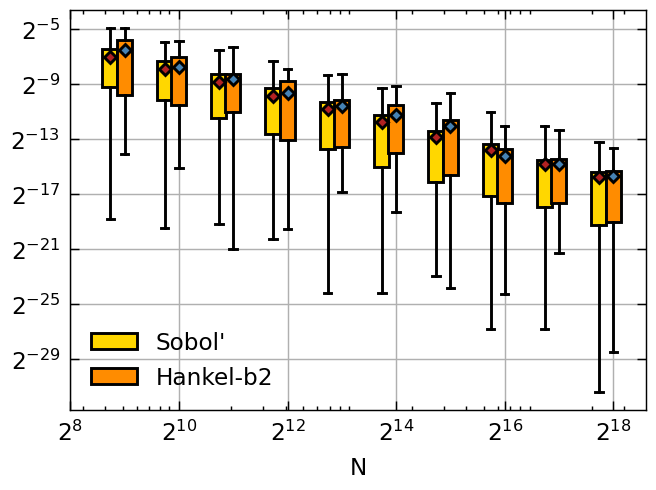}
        \caption{Median, equal weight}
        \label{fig:m18_s50_wce_median_Hankel}
    \end{subfigure}
    \vspace{0.5em} 
    \caption{Comparison between LMS+Sobol' sequence and the optimized HRD for $c = 1.5$ with exponential weight. Results are shown for both the mean and median estimators, with $r = \lceil m \log{m} \rceil$. Each boxplot is constructed from 448 samples. }
    \label{fig:Hankel_sobol_c15}
\end{figure}

We plot the squared errors of the mean and median estimators in Figure~\ref{fig:Hankel_sobol_c15}. Both the mean and median estimators are taken from a minibatch of size $r = \lceil m\log m \rceil$, as suggested in Remark~\ref{rem:choice_r}. We observe that, for the mean estimator, the greedy optimized Hankel design in base 2 significantly outperforms LMS+Sobol', yielding a reduction in mean squared error of more than a factor of $2^{10}$. For the median estimator, both methods exhibit nearly optimal convergence rates, with the optimized Hankel design remaining slightly superior. In the equal weight case, the effective dimension is high, and both the mean and median estimators do not reveal the asymptotic rate, with the LMS+Sobol' approach slightly outperforming.

\begin{figure}[htbp] 
    \centering
    \begin{subfigure}[b]{0.48\textwidth} 
        \includegraphics[width=\textwidth]{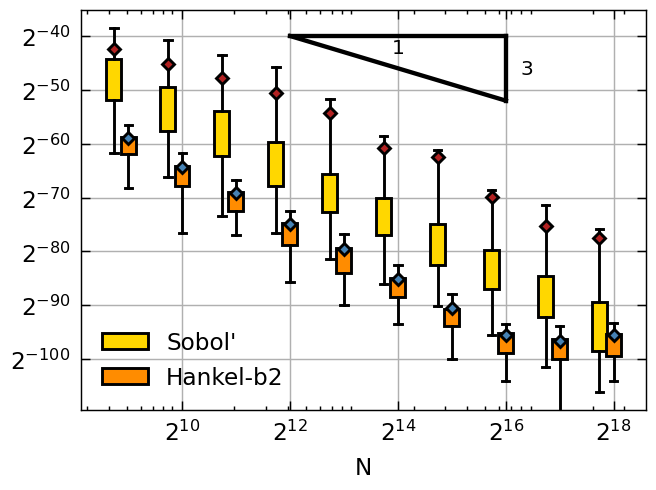}
        \caption{Mean, exponential weight}
        \label{fig:m18_s50_wce_mean_Hankel}
    \end{subfigure}
    \hfill 
    \begin{subfigure}[b]{0.48\textwidth}
        \includegraphics[width=\textwidth]{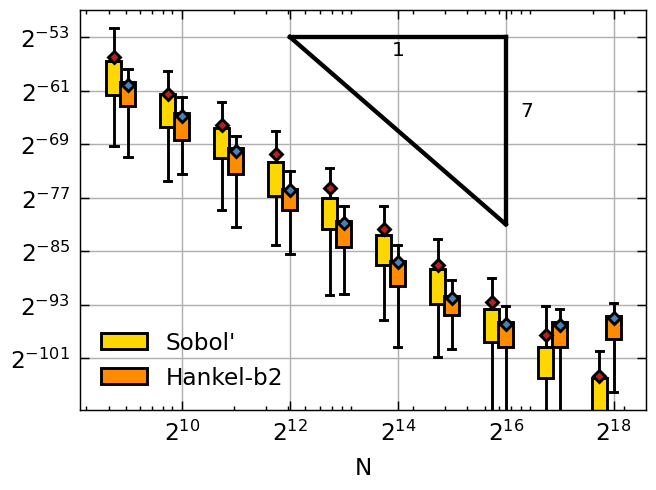}
        \caption{Median, exponential weight}
        \label{fig:m18_s50_wce_median_Hankel}
    \end{subfigure}
    \begin{subfigure}[b]{0.48\textwidth} 
        \includegraphics[width=\textwidth]{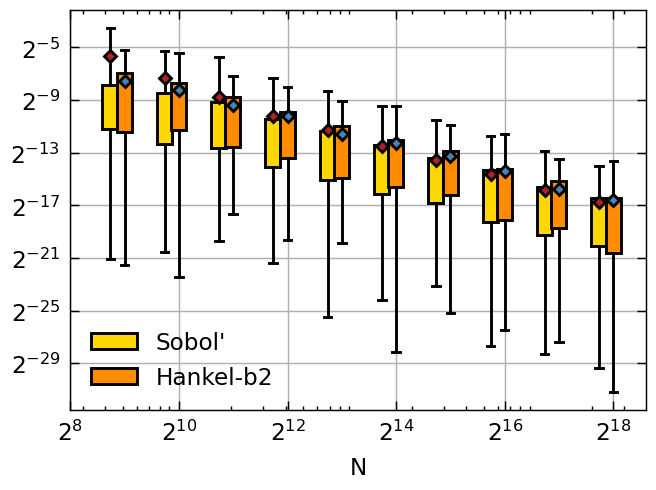}
        \caption{Mean, equal weight}
        \label{fig:m18_s50_wce_mean_Hankel}
    \end{subfigure}
    \hfill 
    \begin{subfigure}[b]{0.48\textwidth}
        \includegraphics[width=\textwidth]{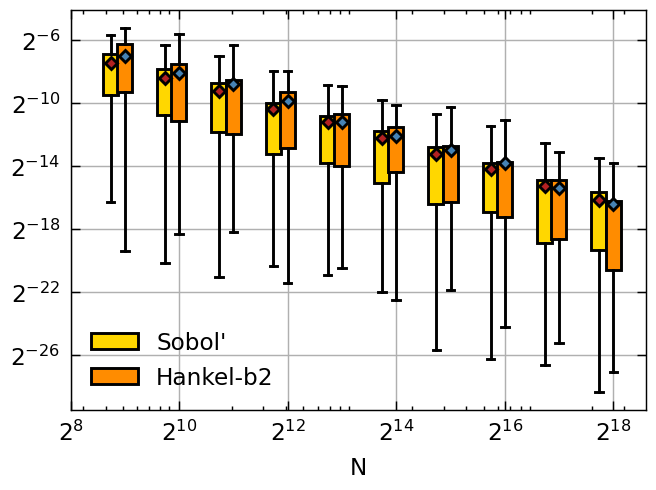}
        \caption{Median, equal weight}
        \label{fig:m18_s50_wce_median_Hankel}
    \end{subfigure}
    \vspace{0.5em} 
    \caption{Comparison between LMS+Sobol' sequence and the optimized HRD for $c = 2.5$ with exponential weight. Results are shown for both the mean and median estimators, with $r = \lceil m \log{m} \rceil$. Each boxplot is constructed from 448 samples.}
    \label{fig:Hankel_sobol_c25}
\end{figure}
Figure~\ref{fig:Hankel_sobol_c25} indicates that, in the case of exponential decay weight, as the function smoothness increases, the advantage of the optimized Hankel design over LMS+Sobol' becomes more pronounced. The gain in the mean squared error is around $2^{20}$. The optimized Hankel design also continues to outperform LMS+Sobol' in the median estimator setting. In the case of equal weight, however, both methods have similar performance. 

\section{Discussion}\label{sec:disc}

In this work, we propose a Hankel random digital net for QMC integration. The construction is not constrained by the maximal dimension and is flexible with respect to the choice of the prime base. We derive the decay of Walsh coefficients for functions in a weighted Sobolev space and analyze the convergence rate for both the median-of-means estimator and a greedy optimization procedure. The proposed HRD is shown to outperform previous randomized designs such as the URD, and the numerical results demonstrate improved performance compared with LMS combined with the Sobol' sequence. Meanwhile, in revisiting the properties of LMS, we also establish the variance equivalence between LMS and Owen-scrambling. 

Note that the cost of evaluating the worst-case error bound is $\mathcal{O}(b^m)$, which is the same as the cost of evaluating the integrands when sampling costs $\mathcal{O}(1)$. However, the proposed greedy optimization method is expected to be particularly efficient in more computationally intensive applications, such as when sampling involves solving differential equations, where the cost of optimization is negligible relative to the overall simulation cost. 

Finally, our numerical results indicate that the error convergence behavior is similar across different choices of base $b$. An interesting open question concerns the best choice of the base $b$, both from theoretical and computational perspectives.  



\bibliographystyle{plain}
\bibliography{ref}

\begin{thebibliography}{10}

\bibitem{BDLNP12}
Jan Baldeaux, Josef Dick, Gunther Leobacher, Dirk Nuyens, and Friedrich Pillichshammer.
\newblock Efficient calculation of the worst-case error and (fast) component-by-component construction of higher order polynomial lattice rules.
\newblock {\em Numer. Algorithms}, 59(3):403--431, 2012.

\bibitem{CE52}
K.~L. Chung and P.~Erd\"os.
\newblock On the application of the {B}orel-{C}antelli lemma.
\newblock {\em Trans. Amer. Math. Soc.}, 72:179--186, 1952.

\bibitem{Dic08}
Josef Dick.
\newblock Walsh spaces containing smooth functions and quasi-{M}onte {C}arlo rules of arbitrary high order.
\newblock {\em SIAM J. Numer. Anal.}, 46(3):1519--1553, 2008.

\bibitem{Dic11}
Josef Dick.
\newblock Higher order scrambled digital nets achieve the optimal rate of the root mean square error for smooth integrands.
\newblock {\em Ann. Statist.}, 39(3):1372--1398, 2011.

\bibitem{DKLNS14}
Josef Dick, Frances~Y. Kuo, Quoc~T. Le~Gia, Dirk Nuyens, and Christoph Schwab.
\newblock Higher order {QMC} {P}etrov-{G}alerkin discretization for affine parametric operator equations with random field inputs.
\newblock {\em SIAM J. Numer. Anal.}, 52(6):2676--2702, 2014.

\bibitem{DP07}
Josef Dick and Friedrich Pillichshammer.
\newblock Strong tractability of multivariate integration of arbitrary high order using digitally shifted polynomial lattice rules.
\newblock {\em J. Complexity}, 23(4-6):436--453, 2007.

\bibitem{DP10}
Josef Dick and Friedrich Pillichshammer.
\newblock {\em Digital Nets and Sequences--Discrepancy Theory and Quasi-{M}onte {C}arlo Integration}.
\newblock Cambridge University Press, Cambridge, 2010.

\bibitem{God15}
Takashi Goda.
\newblock Good interlaced polynomial lattice rules for numerical integration in weighted {W}alsh spaces.
\newblock {\em J. Comput. Appl. Math.}, 285:279--294, 2015.

\bibitem{God16}
Takashi Goda.
\newblock Quasi--{M}onte {C}arlo integration using digital nets with antithetics.
\newblock {\em J. Comput. Appl. Math.}, 304:26--42, 2016.

\bibitem{GD15}
Takashi Goda and Josef Dick.
\newblock Construction of interlaced scrambled polynomial lattice rules of arbitrary high order.
\newblock {\em Found. Comput. Math.}, 15(5):1245--1278, 2015.

\bibitem{GK26}
Takashi Goda and David Krieg.
\newblock A simple universal algorithm for high-dimensional integration.
\newblock {\em Numer. Math.}, 158(1):229--248, 2026.

\bibitem{GL22}
Takashi Goda and Pierre L'Ecuyer.
\newblock Construction-free median quasi--{M}onte {C}arlo rules for function spaces with unspecified smoothness and general weights.
\newblock {\em SIAM J. Sci. Comput.}, 44(4):A2765--A2788, 2022.

\bibitem{GS23}
Takashi Goda and Kosuke Suzuki.
\newblock Improved bounds on the gain coefficients for digital nets in prime power base.
\newblock {\em J. Complexity}, 76:Paper No. 101722, 15, 2023.

\bibitem{GSM24}
Takashi Goda, Kosuke Suzuki, and Makoto Matsumoto.
\newblock A universal median quasi--{M}onte {C}arlo integration.
\newblock {\em SIAM J. Numer. Anal.}, 62(1):533--566, 2024.

\bibitem{GSY18}
Takashi Goda, Kosuke Suzuki, and Takehito Yoshiki.
\newblock Optimal order quadrature error bounds for infinite-dimensional higher-order digital sequences.
\newblock {\em Found. Comput. Math.}, 18(2):433--458, 2018.

\bibitem{Hel02}
Peter Hellekalek.
\newblock Digital {$(t,m,s)$}-nets and the spectral test.
\newblock {\em Acta Arith.}, 105(2):197--204, 2002.

\bibitem{HN03}
Fred~J. Hickernell and Harald Niederreiter.
\newblock The existence of good extensible rank-1 lattices.
\newblock {\em J. Complexity}, 19(3):286--300, 2003.

\bibitem{Hun76}
David Hunter.
\newblock An upper bound for the probability of a union.
\newblock {\em J. Appl. Probability}, 13(3):597--603, 1976.

\bibitem{JK08}
Stephen Joe and Frances~Y. Kuo.
\newblock Constructing {S}obol' sequences with better two-dimensional projections.
\newblock {\em SIAM J. Sci. Comput.}, 30(5):2635--2654, 2008.

\bibitem{Liu25}
Yang Liu.
\newblock Integrability of weak mixed first-order derivatives and convergence rates of scrambled digital nets.
\newblock {\em J. Complexity}, 89:Paper No. 101935, 8, 2025.

\bibitem{Mat98}
Ji\v{r}\'{i} Matou\v{s}ek.
\newblock On the {$L_2$}-discrepancy for anchored boxes.
\newblock {\em J. Complexity}, 14(4):527--556, 1998.

\bibitem{Nie86}
Harald Niederreiter.
\newblock Low-discrepancy point sets.
\newblock {\em Monatsh. Math.}, 102(2):155--167, 1986.

\bibitem{Nie87}
Harald Niederreiter.
\newblock Point sets and sequences with small discrepancy.
\newblock {\em Monatsh. Math.}, 104(4):273--337, 1987.

\bibitem{Nie92}
Harald Niederreiter.
\newblock {\em Random Number Generation and Quasi-{M}onte {C}arlo Methods}, volume~63 of {\em CBMS-NSF Regional Conference Series in Applied Mathematics}.
\newblock Society for Industrial and Applied Mathematics (SIAM), Philadelphia, PA, 1992.

\bibitem{Owe95}
Art~B. Owen.
\newblock Randomly permuted {$(t,m,s)$}-nets and {$(t,s)$}-sequences.
\newblock In {\em Monte {C}arlo and quasi-{M}onte {C}arlo methods in scientific computing ({L}as {V}egas, {NV}, 1994)}, volume 106 of {\em Lect. Notes Stat.}, pages 299--317. Springer, New York, 1995.

\bibitem{Owen97}
Art~B. Owen.
\newblock Monte {C}arlo variance of scrambled net quadrature.
\newblock {\em SIAM J. Numer. Anal.}, 34(5):1884--1910, 1997.

\bibitem{Pan25}
Zexin Pan.
\newblock Automatic optimal-rate convergence of randomized nets using median-of-means.
\newblock {\em Math. Comp.}, 95(359):1415--1446, 2026.

\bibitem{Pan26}
Zexin Pan.
\newblock Dimension-independent convergence rates of randomized nets using median-of-means.
\newblock {\em Math. Comp.}, 2026.

\bibitem{PO23b}
Zexin Pan and Art~B. Owen.
\newblock The nonzero gain coefficients of {S}obol's sequences are always powers of two.
\newblock {\em J. Complexity}, 75:Paper No. 101700, 16, 2023.

\bibitem{PO23}
Zexin Pan and Art~B. Owen.
\newblock Super-polynomial accuracy of one dimensional randomized nets using the median of means.
\newblock {\em Math. Comp.}, 92(340):805--837, 2023.

\bibitem{PO24}
Zexin Pan and Art~B. Owen.
\newblock Super-polynomial accuracy of multidimensional randomized nets using the median-of-means.
\newblock {\em Math. Comp.}, 93(349):2265--2289, 2024.

\bibitem{RT97}
M.~Yu. Rosenbloom and M.~A. Tsfasman.
\newblock Codes for the {$m$}-metric.
\newblock {\em Problemy Peredachi Informatsii}, 33(1):55--63, 1997.

\bibitem{Sob67}
I.~M. Sobol'.
\newblock Distribution of points in a cube and approximate evaluation of integrals.
\newblock {\em \v Z. Vy\v cisl. Mat i Mat. Fiz.}, 7:784--802, 1967.

\bibitem{SY16}
Kosuke Suzuki and Takehito Yoshiki.
\newblock Formulas for the {W}alsh coefficients of smooth functions and their application to bounds on the {W}alsh coefficients.
\newblock {\em J. Approx. Theory}, 205:1--24, 2016.

\end{thebibliography}

\appendix
\section{A fast implementation}
\label{sec:fast_implementation_digital_nets}
Here we describe the construction algorithm of the Hankel random digital net in 1-d. The case of multiple dimensions can be obtained by combining independent 1-d generators. The key idea is to use a Gray code to efficiently generate the sequence of points in a way that only one digit changes at each step. This allows us to update the point incrementally, which is computationally efficient. The algorithm is presented in Algorithm~\ref{alg:hrd_construction}.

\begin{algorithm}[H]
\label{alg:hrd_construction}
\DontPrintSemicolon
\SetKwFunction{GraySteps}{GrayStepsBase-$b$}
\SetKwFunction{HCols}{HankelColumnsMod-$b$}
\SetKwFunction{Gen}{GrayGen-$s$-Mod-$b$}

\caption{Construction of a 1-d Hankel random digital net}

\BlankLine
\textbf{Procedure} \GraySteps{$b,m$}:\;
\Indp
$a[0..m{-}1] \gets 0$,\quad $\mathrm{dir}[0..m{-}1] \gets +1$\;
\For{$n \gets 1$ \KwTo $b^m{-}1$}{
  \For{$t \gets 0$ \KwTo $m{-}1$}{
    \lIf{$0 \le a[t] + \mathrm{dir}[t] < b$}{
      $\mathrm{inc} \gets \mathrm{dir}[t]$;\quad
      $a[t] \gets a[t] + \mathrm{inc}$;\quad
      \For{$j \gets 0$ \KwTo $t{-}1$}{
        $\mathrm{dir}[j] \gets -\mathrm{dir}[j]$
      }
      \textbf{yield} $(t,\mathrm{inc})$;\quad \textbf{break}
    }
  }
}
\Indm
\BlankLine

\textbf{Procedure} \HCols{$b,m,\mathrm{RNG}$}:\;
\Indp
Draw $g[0..2m{-}2] \sim \mathrm{Unif}\{0,\dots,b{-}1\}$ using \(\mathrm{RNG}\)\;
\For{$t \gets 0$ \KwTo $m{-}1$}{
  \For{$i \gets 0$ \KwTo $m{-}1$}{
    $c_t[i] \gets g[i{+}t] \bmod b$
  }
}
\textbf{return} $\{c_t\}_{t=0}^{m-1}$\;
\Indm
\BlankLine

\textbf{Generator} \Gen{$b,m,\mathrm{hankel\_seed},\mathrm{shift\_seed}$ (optional)}:\;
\Indp
$C \gets$ \HCols{$b,m,\mathrm{RNG}(\mathrm{hankel\_seed})$};\quad
$w[i] \gets b^{-(i+1)}$ for $i = 0,\dots,m{-}1$\;
\lIf{\(\mathrm{shift\_seed}\) is \emph{None}}{$\delta \gets 0^m$}
\lElse{Draw $\delta[0..m{-}1] \sim \mathrm{Unif}\{0,\dots,b{-}1\}$ using \(\mathrm{RNG}(\mathrm{shift\_seed})\)}
$d \gets \delta$;\quad $s \gets \langle d, w\rangle$;\quad \textbf{yield} $s$ \tcp*{$s_0$}
\ForEach{$(t,\mathrm{inc})$ in \GraySteps{$b,m$}}{
  \lIf{$\mathrm{inc}=+1$}{$d' \gets (d + c_t) \bmod b$}
  \lElse{$d' \gets (d - c_t) \bmod b$}
  $s \gets s + \langle d' - d, w\rangle$;\quad $d \gets d'$;\quad \textbf{yield} $s$
}
\Indm
\end{algorithm}

We compare and summarize the storage and computational costs of HRD, URD and LMS for generating a $(t, m, s)$-net in base $b$ in Table~\ref{tab:comparison_storage_computational_cost}. 
\begin{table}[htbp]
    \centering
    \begin{tabular}{c | c c}
      Method & Storage  &  Cost \\ \hline
      HRD   & $\mathcal{O}(m+E - 1)$  & $\mathcal{O}(b^m)$ \\
      URD   & $\mathcal{O} (mE) $ & $\mathcal{O}(b^m)$ \\
      LMS+Sobol'   &  $\mathcal{O}(mE)$ & $\mathcal{O}(m^2 + b^m)$\\
    \end{tabular}
    \caption{The storage and cost of different randomization methods. The cost analysis is based on the Gray-code update approach.}
    \label{tab:comparison_storage_computational_cost}
\end{table}

\section{Proof of Lemma~\ref{lemma:W_linfty_bound}}
\label{sec:proof_W_linfty_bound}
\begin{proof}
    As we consider the case $k>0$, we have $\nu\ge 1$. From \cite[Propositions~3.5 and 3.7]{SY16}, it holds that
\begin{equation}
    \begin{aligned}
    \norm{W(k_{\alpha}^{+})}_{L^{\infty}} &= 2^{-\mu_{\alpha}(k_{\alpha}^{+})}, \quad b = 2,\\ 
    \norm{W(k_{\alpha}^{+})}_{L^{\infty}} &\leq \frac{b^{-\mu_{\alpha}(k_{\alpha}^{+})}}{m_b^{n_{\alpha}(k)}} \left( M_b + \frac{b m_b}{b - M_b} \left( 1 - \left(\frac{M_b}{b}\right)^{n_{\alpha}(k)} \right) \right), \quad b > 2,
    \end{aligned}
\end{equation}
with
\begin{equation}
    \begin{aligned}
        m_b &= 2 \sin(\pi / b),\\
        M_b &= \begin{cases}
            2 & \text{if $b$ is even},\\
            2\sin((b+1)\pi/2b) & \text{if $b$ is odd}.
        \end{cases}
    \end{aligned}
\end{equation}
We can find the following bounds for $m_b$ and $M_b$, when $b > 2$:
\begin{equation}
    \begin{aligned}
        4/b &\leq m_b \leq 2\pi/b,\\
        2(b-1)/b &\leq M_b \leq 2.
    \end{aligned}
\end{equation}
Using the bounds of $m_b$, $M_b$, we have when $b > 2$:
\begin{equation}
    \begin{split}
        \norm{W(k_{\alpha}^{+})}_{L^{\infty}} &\leq b^{-\mu_{\alpha}(k_{\alpha}^{+})} (4b^{-1})^{-n_{\alpha}(k)} (2 + 2\pi) \\
        &\leq \left(\frac{1 + \pi}{2} \right)  b^{n_{\alpha}(k)-\mu_{\alpha}(k_{\alpha}^{+})}.
    \end{split}
\end{equation}
This bound also holds when $b = 2$. Thus, the conclusion follows. 

\end{proof}

\section{Additional median-of-means numerical results}
\label{sec:mom_numex}

We present numerical results for two types of isotropic integrands to demonstrate the robustness of HRD. 
Our first example is $f(t) = \exp \left( \sum_{j=1}^s \Phi^{-1} (t_j) \right)$, where $\Phi^{-1}$ is the inverse cumulative distribution function of the normal distribution. This integrand is unbounded and has infinite Hardy--Krause variation. 

Figure~\ref{fig:m24_s2_25_50_lognorm} presents squared error of the median-of-means estimator for the integrand estimate of the first example, where the median is taken over 15 independent replicas. Due to the singularity of this example, the convergence rate do not exceed $\mathcal{O}(N^{-1})$. 
\begin{figure}[htbp]
    \centering
    \includegraphics[width=0.325\linewidth]{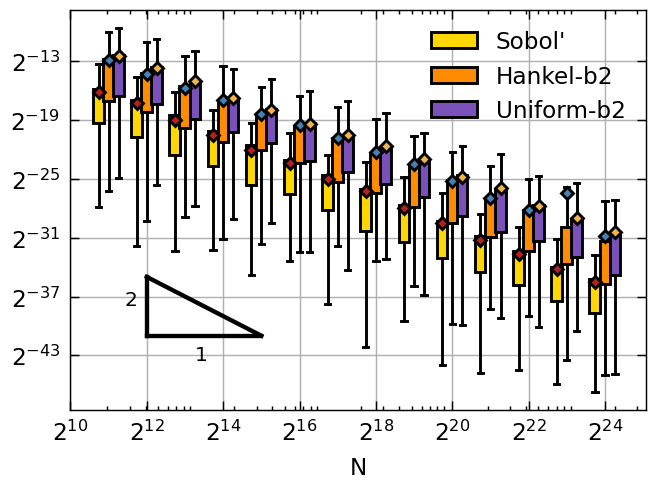}
    \includegraphics[width=0.325\linewidth]{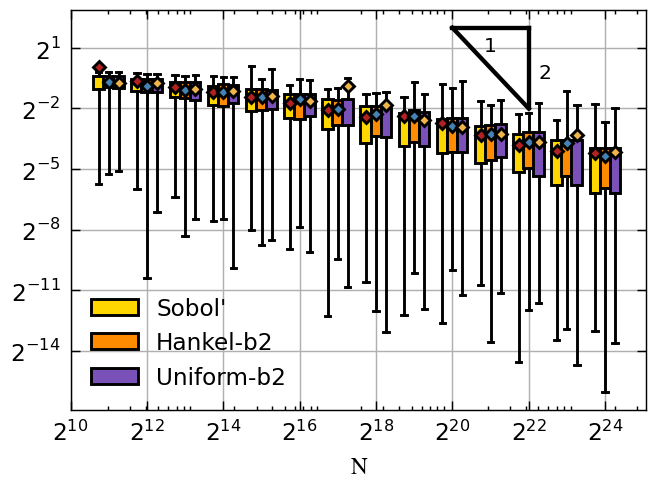}
    \includegraphics[width=0.325\linewidth]{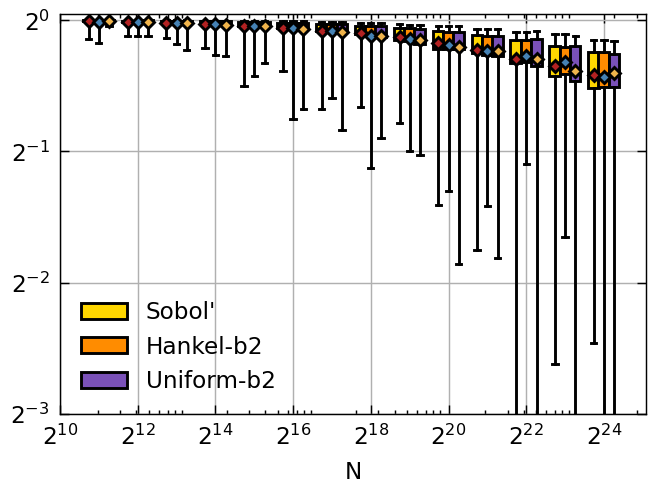}
    \caption{Squared error of the median estimator based on 15 randomized quasi-Monte Carlo (RQMC) samples. Each boxplot is constructed from 448 samples. From left to right, $s = 2, 25, 50$.}
    \label{fig:m24_s2_25_50_lognorm}
\end{figure}

The second example is $f(t) = \prod_{j=1}^s t_j \exp(t_j)$, which is in $C^{\infty}([0, 1)^s)$. In this example, the challenge arises in the variance $\var{f} = \left(\frac{e^2 - 1}{4} - 1\right)^s$, which grows exponentially with dimension $s$. Figure~\ref{fig:m24_s2_25_50_xexpx} presents the squared error of the median-of-means estimator for dimensions $s = 2, 25, 50$.

\begin{figure}[htbp]
    \centering
    \includegraphics[width=0.325\linewidth]{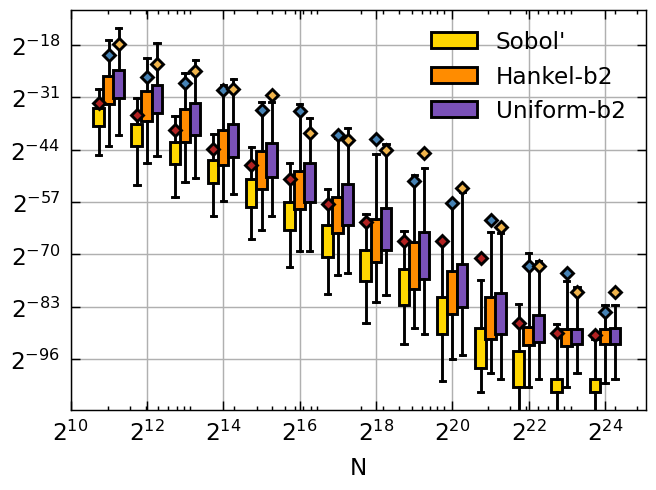}
    \includegraphics[width=0.325\linewidth]{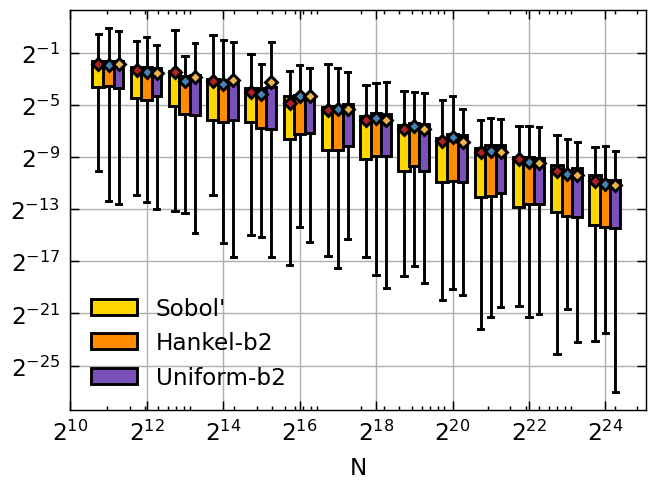}
    \includegraphics[width=0.325\linewidth]{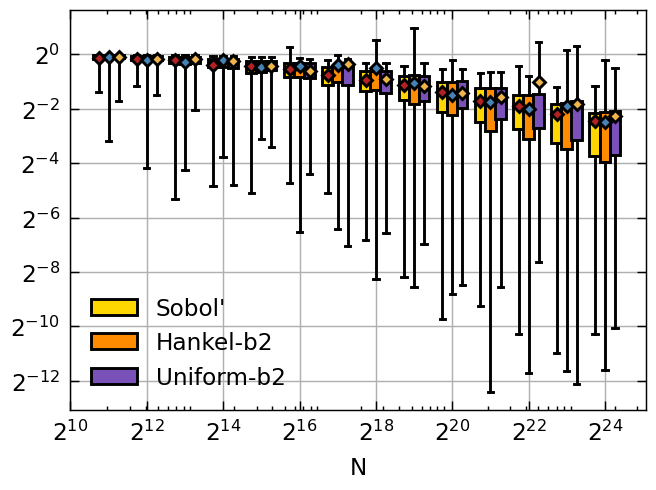}
    \caption{Squared error of the median estimator based on 15 randomized quasi-Monte Carlo (RQMC) samples. Each boxplot is constructed from 448 samples. From left to right, $s = 2, 25, 50$.}
    \label{fig:m24_s2_25_50_xexpx}
\end{figure}

We observe that, for both types of integrand, when $s = 2$, the Sobol' sequence achieves smaller error than HRD and URD. For $s = 25$ and $s = 50$, all three methods are in the pre-asymptotic regime and exhibit comparable errors. This observation further demonstrates the robustness of HRD. 


\end{document}